\documentclass[reqno,11pt]{amsart}

\usepackage[top=2.0cm,bottom=2.0cm,left=3cm,right=3cm]{geometry}
\usepackage{amsthm,amsmath,amssymb,dsfont}
\usepackage{mathrsfs,amsfonts,functan,extarrows,mathtools}
\usepackage[colorlinks]{hyperref}

\usepackage{marginnote}
\usepackage{xcolor}
\usepackage{stmaryrd}
\usepackage{esint}
\usepackage{graphicx}
\usepackage{bbm}
\usepackage{wasysym}
\usepackage{textcomp}
\usepackage{float}
\usepackage{wasysym}
\usepackage{fmtcount}
\usepackage{tikz}
\usetikzlibrary{calc,decorations.pathreplacing}

\newtheorem{theorem}{Theorem}[section]

\newtheorem{definition}{Definition}[section]
\newtheorem{corollary}{Corollary}[section]
\newtheorem{remark}{Remark}[section]
\newtheorem{lemma}{Lemma}[section]
\newtheorem{example}{Example}[section]
\numberwithin{equation}{section}
\allowdisplaybreaks

\def\R{\mathbb{R}}
\def\S{\mathbb{S}}

\def\exp{\mathrm{exp}}

\DeclareMathOperator{\supp}{supp}

\DeclareMathOperator{\loc}{loc}

\newcounter{wronumber}

\begin{document}
\title[Renormalized Solutions]
			{Global renormalized solutions for hard potential non-cutoff Boltzmann equation without defect measure}

\author{Yi-Long Luo}
\address[Yi-Long Luo]
{\newline School of Mathematics, Hunan University, Changsha, 410082, P. R. China}
\email{luoylmath@hnu.edu.cn}

\author{Jing-Xin Nie}
\address[Jing-Xin Nie]{
	School of Mathematics, Hunan University, Changsha, 410082, P. R. China} \email{2359093513@qq.com}

\thanks{\today}

\maketitle

\begin{abstract}
  The existence of global renormalized solutions to the Boltzmann equation with long-range interactions without angular cutoff was first established by Alexandre and Villani [Comm. Pure Appl. Math., 55(1), 30-70, 2002]. Their result relies on a definition of renormalized solutions involving a non-negative defect measure.  In this paper, we address this issue for the inverse power law model in the case of hard potentials ($0 \leq \gamma \leq 1$). By exploiting the stronger coercivity estimates provided by hard potentials, we prove that the defect measure actually vanishes. Consequently, we establish the global existence of renormalized solutions for the non-cutoff Boltzmann equation with hard potentials in the standard sense, without any defect measure. Finally, we construct a counterexample showing that the approach developed for the hard potential case fails for soft potential model ($-3 < \gamma < 0$).  \\


\end{abstract}

\section{Introduction and main results}

\subsection{Motivations}

In 2002, Alexandre and Villani \cite{Alex-Vill-2002-CPAM} established the existence of renormalized solutions with a non-negative defect measure for the Boltzmann equation without angular cutoff in the whole space. {\em The aim of this paper is to prove that, for hard potentials, this defect measure vanishes.} The model considered in present work is given by the following  Boltzmann equation 
\begin{equation}\label{Boltz}
	\frac{ \partial  f }{ \partial  t } + v \cdot \nabla_x f = Q (f , f ) ,
\end{equation}
where the distribution function $f(t, x, v)$ is non-negative. For any time $t \geq 0$, $f(t, \cdot, \cdot)$ represents the particle density in phase space, with position $x \in \R^3_x $ and velocity $v \in \R^3_v $. The term $Q(f,f)$ denotes the Boltzmann collision operator, which acts only on the velocity dependence of the distribution function $f$ (reflecting the physical assumption that collisions are localized in time and space),
\begin{equation}\label{Q}
	Q(f,f) = \int_{\R ^3 _{v_*} } d v_* \int_{\S^2} B ( v- v_* ,\sigma ) ( f ' f'_* - f f_*) d \sigma .
\end{equation}
Here $ f' = f( v'), f'_* = f( v'_* ) , f_* = f( v_*)$ (with $t$ and $x$ being parameters only), and  
\begin{equation}\label{v} 
	\begin{cases}
		v'   &= \dfrac{v + v_*}{2} + \dfrac{|v - v_*|}{2}\sigma \\[10pt]
		v'_* &= \dfrac{v + v_*}{2} - \dfrac{|v - v_*|}{2}\sigma
	\end{cases} ,
\end{equation}
where $ \sigma \in \S^2 $, and $ (v , v_*)  , \ ( v', v_*') \in \R^3 \times \R^3$ denote the pre-collision and post-collision velocities of two monatomic molecules; see Grad's work \cite{Grad-cut-off} and the references therein. Moreover, the above formulae express the conservation laws of momentum and kinetic energy for elastic collisions
\begin{equation}\label{v-law} 
	\begin{cases}
		&v' + v'_* =  v + v_*  \\[3pt]
		&|v'|^2+ |v'_*|^2 =  |v|^2 + |v_*|^2 
	\end{cases} .
\end{equation}

We shall use the notation
\begin{equation}
	k = \frac{v-v_*}{|v-v_*|} , \quad k \cdot \sigma = \cos \theta , \ \ 0 \leq \theta \leq \pi.
\end{equation}
Without loss of generality, we assume that the support of $B(v - v_*, \sigma)$ lies in the set ${0 \leq \theta \leq \pi/2}$, i.e., $(v - v_*) \cdot  \sigma  \geq 0$. Otherwise, this case can be recovered by replacing $B$ with its symmetrized  version
\begin{equation*}
	\bar{B} ( v - v_* , \sigma ) = \left[ B ( v - v_* , \sigma ) + B ( v - v_* , - \sigma )\right] \mathbf{1}_{ (v - v_*) \cdot  \sigma  >0} .
\end{equation*}

In this paper, we are primarily concerned with the case where 
$B$ exhibits an angular singularity such that for a generic relative velocity $z$, the integral satisfies
\begin{align}\label{B-infty}
\int_{\S^{2}} B(z,\sigma ) d \sigma = +\infty.
\end{align}
 A typical example is the inverse $s$-power law in dimension $3$, where the collision kernel is given by
\begin{equation}\label{B-gamma}
	B(z, \sigma ) = |z|^\gamma b(k \cdot \sigma ) , \quad \text{with } \gamma = 1- \frac{4}{s-1}
\end{equation}
and the angular function satisfies the asymptotic behavior
\begin{equation}\label{nu}
	\sin  \theta\,  b(\cos \theta) \sim K\theta^{-1-\nu} \text{ as } \theta \rightarrow 0, \quad \nu = \frac{2}{s-1}, \quad K>0.
\end{equation}
Here the range $0 \leq \gamma \leq 1$ corresponds to hard potentials, while the range  $-3 < \gamma < 0$ refers to soft potentials, and $\gamma = -3$ represents the Coulomb potential. Specifically, within the hard potential regime, $\gamma = 0$ denotes Maxwell molecules and $\gamma = 1$ represents hard spheres, i.e., the case $s= \infty$ with $\sin  \theta\,  b(\cos \theta) \sim K\theta^{-1 } \text{ as } \theta \rightarrow 0$. This paper focuses on the case where $0 \leq \gamma \leq 1$ which necessitates the condition $5 \leq s \leq \infty$. Moreover, one has $\nu \in [0, 1/2]$ for $0 \leq \gamma \leq 1$.

For a given collision kernel $B$, Alexandre-Villani \cite{Alex-Vill-2002-CPAM} introduced the notion of momentum transfer $M$ as
 \begin{equation}\label{M-1}
	\int_{\S^2 } B ( v- v_* , \sigma ) \left(  v - v'  \right)  d \sigma = \frac{1}{2} ( v- v_*) M ( |v - v_*|).
\end{equation}
Noting that  $ v - v ' =  \frac{1}{2} ( v - v_*  )  - \frac{1}{2} |v_* - v | \sigma$, we take the dot product of both sides of \eqref{M-1} with $v-v_* $ to obtain
\begin{align*}
	   \frac{1}{2} | v - v_* |^2  \int_{ \S ^2 } B ( v-v_* , \sigma ) (1- k\cdot \sigma ) d \sigma  = \frac{1}{2}| v - v_* |^2 M ( |v - v_*|).
\end{align*}
Consequently, the momentum transfer $M$ admits the equivalent representation
 \begin{equation}\label{M-2}
 	M ( |z|)= \int_{\S ^2 } B ( z, \sigma ) ( 1-k \cdot \sigma ) d \sigma .
 \end{equation}
Similarly, we define
 \begin{equation}\label{M-pri}
	M'(|z|) = \int_{\S ^2} B'(z, \sigma)(1 - k \cdot \sigma) \, d\sigma,
\end{equation}
where
\begin{equation}\label{B-pri}
	B'(z, \sigma) = \sup_{1 < \lambda \leq \sqrt{2}} \frac{|B(\lambda z, \sigma) - B(z, \sigma)|}{(\lambda - 1)|z|}.
\end{equation}
The magnitude of the non-negative function $M'$ provides a very mild control on the regularity of $B $ with respect to the relative velocity variable. Furthermore, it was shown in \cite{Alex-Vill-2002-CPAM} that appropriate assumptions on $M$ and $M'$ give a precise meaning to the renormalized formulation of the inhomogeneous Boltzmann equation.

\subsection{Properties of momentum transfer} 

 To study the general case of long-range interactions, Alexandre and Villani  \cite{Alex-Vill-2002-CPAM} introduced several assumptions on the collision kernel. For the inverse $s$-power law potentials with hard and soft interactions, the collision kernel naturally satisfies these assumptions, which correspond to Properties  \uppercase\expandafter{\romannumeral 1} and  \uppercase\expandafter{\romannumeral 2} detailed below.

 {\bf Property \uppercase\expandafter{\romannumeral 1}: Local Integrability.}\label{Prop-1}  
 We verify that  the functions $M(|z|)$ and $|z| M'(|z|)$  corresponding to the model case \eqref{B-gamma} are locally integrable. Applying the spherical coordinate transformation yields  
\begin{equation}\label{Mz}
	\begin{aligned}
		M ( |z| )=& \int_{\S^2} B ( z , \sigma ) ( 1-k \cdot \sigma ) d \sigma \\
		= & 2\pi  \int_{0}^{\frac{\pi }{2} } B ( |z| , \cos \theta )( 1- \cos \theta ) \sin   \theta d \theta .
	\end{aligned}
\end{equation}
Recalling the expression \eqref{B-gamma} for the collision kernel $B$ and \eqref{nu}, we conclude that
\begin{equation}\label{M2}
  \begin{aligned}
    M (|z|)  =& \int_{ \S^2 } B ( z , \sigma ) ( 1-k \cdot \sigma ) d \sigma  \\
	= & 2\pi |z|^\gamma \int_{0}^{\frac{\pi }{2} } b ( \cos \theta )( 1- \cos \theta )  \sin   \theta  d \theta \\
	\leq & C K \int_{0}^{\frac{\pi }{2} } \frac{1}{\theta ^{\nu -1 }} d \theta |z|^\gamma = C_{\nu, K} |z|^\gamma .
  \end{aligned}
\end{equation}
 Similarly, we have
 \begin{equation} \label{zM2}
 \begin{aligned}
    |z|M'(|z|) &= |z| \int_{ \S^2 } \left[ \sup_{1 < \lambda \leq \sqrt{2} } \frac{B (\lambda z ,\sigma ) - B (z, \sigma )}{(\lambda -1 )|z|} \right] ( 1- k \cdot \sigma ) d \sigma \\ 
	&= 2\pi |z|  \int_{0}^{\frac{\pi}{2} } \left[ \sup_{1 < \lambda \leq \sqrt{2} } \frac{|\lambda z|^\gamma - |z|^\gamma}{(\lambda -1) |z| } \right] b ( \cos \theta ) \sin ^2 \left(\frac{\theta}{2} \right) \sin   \theta d \theta \\
	& \leq C'  K \int_{0}^{\frac{\pi}{2} }  \frac{1}{\theta ^{\nu -1 }}  d \theta |z|^\gamma = C'_{\nu, K} |z|^\gamma.
\end{aligned} 
\end{equation}
Due to the local integrability of $|z|^\gamma $ with $\gamma \in (-3,1]$,  we conclude that $M (|z|), |z| M' (|z|) \in L^1_{\loc} (\R^3_z)$.

{\bf Property \uppercase\expandafter{\romannumeral 2}:  $M  ( |z| ) $ and $ |z|M'(|z|) $  are both $ O (|z|^\gamma )$ as $|z|\rightarrow \infty $.}\label{Prop-2}
 Following the same reasoning as in \eqref{M2} and \eqref{zM2}, we derive that
\begin{equation}\label{M-gamma}
\begin{aligned}
   M (|z |)  &\sim   K  |z|^{\gamma   }   \int_{0}^{\frac{\pi}{2} }  \frac{1}{\theta ^{\nu -1 }}  d\theta, \\
    |z|  M'  (|z |)   &\sim    K |z|^{\gamma   }   \int_{0}^{\frac{\pi}{2} }  \frac{1}{\theta ^{\nu -1 }}  d\theta,
\end{aligned}
\end{equation}  
where $\gamma = 1 - 4/( s-1 )$ and $ \nu = \frac{2}{s-1} $. Consequently, both $M  ( |z| ) $ and $ |z|M'(|z|) $ are of order $O (|z|^\gamma )$ as $|z|\rightarrow \infty $.

\subsection{Notations and main results}

\subsubsection{Notations}

In this paper, the symbol $\R^3$ denotes the three-dimensional Euclidean space. To distinguish between variables in specific contexts, we attach subscript indices to $\R^3$ (e.g., $\R^3_x$, $\R^3_v$, etc.). Let $B_R$ denotes the closed balls of radius $R$ centered at the origin in $\R^3$. To indicate the variable under consideration, we employ the notation $B_R(x) \subseteq \R^3_x$, $B_R(v) \subseteq \R^3_v$, etc. The symbol $\R_+$ denotes the positive half-line defined by $\R_+ = \{ t \in \R^1; t > 0 \}$. Moreover, $\mathbb{S}^2 = \{ \sigma \in \R^3; |\sigma| =1 \}$ represents the unit sphere in $\R^3$. The set of non-negative integers is denoted by $\mathbb{N}$, and we define $\mathbb{N}^k = \{ \mathbf{n} = (n_1, \cdots, n_k) | n_i \in \mathbb{N}, 1 \leq i \leq k \}$ for any integer $k \geq 1$. For any $\mathbf{n} \in \mathbb{N}^k$, its $1$-norm is defined by $| \mathbf{n} | = \sum_{i=1}^{k} n_i$. The abbreviation ``a.e." stands for ``almost everywhere". The symbol $L^p (d \mu)$ represents the standard $L^p$ space associated with the measure $\mu$. For any functional space $F$, we write $F_+ = \{ f \in F; f \geq 0 \textrm{ a.e.} \}$.   The notation $f^n \rightrightarrows f$ indicates the uniform convergence of $f^n$ to $f$, while $f^n \rightharpoonup f$ denotes the weak convergence of  $f^n$ to $f$ in the weak $L^1$ topology. The Fourier transform operator is denoted by $\mathcal{F}$; the Fourier transform of $f$ is occasionally written as $\hat {f}$, so that $\mathcal{F} f = \hat{f}$. We write $a \wedge b = \min \{ a,b\}$.  Furthermore, we introduce the weighted norm $\|f\|_{L^1_\alpha( \R^3_v )} = \int_{ \R ^3_v } f ( v )( 1 + |v|^\alpha ) d v   $.  We also introduce the space $L \log L$ as
\begin{equation*}
  \begin{aligned}
    L \log L = \{ f : \R^3_v \to \R_+ ; \  f |\log f| \in L^1 (\R^3_v) \} \,.
  \end{aligned}
\end{equation*}

For any smooth function $f (t,x,v )$ on $ \R_+ \times \R^3_x \times \R^3_v $ and any multi-index $\boldsymbol{\zeta} = ( \zeta_0, \zeta_1, \cdots, \zeta_6 ) \in \mathbb{N}^7$, we define the partial derivative operator $D^{ \boldsymbol{ \zeta } }$ by
\begin{equation*}
	\begin{aligned}
		D^{ \boldsymbol{ \zeta } } f = \frac{ \partial^{ | \boldsymbol{\zeta} | } f }{  \partial t^{\zeta_0} \partial x_1^{\zeta_1} \partial x_2^{\zeta_2} \partial x_3^{\zeta_3} \partial v_1^{\zeta_4} \partial v_2^{\zeta_5} \partial v_3^{\zeta_6} } \,.
	\end{aligned}
\end{equation*}
Let $u : \mathbb{R}^3_x \rightarrow \mathbb{R}$ be a continuous bounded function. For any exponent $\alpha \in (0, 1)$, the supremum norm and the $\alpha$-Hölder semi-norm are defined, respectively, by
\begin{align*}
	\|u\|_{C(\mathbb{R}^3_x)}  := \sup_{x\in \mathbb{R}^3_x } |u(x)|, \quad 
	[u]_{C^{\alpha}(\mathbb{R}^3_x)}  := \sup_{\substack{x,y\in \mathbb{R}^3 \\ x\neq y}} \frac{|u(x)-u(y)|}{|x-y|^\alpha}.
\end{align*}
The Hölder space $C^{\alpha}(\mathbb{R}^3_x)$ consists of all functions with finite $\alpha$-Hölder norm
\begin{equation*}
	C^{\alpha}(\mathbb{R}^3_x) := \left\{ u : \mathbb{R}^3_x \to \mathbb{R} \ \middle|\ \|u\|_{C^{\alpha}(\mathbb{R}^3_x)} := \| u\|_{C(\mathbb{R}^3_x)} + [ u]_{C^{ \alpha}(\mathbb{R}^3_x)} < \infty \right\} .
\end{equation*} 
We denote by $C^\alpha_c (\mathbb{R}^3_x)$ the closure of $C^\infty_c (\mathbb{R}^3_x)$ under the $C^\alpha$ norm. Following Remark 2.2.2 in \cite{Grafakos} and Proposition 2.76 in \cite{Bahouri}, the dual space of $C^\alpha_c (\mathbb{R}^3_x)$ is identified with the Besov space $B^{-\alpha}_{1,1} (\mathbb{R}^3_x)$. 
To define this space properly, we introduce the Littlewood-Paley decomposition. Let $\mathcal{C}$ be the annulus $\{\xi \in \mathbb{R}^3_\xi \mid 3/4 \le |\xi| \le 8/3\}$. There exist radial smooth functions $\chi \in \mathcal{D}(B(0,4/3))$ and $\varphi \in \mathcal{D}(\mathcal{C})$ taking values in $[0, 1]$, such that
\begin{align*}
	\forall \xi \in \mathbb{R}^3_\xi, \quad \chi(\xi) + \sum_{j \ge 0} \varphi(2^{-j}\xi)  = 1, & \quad 
	\forall \xi \in \mathbb{R}^3_\xi \setminus \{0\}, \quad \sum_{j \in \mathbb{Z}} \varphi(2^{-j}\xi)  = 1, \\
	|j - j'| \ge 2 \implies   \operatorname{supp} & \varphi(2^{-j}\cdot) \cap \operatorname{supp} \varphi(2^{-j'}\cdot)  = \emptyset, \\
	j \ge 1 \implies \operatorname{supp} & \chi \cap \operatorname{supp} \varphi(2^{-j}\cdot)  = \emptyset.
\end{align*}
Let $h = \mathcal{F}^{-1}\varphi$ and $\tilde{h} = \mathcal{F}^{-1}\chi$. The nonhomogeneous dyadic blocks $\Delta_j$ are defined as follows
\begin{gather*}
	\Delta_j u = 0 \quad \text{if} \quad j \le -2, \\
	\Delta_{-1} u = \chi(D)u = \int_{\mathbb{R}^3_y} \tilde{h}(y)u(x - y) \, dy, \\
	\Delta_j u = \varphi(2^{-j}D)u = 2^{3j} \int_{\mathbb{R}^3_y} h(2^j y)u(x - y) \, dy \quad \text{if} \quad j \ge 0.
\end{gather*}
For $\alpha \in (0,1)$, the nonhomogeneous Besov space $B^{-\alpha}_{1,1}(\mathbb{R}^3_x)$ consists of all tempered distributions $u \in \mathcal{S}'(\mathbb{R}^3_x)$ such that
$$
\|u\|_{B^{-\alpha}_{1,1}(\mathbb{R}^3_x)} := \sum_{j \ge -1} 2^{-j\alpha} \|\Delta_j u\|_{L^1(\mathbb{R}^3_x)} < \infty.
$$

Furthermore, a function $f$ belongs to the space $C ( \mathbb{R}_+ ; w\text{-}L^1(\mathbb{R}^3_x; B^{-\alpha}_{1,1} (\mathbb{R}^3_v)) )$ provided that for every test function $\varphi \in L^\infty (\mathbb{R}^3_x ; C^\alpha_c (\mathbb{R}^3_v))$, the mapping
$$
t \mapsto \iint_{ \mathbb{R}^3_x \times \mathbb{R}^3_v } f(t, x,v) \, \varphi (x,v)\, dx \, dv
$$
is continuous on $\mathbb{R}_+$.
\subsubsection{Initial data}
In this paper, we focus on the Cauchy problem of the equation  \eqref{Boltz}, whose initial data is 
\begin{equation}\label{f0}
	f(t,x,v) |_{t=0} = f_0 ( x, v ) \geq 0, \quad \text{a.e. on } \  \mathbb{R}^3_x \times \mathbb{R}^3_v
\end{equation}  
with the following assumptions:
\begin{equation}\label{f0-p}
\iint_{\R^3 _x \times \R^3_v } f_{ 0} (1+|x|^2 +| v |^2 +|\log f_0 | ) d x d v  <\infty \,.
\end{equation}

\subsubsection{Definition of renormalized solutions}
 \begin{definition}[Renormalized solution]\label{def-re}
	Let $\alpha \in (\nu , 1)$ with $\nu $ defined in \eqref{nu}.  A function
	\begin{align}
	 f (t,x,v) \in C \big( \R_+ ; w\text{-} L^1(\mathbb{R}^3_x; B^{-\alpha}_{1,1} (\mathbb{R}^3_v))  \big) \cap L^\infty ( \R_+; L^1 ( 1+|x|^2 +|v|^2 ) d x d v )  
	\end{align}
	is a renormalized solution of \eqref{Boltz} if 
\begin{itemize}
  \item $f (t, x, v) \geq 0$ $a.e.$ on $ \R_+ \times \R^3_x \times \R^3_v$;
  \item for every $\beta \in C^2 (\R_+ , \R _+)$ such that
	\begin{align}\label{beta}
	\beta (0)= 0 , \quad 0 < \beta ^\prime (t ) \leq C/ (1+ t ), \quad \beta^{ \prime \prime } (t) \leq 0
	\end{align}
	 for some $C > 0$, the function $\beta (f) $ solves 
	\begin{equation}\label{beta-Boltz}
		\frac{ \partial  \beta(f) }{ \partial  t } + v\cdot  \nabla_x  \beta(f) =\beta ' (f) Q (f , f ) 
	\end{equation}
	in the sense of distributions.
\end{itemize}
\end{definition}

\begin{remark}
  Alexandre and Villani \cite{Alex-Vill-2002-CPAM} introduced a different definition of renormalized solutions to the Boltzmann equation \eqref{Boltz} with defect measure, in which equation \eqref{beta-Boltz} is replaced by
  \begin{equation*}
    \begin{aligned}
      \frac{ \partial  \beta(f) }{ \partial  t } + v\cdot  \nabla_x  \beta(f) \geq \beta ' (f) Q (f , f ) \,,
    \end{aligned}
  \end{equation*}
  which holds in the sense of distributions.
\end{remark}

Next, we define the Boltzmann-type H-functional by
\begin{equation}\label{Hf} 
		H(f)(t)  
	 =   \iint_{\R^3 _x \times \R^3_v  } f  \log f  d x d v   . 
\end{equation}

We are now in a position to state the main result of this paper.
\begin{theorem}\label{MainThm}
Let $0 \leq \gamma \leq 1$ and $\nu < \alpha < 1$ with $\nu \in [0, 1/2]$ defined in \eqref{nu}. Suppose that the initial data  $f_0$ satisfies \eqref{f0} and \eqref{f0-p}. Then the Cauchy problem \eqref{Boltz}-\eqref{f0} admits a renormalized solution  $$f  \in C \big( \R_+ ; w\text{-} L^1( \mathbb{R}^3 _x; B^{-\alpha}_{1,1} ( \mathbb{R} ^3 _v)) \big) \cap L^\infty ( \R_+; L^1 ( 1 + | x |^2 + | v |^2 ) d x d v ) $$ 
in the sense of Definition \ref{def-re}, satisfying the following statements:
\begin{enumerate}
\item { For any $T>0$ and any $t \in [0,T]$, there exists a constant $C_T > 0$ such that 
\begin{equation}\label{bdd-f}
	\iint_{ \R ^3_x \times \R ^3_v  } f(t, \cdot ) ( 1 + |x|^2 + |v|^2 + |\log f |) d x d v \leq C_T \,.
	\end{equation}}
\item The following entropy inequality holds:
\begin{equation}\label{entropy-th}
H( f ) ( t ) + \frac{1}{4} \int_{0}^{t} ds \iiiint_{ \R ^3_x \times \R ^3_v  \times \R ^3_{v_* }\times \S^2 }  B    
( { f }' {f'_* }  - f  f _* ) \log \left( \frac{ {f }'  {f'_* }  }{f  f _*} \right) d x d v d v_* d \sigma \leq H ( f_0 ) \,.
	\end{equation}
\end{enumerate}
\end{theorem}

\begin{remark}[Soft potential model $-3 < \gamma < 0$]
  The smoothing estimate in Lemma \ref{L-FZn} and the average time-velocity estimate in Lemma \ref{Lmm-ATV} are necessitated in proving the existence of renormalized solutions. To establish these two estimates, one needs to control the collision integral $\iint_{\R_v^3 \times \R^3_{v_*}} f^n (v) f^n (v_*) |v - v_*|^\gamma d v d v_*$. However, in Example \ref{soft-f}, we construct a counterexample showing that there exist a positive function $f (v)$ with $\int_{\R^3_v} f (1 + |v|^2 + |\log f|) d v < \infty$ such that $\iint_{\R_v^3 \times \R^3_{v_*}} f (v) f (v_*) |v - v_*|^\gamma d v d v_* = + \infty$. Consequently, the approach developed for the hard potential case fails to solve the soft potential model in the sense of Definition \ref{def-re}.
  
  The collision integral $\iint_{\R_v^3 \times \R^3_{v_*}} f^n (v) f^n (v_*) |v - v_*|^\gamma d v d v_*$ can be bounded in terms of $\| f^n \|^2_{L^p (\R^3_v)}$ with $p = \frac{6}{6 + \gamma}$ for $- 3 < \gamma < 0$. Noting that $H^{\nu /2}_{loc} (\R^3_v) \hookrightarrow L^p_{loc} (\R^3_v)$, we thus obtain a bound on this collision integral provided that the $H^{\nu /2}_{loc} (\R^3_v)$-norm of $f^n$ is uniformly bounded. For the hard potential case $0 \leq \gamma \leq 1$, such a bound was established in \cite{ADVW-bound-2000}. However, the corresponding estimate for soft potential fails, as the $H^{\nu /2}_{loc} (\R^3_v)$-norm of $f^n$ would need to be controlled by the collision integral itself, resulting in a circular dependence.
\end{remark}

\subsection{Comparison with the work \cite{Alex-Vill-2002-CPAM} of Alexandre and Villani}

In this subsection, we discuss the differences between the present work and that of Alexandre and Villani's work \cite{Alex-Vill-2002-CPAM}.

First, the model \eqref{Boltz} considered in this paper differs from that studied by Alexandre and Villani in \cite{Alex-Vill-2002-CPAM}. Indeed, Alexandre and Villani imposed decay assumptions on the momentum transfer $M (|z|)$ and the mild regular momentum transfer $M' (|v|)$ at infinity,
\begin{equation}\label{AV-Ass}
  \begin{aligned}
    M (|z|) = o (1) \,, \quad |z| M' (|z|) = o (|z|^2) \quad \textrm{as } \ |z| \to \infty \,,
  \end{aligned}
\end{equation}
see (2.9) on Page 42 of \cite{Alex-Vill-2002-CPAM}\footnote{In fact, they introduced a weighted momentum transfer $M^\alpha (|z|)$ for $0 \leq \alpha \leq 2$ defined by
\begin{equation*}
  \begin{aligned}
    M^\alpha (|z|) = \int_{\mathbb{S}^2} B (z, \sigma) ( 1 - k \cdot \sigma )^\frac{\alpha}{2} d \sigma \,, \quad k = \frac{z}{|z|} \,,
  \end{aligned}
\end{equation*}
which satisfies $M^2 (|z|) = M (|z|)$, see \eqref{Mz}. They then assumed that for $\alpha \in [0, 2]$, as $|z| \to \infty$,
\begin{equation*}
  \begin{aligned}
    M^\alpha (|z|) = o (|z|^{2 - \alpha}) \,.
  \end{aligned}
\end{equation*}}. As shown in \eqref{M-gamma}, the momentum transfer $M (|z|)$ and the mild regular momentum transfer $M' (|v|)$ in the present work satisfy
\begin{equation*}
  \begin{aligned}
    M (|z|) = O (|z|^\gamma) \,, \quad |z| M' (|z|) = O (|z|^\gamma) \quad \textrm{as } \  |z| \to + \infty
  \end{aligned}
\end{equation*}
for $0 \leq \gamma \leq 1$. Namely, the model considered here does not satisfy the assumptions \eqref{AV-Ass} of Alexandre and Villani, which require faster decay at infinity. 

In \cite{Alex-Vill-2002-CPAM}, assumptions \eqref{AV-Ass} at infinity are utilized to establish the integrability of the $(\mathcal{R}_2)$-part of the renormalized collision operator $\beta' (f) Q (f, f)$; see \eqref{R123} below. In this procedure, for given $v_*$, the linear operator \eqref{T}, namely
\begin{equation*}
  \begin{aligned}
    \mathcal{T} : \varphi \mapsto \int_{\mathbb{S}^2} B (v - v_*, \sigma) (\varphi' - \varphi) d \sigma
  \end{aligned}
\end{equation*}
plays a key role in proving the integrability of $(\mathcal{R}_2)$. Here $\mathcal{T}$ is the adjoint operator of $f \mapsto Q ( \delta_{v_*}, f )$. Alexandre and Villani proved that $\mathcal{T}$ is bounded from $W^{2, \infty}$ to $L^\infty$. Namely, for all $\varphi \in W^{2, \infty} (\R_v^3)$,
\begin{equation*}
  \begin{aligned}
    | \mathcal{T} \varphi (v)| \leq & \tfrac{1}{2} \| \varphi \|_{W^{2, \infty}} |v - v_*| ( 1 + \tfrac{|v - v_*|}{2} ) M (|v - v_*|) \,, \\
    | \mathcal{T} \varphi (v)| \leq & 2 \| \varphi \|_{W^{2, \infty}} ( 1 + |v - v_*| )^\alpha M^\alpha (|v - v_*|) \quad \forall \ \alpha \in [0,2] \,,
  \end{aligned}
\end{equation*}
see Proposition 3.3 on Page 52 of \cite{Alex-Vill-2002-CPAM}. Owing to the polynomial factors $|v - v_*| ( 1 + \tfrac{|v - v_*|}{2} )$ and $( 1 + |v - v_*| )^\alpha$ in the above bounds, the assumptions \eqref{AV-Ass} at infinity are requisite for establishing the integrability of $(\mathcal{R}_2)$ by means of the a priori bound \eqref{bdd-f}.

In our work, for the hard potential model ($0 \leq \gamma \leq 1$), $M (|z|)$ and $|z| M' (|z|)$ do not exhibit decay properties at infinity (see \eqref{M-gamma}). To overcome this difficulty, we establish the boundedness of the operator $\mathcal{T}$ form $C_c^\alpha (\R_v^3)$ ($\alpha \in (\nu, 1)$) to $L^\infty_{loc} (\R_v^3)$, i.e., for all $\varphi \in C_c^\alpha (\R_v^3)$,
\begin{equation*}
  \begin{aligned}
    | \mathcal{T} \varphi (v) | \leq C_{\nu, \alpha} \| \varphi \|_{C_c^\alpha} |v - v_*|^{\alpha + \gamma} \leq C' _{\nu, \alpha} \| \varphi \|_{C_c^\alpha} |v - v_*|^{\alpha} M ( |v - v_*| ) \,,
  \end{aligned}
\end{equation*}
see Lemma \ref{Lmm-T} below. The integrability of $(\mathcal{R}_2)$ is then established by means of the a priori bound \eqref{bdd-f}.

Second, we establish a so-called {\em average time-velocity estimates for $(\mathcal{R}_3)^n$} 
\begin{equation*}
  \begin{aligned}
    \int_{0}^{T} dt \int_{   \R ^3_v } ( \mathcal{R}_3)^n d v 
    \leq & C_{\nu ,R} \sup_{ t \in [0, T]} \int_{  {\R}^3_v } f^n dv  \int_{   \R ^3_{v_*}  } f^n ( 1+ |v_*|^2) dv_* + \frac{1}{2} \int_{0}^{T} dt \int_{\mathbb{R}^{ 3} _v } e^n dv   \text{ a.e.} \  x \in \R_x^3
  \end{aligned}
\end{equation*} 
for $0 \leq \gamma \leq 1$ (see \eqref{R3n-tv} in Lemma \ref{Lmm-ATV} below), where $e^n$ denotes the entropy dissipative density defined in \eqref{d-en} below. We also establish a so-called {\em uniform smallness estimate for small angular deflections}
\begin{equation*}
  \begin{aligned}
    \iint_{B_R(v)\times \S ^2} B_k ( v-v_*, \sigma )  \left( g(v') - g(v) \right)^2   \, d\sigma dv \le  \frac{C}{k} (1+|v_*|)^{1+\gamma+\nu} \|g\|^2_{{H^{\nu/2}}{(B_R (v) )}}
  \end{aligned}
\end{equation*}
for $0 \leq \gamma \leq 1$ (see \eqref{Hnu} in Lemma \ref{LBep} below). These two estimates constitute the most essential observation for removing the defect measure appearing in the work of Alexandre-Villani \cite{Alex-Vill-2002-CPAM}. The defect measure arises from Fatou lemma argument
\begin{equation*}
  \begin{aligned}
    \int_0^T d t \iint_{\R_x^3 \times \R_v^3} ( \mathcal{R}_3 ) d x d v \leq \liminf_{n \to \infty} \int_0^T \iint_{\R_x^3 \times \R_v^3} \left( 1 + \frac{1}{n} \int_{\R_v^3} f^n d v \right)^{-1} (\mathcal{R}_3)^n d x d v
  \end{aligned}
\end{equation*}
under the strong convergence $f^n \to f$ in $L^1 ( [0,T] \times \R_x^3 \times \R_v^3 )$, corresponding to the arguments on Page 61 of \cite{Alex-Vill-2002-CPAM}.

In our approach, we observe that
\begin{equation*}
  \begin{aligned}
    f^n_* \frac{  ({f ^n}'-f^n )^2   }{( 1+   f^n ) ^2 ( 1+  {f ^n}')} \leq   \frac{1}{2}  \left(\sqrt{ {f ^n}' }  - \sqrt{f^n} \right) ^2 \left( \sqrt{ {f_*^n}' } + \sqrt{ f^n_*}  \right) ^2 \,,
  \end{aligned}
\end{equation*}
which implies
\begin{equation*}
  \begin{aligned}
    \int_{0}^{T} dt \int_{   \R ^3 _v } ( \mathcal{R}_3)^n d v \leq  &\frac{1}{2} \int_{0}^{T} dt \iiint_{ \mathbb{R}^{ 3} _v \times \mathbb{R}^{ 3} _{v_* } \times  \S^{2} } B_n    \left(\sqrt{ {f ^n}' }  - \sqrt{f^n} \right) ^2 \left( \sqrt{ {f_*^n}' } + \sqrt{ f^n_*}  \right) ^2   d v dv_* d \sigma \,.
  \end{aligned}
\end{equation*}
Moreover, combining the cancellation lemma (see Lemma \ref{prop-S} below) with the lower bound \eqref{sq-e1} on the entropy dissipative density, i.e.,
\begin{equation*}
  \begin{aligned}
    & \iiint_{ \mathbb{R}^{ 3}_v  \times \mathbb{R}^{ 3} _{v_* } \times  \S^{2} } B_n   \left(\sqrt{ {f ^n}' }  - \sqrt{f^n} \right) ^2 \left( \sqrt{ {f_*^n}' } + \sqrt{ f^n_*}  \right) ^2  d v dv_* d \sigma \\
		& +  \iiint_{ \mathbb{R}^{ 3}_v  \times \mathbb{R}^{ 3}_{v_* } \times  \S^{2} } B_n   ( {f ^n}' -f^n   ) ( {f_*^n}'  - f^n_*  )    d v dv_* d \sigma \leq \int_{\mathbb{R}^{ 3} _v } e^n dv,
  \end{aligned}
\end{equation*}
we successfully establish the average time-velocity estimate. By using the decomposition $v' - v = h = h_1 + h_2$, where $h_1 = - u \sin^2 ( \frac{\theta}{2} )$, $h_2 = |u| \sin (\frac{\theta}{2} ) \cos ( \frac{\theta}{2} ) \omega$, $u = v - v_*$, and $\omega \perp u$, we obtain the change of variables formula
$$d \sigma = \frac{4}{|u|^2} \frac{1}{ \sqrt{1 - \frac{4}{|u|^2} |h_2|^2} } d h_2 \,.$$ 
Together with the asymptotic behavior $\theta \thicksim \frac{2}{|u|} |h_2|$ and $b (\cos \theta) \thicksim K \theta^{-2 - \nu}$ as $\theta \to 0^+$, we deduce that
\begin{equation*}
  \begin{aligned}
    & \iint_{ B_R (v) \times \mathbb{S}^2 } B_k (v - v_*, \sigma) ( g (v') - g (v) )^2 d \sigma d v \\
    \lesssim & \frac{1}{k} (1 + |v_*|^{1 + \gamma + \nu}) \int_{B_R (v)} d v \int_{|h| \leq \sqrt{2} r_k} \frac{ ( g (v + h) - g (v) )^2 }{ | h |^{3 + \nu} } d h \,.
  \end{aligned}
\end{equation*}
Combined with  the equivalent relation
\begin{equation*}
  \begin{aligned}
    \| f \|^2_{H^s (\R^3_v)} \approx \| f \|^2_{L^2 (\R^3_v)} + \iint_{ \R^3_v \times \R^3_{v'} } \frac{ ( f (v') - f (v) )^2 }{| v' - v |^{3 + 2 s}} \mathbf{1}_{|v - v'| \leq 1} d v d v'
  \end{aligned}
\end{equation*}
given in Section 7.3 of \cite{GS-classical-2009}, we establish the uniform smallness estimate for small angular deflections. Upon obtaining these estimates, we derive the equi-integrability  of the sequence $\{ \int_0^T d t \int_{R_v^3} ( 1 + \frac{1}{n} \int_{\R_v^3} f^n d v )^{-1} ( \mathcal{R}_3 )^n \varphi d v \}_{n \geq 1}$ via the renormalized form \eqref{eq6.75} of the approximated equation, where the test function $\varphi (t,x,v) \in L^\infty ( [0,T] \times \R_x^3; C_c^\alpha (\R_v^3) )$ has compact support in $[0, T] \times B_R (x) \times B_R (v)$; see Lemma \ref{equi-R3} below. 

In order to remove the defect measure, we must derive the convergence \eqref{cnv-R3} in Lemma \ref{l-cnvR3}, i.e.,
\begin{equation*}
  \begin{aligned}
    \lim_{n \to \infty}  \int_{0}^{T} dt \iint_{ \mathbb{R}^3_x \times \mathbb{R}^3_v } \left( 1+ \frac{1}{n} \int_{\R^3_v} f^n d v \right)^{-1}  ( \mathcal{R}_3 )^n \varphi d x d v =  \int_{0}^{T} dt \iint_{ \mathbb{R}^3_x \times \mathbb{R}^3_v } ( \mathcal{R}_3 )  \varphi d x d v \,.
  \end{aligned}
\end{equation*}
The idea is to decompose $(\mathcal{R}_3)^n$ and $( \mathcal{R}_3 )$ into the angular singularity part $(\mathcal{R}_3)^{n}_k$, $( \mathcal{R}_3 )_k$ and far-field angular singularity part $(\mathcal{R}_3)^{n,k}$, $( \mathcal{R}_3 )^k$, respectively. We introduce the quantities
\begin{equation*}
  \begin{aligned}
    M^n (x) = \sup_{t \in [0,T]} \int_{\R^3_v} f^n (1 + |v|^2) d v + \int_0^T d t \int_{\R^3_v} e^n d v \,, 
  \end{aligned}
\end{equation*}
and
\begin{equation*}
  \begin{aligned}
    M (x) = \sup_{t \in [0,T]} \int_{\R^3_v} f (1 + |v|^2) d v + \int_0^T d t \int_{\R^3_v} e d v \,,
  \end{aligned}
\end{equation*}
and define the  sets
\begin{equation*}
  \begin{aligned}
    N_L^n = \{ x \in B_R (x) : M^n (x) > L \} \,, \quad N_L = \{ x \in B_R (x) : M (x) > L \} \,.
  \end{aligned}
\end{equation*}
By the uniform bound \eqref{FWC2} in Lemma \ref{Lemma 4.17}, the measures of these sets satisfy $\textrm{meas} (N_L^n) + \textrm{meas} (N_L) \leq \frac{C}{L}$ uniformly in $n \geq 1$. On the other hand, the average time-velocity estimate \eqref{R3n-tv} shows that the sequence $\{ \int_0^T d t \int_{R_v^3} ( 1 + \frac{1}{n} \int_{\R_v^3} f^n d v )^{-1} ( \mathcal{R}_3 )^{n,k} \varphi \mathbf{1}_{(N_L^n)^c} d v \}_{n \geq 1}$ is uniformly bounded in $n \geq 1$ and $x \in (N_L^n)^c = \R_x^3 \setminus N_L^n$. By uniform boundednes and equi-integrability of this sequence, we obtain the convergence \eqref{R3k} for the far-field angular singularity part $(\mathcal{R}_3)^{n,k}$  since
\begin{equation*}
  \begin{aligned}
    (\mathcal{R}_3) ^{n, k}  = -  \iint_{\mathbb{R}^3_{v_*} \times \S^2} B_n^k ( v-v_*, \sigma )f_*^n \Gamma(f^n, {f ^n}') \, dv_* d\sigma  
  \leq    C(f^n + {f^n}') \int_{\mathbb{R}^3_{v_*}} f_*^n \, dv_*
  \end{aligned}
\end{equation*}
and the Lebesgue's dominated convergence theorem; see the proof of procedure of Lemma \ref{L-R3nc} below. Then, using the average time-velocity estimate and uniform smallness estimate for small angular deflections, we prove that the angular singularity parts $(\mathcal{R}_3)_{n,k}$ and $(\mathcal{R}_3)_k$ vanish in the limit $k \to \infty$, namely,
\begin{equation*}
  \begin{aligned}
    \lim_{k \to \infty } \sup_{n} \int_0^T dt 	\iint_{\R^3_x \times \mathbb{R}^3_v}  \left( \left(1 + \frac{1}{n} \int_{ {\R}^3_v} f^n \, dv \right)^{-1}( \mathcal{R}_3 )_k^n   + ( \mathcal{R}_3 )_k \right)\varphi dxdv =0;
  \end{aligned}
\end{equation*}
see \eqref{R3-0} below. Then we establish the convergence 
\begin{equation*}
  \begin{aligned}
    \lim\limits_{n\to \infty} \int_0^T dt \iint_{\mathbb{R}^3_x \times \mathbb{R}^3_v }  \left(1 + \frac{1}{n} \int_{\mathbb{R}^3_v} f^n \, dv \right)^{-1}  (\mathcal{R}_3) ^{n } \varphi   dx dv  
    = \int_0^T dt \iint_{\mathbb{R}^3_x \times \mathbb{R}^3_v}  (\mathcal{R}_3)  \varphi  \, dx dv;
  \end{aligned}
\end{equation*}
see the proof of procedure of Lemma \ref{l-cnvR3} below. Therefore, the convergence \eqref{cnv-R3} is established and the defect measure is removed.

Although the hard potential model ($0 \leq \gamma \leq 1$) considered in this paper differs from that of Alexandre and Villani \cite{Alex-Vill-2002-CPAM}, our approach can also eliminate the defect measure appearing in \cite{Alex-Vill-2002-CPAM} when $0 \leq \gamma \leq 1$.

Third, the existence of renormalized solutions to the soft potential model $(- 3 < \gamma < 0)$ fails because the collision integral $\iint_{ \R_v^3 \times \R_{v_*}^3 } f^n f_*^n |v - v_*|^\gamma d v d v_* $ is uncontrollable. When deriving the smoothing estimate for $\mathcal{F} ( \sqrt{f^n \chi_R} )$ in Lemma \ref{L-FZn}, i.e.,
\begin{equation*}
  \begin{aligned}
    &  	\int_{|\xi |\geq 1 } | \mathcal{F} (\sqrt{ f^n } \chi_R) |^2    Z_n\left( \frac{1}{|\xi | } \right)  d \xi   \\
		\leq & \frac{  C_0 C_2 J^2   (2\pi )^3   }{  \|f ^n \|_{ L^1( B_R(v ) }  }   \left\{   \int_{  \R^{3}_v } e^n  ( t,x,v)   dv   + \frac{C_1 }{r_0 ^2 }  \|f ^n\|^2_{L^1_2 ( \R^3_v ) } +  \iiint_{ \mathbb{R}^{ 3}_v \times \R^{3}_{v_*} \times  \S^2 } B_n f^n ( {f^n_*}' -f^n_*)   d vdv_* d \sigma \right\} ,
  \end{aligned}
\end{equation*}
the quantity $\iiint_{ \mathbb{R}^{ 3}_v \times \R^{3}_{v_*} \times  \S^2 } B_n f^n ( {f^n_*}' -f^n_*)   d vdv_* d \sigma$ in the right-hand side of the above inequality is bounded by the collision integral $\iint_{ \R_v^3 \times \R_{v_*}^3 } f^n f_*^n |v - v_*|^\gamma d v d v_*$ via the cancellation lemma (Lemma \ref{prop-S}). In deriving the average time-velocity estimate, the cancellation lemma is also employed. Then the quantity $\int_0^T d t \int_{\R_v^3} ( \mathcal{R}_3 )^n d v$ is bounded by
$$
 \int_{0}^{T}   dt \iint_{\R^3_v \times \R^3_{v_*}} f^n f^n_* |v-v_*|^\gamma \, dv \, dv_* + \frac{1}{2} \int_{0}^{T} \! dt \int_{\R^3_v} e^n \, dv 
 $$
for both hard and soft potentials by means of the cancellation lemma; see \eqref{R3n-1} and Remark \ref{Rmk-6.2} below. Consequently, we must control the collision integral $\iint_{ \R_v^3 \times \R_{v_*}^3 } f^n f_*^n |v - v_*|^\gamma d v d v_*$ for the soft potential model $- 3 < \gamma < 0$. Unfortunately, we construct a counterexample showing that there exist a non-negative function $f(v)$ with finite mass, energy, and entropy 
\begin{equation*}
  \begin{aligned}
    \int_{\R_v^3} f ( 1 + |v|^2 + |\log f| ) d v < \infty \,,
  \end{aligned}
\end{equation*}
such the collision integral
\begin{equation*}
  \begin{aligned}
    \iint_{\R_v^3 \times \R_{v_*}^3} f (v) f (v_*) |v - v_*|^\gamma d v d v_* = + \infty \,,
  \end{aligned}
\end{equation*}
see Example \ref{soft-f} below. Therefore, the approach developed for the hard potential model in this paper does not apply to the soft potential model.

In fact, the soft potential case $- 3 < \gamma < 0$ in the work of Alexandre and Villani \cite{Alex-Vill-2002-CPAM} encounters the same obstacle. When proving the strong convergence of $f^n$ in the $L^1$ sense, they employed the following estimate (corresponding to the estimate in Lemma \ref{L-FZn} of the present work)
\begin{equation*}
  \begin{aligned}
    \int_{|\xi| \geq 1} |\mathcal{F} \sqrt{f^n_R}|^2 Z_n ( \tfrac{1}{|\xi|} ) d \xi \leq C (f^n, R, \Phi_0) [ D (f^n) + \| f^n \|^2_{L^1_2} ] \,,
  \end{aligned}
\end{equation*}
see (4.4) on Page 58 of \cite{Alex-Vill-2002-CPAM}. This estimate is cited from the work of Alexandre, Desvillettes, Villani and Wennberg \cite{ADVW-bound-2000} (Theorem 1 on Page 331-332). Note that in \cite{ADVW-bound-2000}, the assumption $M (|z|) + |z| M' (|z|) \leq C_0 (1 + |z|)^2$ is required\footnote{They used the notation $\Lambda (|z|) + |z| \Lambda' (|z|) \leq C_0 (1 + |z|)^2$, where $\Lambda (|z|)$ and $\Lambda' (|z|)$ were defined in (14) and (15) on Page 331 of \cite{ADVW-bound-2000}, repectively, which coincide with the same definitions of $M (|z|)$ and $M' (|z|)$ in the present work.}. In \cite{Alex-Vill-2002-CPAM}, the authors only assumed $M (|z|), |z| M' (|z|) \in L^1_{loc} (\R^N)$ (see Page 42 of \cite{Alex-Vill-2002-CPAM}), which includes the case $M (|z|) + |z| M' (|z|) \leq C_1 |z|^\gamma \in L^1_{loc} (\R^N)$ for $- N < \gamma < 0$. In this case, the collision integral $\iint_{ \R_v^3 \times \R_{v_*}^3 } f^n f_*^n |v - v_*|^\gamma d v d v_*$ remains uncontrollable. For the collision kernel $B (z, \sigma) = |z|^\gamma b (k \cdot \sigma)$ considered in this paper, the condition $M (|z|) + |z| M' (|z|) \leq C_0 (1 + |z|)^2$ corresponds to the hard potential $0 \leq \gamma \leq 1$, while $M (|z|) + |z| M' (|z|) \leq C_1 |z|^\gamma $ corresponds to soft potential $- 3 < \gamma < 0$.

\subsection{Historical remarks.}
Research on the Boltzmann equation and related models can generally be classified into two main categories: the regime of classical solutions and the regime of renormalized solutions. In the framework of classical solutions, works typically rely on the assumption of small initial data (see, for instance, \cite{Guo-classical-2006,Guo-classical-2010,GJJ-classical-2010,JL-classical-2022,JLT-classical-2024,JLZ-classical-2023,JXZ-classical-2018,SG-classical-2008}). The following review focuses primarily on the regime of renormalized solutions.

In the early stage of studying the Boltzmann equation, the lack of global a priori estimates for solutions and their derivatives limited the analysis. As a result, most works focused on the spatially homogeneous case, such as \cite{Arkeryd-part1-1972,Arkeryd-part2-1972,Arkeryd-existence-1981, MORGENSTERN-Maxwellian-1954}. Later, Kaniel and Shinbrot \cite{KS-locexi-1978} extended the study to the spatially inhomogeneous setting in bounded domains. They constructed the approximated solutions controlled by a local Maxwellian distribution and proved the existence and uniqueness of solutions to the initial boundary value problem for a short time interval. 

A milestone breakthrough in this field began in 1989, when DiPerna and Lions \cite{DL- renormalized-1989} introduced the renormalized solution to the transported equations. Subsequently, in \cite{Diperna-Lions} they established the existence theory of global renormalized solutions for the classical Boltzmann equation satisfying Grad's angular cutoff assumption with large initial data in the whole space. In \cite{DL-enpropy-1991}, they proved that the renormalized solutions to the cutoff Boltzmann equation satisfy the entropy inequality. However, for long-range interactions, such as those described by the hard and soft potential models, the collision cross section exhibits non-integrable singularities primarily in the angular variable. In this case, the collision kernel $B(v - v_*, \sigma)$  is not integrable with respect to the angular variable. Such singular behavior lies beyond the scope of the DiPerna-Lions theory.

Focusing on cutoff models, Lions studied the convergence behavior of solution sequences in the series works \cite{Lions-Compactness-1,Lions-Compactness-2}. He showed that the sequence of solutions $ f^n $ is strongly compact in $ L^1([0,T] \times \mathbb{R}^3_x \times \mathbb{R}^3_v)$ if and only if the initial data sequence $ f_0^n $ is strongly compact in $ L^1(\mathbb{R}^3 _x\times \mathbb{R}^3_v ) $. Based on this result, Lions proposed a conjecture that this reversible propagation of convergence comes from the angular cutoff assumption in the collision operator. In contrast, for non-cutoff models such as the Landau equation, the evolution process itself may create a regularizing effect. This means that the solution sequence can become strongly compact even when the initial data do not have this property. Subsequently, in \cite{Lions-Landau-1994}, Lions studied the Cauchy problem for the Landau equation and confirmed this idea. He proved that the solution sequence of the Landau equation enjoys strong compactness, while showing that such a property does not hold for the Boltzmann equation under the angular cutoff assumption. Inspired by these results, Villani \cite{Villani-Landau-1996} proved in 1996 the global existence of renormalized solutions for the Landau equation with defect measure.

Regarding the regularity of the non-cutoff Boltzmann equation, Desvillettes \cite{Desvillettes-Kac-1995,Desvillettes-Boltzmann-1996,Desvillettes-Maxwellian-1997} carried out a systematic study of its regularizing effects. He first established the theory for the two-dimensional equation with radially symmetric Maxwellian molecules. He then extended these results to the case of radially symmetric, velocity-dependent cross sections, and further to the case of non-radially symmetric Maxwellian molecules. 

In terms of entropy dissipation estimates, Villani \cite{Villani-regular-1999} employed a decomposition of the entropy dissipation together with the Carleman representation to show, for the first time, that under the assumption of a local lower bound on $f$, the Sobolev norm of $\sqrt{f}$ can be controlled by the entropy dissipation. Subsequently, Alexandre, Desvillettes, Villani, and Wennberg \cite{ADVW-bound-2000} refined these conclusions using tools such as Fourier analysis, successfully providing a complete proof without the lower bound assumption. Building on this, Alexandre and Villani \cite{Alex-Vill-2002-CPAM} introduced renormalized solutions with defect measures for the Boltzmann equation with long-range interactions. They established the global existence of such renormalized solutions and proved their strong compactness, thus confirming Lions’ conjecture.

The works mentioned above on renormalized solutions were all carried out in the whole space $\R_x^3$. For bounded domains with various boundary conditions, several results have been obtained so far. The first breakthrough was achieved by Mischler \cite{Mischler-ASENS-2010} in 2010, who proved the existence of renormalized solutions to the cutoff Boltzmann equation, Vlasov-Poisson and Fokker-Planck type models, in bounded domains with Maxwell reflection boundary conditions. In 2019, Jiang-Zhang investigated renormalized solutions to the
Boltzmann equation with angular cutoff model \cite{JZ-2019-SIMA}, and renormalized solutions with a non-negative defect measure to the Boltzmann equation without angular cutoff model \cite{JZ-2019-JDE}, both in bounded domains with incoming boundary conditions.

In the past three decades, the hydrodynamic limit of the Boltzmann equation has attracted considerable attention. Research in the framework of renormalized solutions began with the work of Bardos, Golse, and Levermore in the 1990s \cite{BGL-limit1-1991,BGL-limit2-1993}, who established a unified framework linking the DiPerna-Lions renormalized solutions of the Boltzmann equation with the Leray weak solutions of the incompressible Navier–Stokes equations. Building on this foundation, Golse and Saint-Raymond \cite{GS-limit1-2004,GS-limit2-2009} achieved a major breakthrough by rigorously deriving macroscopic fluid equations from the microscopic kinetic description without relying on any nonlinear weak compactness assumptions. For problems on bounded spatial domains, serval related works have also been developed within the renormalized solutions framework. In \cite{MS-liminbdd-2003}, Masmoudi and Saint-Raymond proved the linear Stokes limit from fluctuations of renormalized solutions to the Boltzmann equation in a bounded domain. In \cite{Saint-limit-2009}, Saint-Raymond rigorously justified the Navier-Stokes-Fourier limit in the renormalized solutions regime over a bounded domain. In \cite{MS-liminbdd-2003,Saint-limit-2009}, the fluid boundary conditions were either the Dirichlet or the Navier slip boundary condition, depending on the relative sizes of the accommodation coefficient and the Knudsen number $Kn$. In \cite{JM-liminbdd-2017}, Jiang and Masmoudi rigorously proved the Navier-Stokes-Fourier limit with Neumann boundary condition from renormalized solutions to the Boltzmann equation in a bounded domain with Maxwell reflection boundary condition for small accommodation coefficient $\sqrt{Kn}$. Moreover, Jiang and Zhang \cite{JZ-2019-SIMA, JZ-2019-JDE} rigorously verified the Navier-Stokes-Fourier limit from renormalized solutions to the Boltzmann equation (including both cutoff and non-cutoff cases) in a bounded domain with incoming boundary condition.

\subsection{Sketch of proofs}
In this subsection, we mainly sketch the ideas of proving the main
theorem. In 2002, Alexandre and Villani \cite{Alex-Vill-2002-CPAM} established the existence of renormalized solutions to the Boltzmann equation without angular cutoff in the whole space, introducing a non-negative defect measure. 
Due to the singularity of the collision kernel in the angular variable, the classical approach of splitting the collision operator into “gain” and “loss” terms is no longer valid. To overcome this, the renormalized collision operator $\beta'_\delta(f) Q(f,f)$ is decomposed into three components: $(\mathcal{R}_1)$, $(\mathcal{R}_2)$, and $(\mathcal{R}_3)$. Consequently, the existence of renormalized solutions reduces to establishing the convergence of the approximated operators $(\mathcal{R}_1)^n$, $(\mathcal{R}_2)^n$, and $(\mathcal{R}_3)^n$. Among these, the convergence of $(\mathcal{R}_3)^n$ presents the primary difficulty. In earlier works, due to the absence of sufficient a priori estimates for this term, one could only apply Fatou’s lemma (Lemma \ref{fatou}) to obtain lower semicontinuity, which naturally leads to the appearance of a non-negative defect measure in the solution.
However, for hard potential models, we observe that by expanding the squared term and utilizing the entropy inequality, one can derive the key average time-velocity estimate for $(\mathcal{R}_3)^n$. The equi-integrability of $\int_{\mathbb{R}^3_v} \big(1 + \frac{1}{n} \int_{\R^3 _v } f^n dv \big)^{-1} (\mathcal{R}_1)^nd v $ and $\int_{\mathbb{R}^3_v} \big(1 + \frac{1}{n} \int_{\R^3 _v } f^n dv \big)^{-1} (\mathcal{R}_2)^nd v$, together with the strong convergence of $f^n$, then implies the equi-integrability of $\int_0^T dt \int_{\mathbb{R}^3_v} \big(1 + \frac{1}{n} \int_{\R^3 _v } f^n dv \big)^{-1}(\mathcal{R}_3)^n \varphi \, dv$ on $B_R(x)$. To address the singularity at $\theta = 0$, we decompose the spherical integral of $(\mathcal{R}_3)^n$ into singular and non-singular regions. We establish the convergence of the non-singular part and show that the singular part converges uniformly to zero. Combining these results, we obtain the convergence of $(\mathcal{R}_3)^n$. This advancement enables us to prove the global existence of renormalized solutions for the non-cutoff Boltzmann equation with hard potentials in the standard sense, without requiring any defect measure.

{\bf Step 1. Construction of smooth approximation problem \eqref{approximation}.}
 We begin by constructing a smooth approximation problem  \eqref{approximation} for the non-cutoff Boltzmann equation.  
 The approximate initial data $ f_0^n $ are taken in the Schwartz space. They satisfy a Gaussian lower bound and converge to the original initial data $ f_0 $ in entropy and in weighted norms, see Lemma \ref{Lemma 4.8}. This choice provides sufficient regularity and ensures consistency with the original problem. 
 To handle the angular singularity associated with long range interactions, we introduce a cutoff collision kernel $ B_n(z,\sigma) $. The angular part $ b_n(\cos \theta)$ is modified by a smooth function depending on $ n \sin(\theta/2) $. This modification removes the contribution of grazing collisions when $ \theta$ is close to zero.
 
 Based on this modified kernel $B_n(z, \sigma)$, one can construct the approximated collision operator $\tilde{Q}_n$ in \eqref{Q L}.   Note that the factor $\left(1 + \frac{1}{n} \int_{\R^3 _v } f dv \right)^{-1}$  is such that the nonlinear approximated collision operator $\tilde{Q}_n (f^n, f^n)$ admits linear upper bounds in $L^1 \cap L^\infty(\R^3_v )$ as in Lemma \ref{1-inf} and ensures Lipschitz continuity in the $ L^1(\R^3_v ) $ space.  These properties imply the existence, uniqueness, and positivity of a global smooth solution $ f^n $ to the approximation problem \eqref{approximation}. 
 Finally, the approximate solutions preserve the fundamental physical properties of the Boltzmann equation, including the conservation of mass, momentum, and energy. Besides, an entropy identity is also satisfied and provides uniform in $ n $ bounds on entropy and velocity moments, see Lemma \ref{Lemma 4.17}. These uniform estimates supply the compactness required for the limit procedure.

{\bf Step 2. Well-definition of the renormalized collision operator.} 
To establish the validity of the collision operator for the non-cutoff Boltzmann equation, we adopt the renormalization method. We first decompose the renormalized collision term $\beta _\delta '(f) Q(f,f) $ into three distinct parts, denoted by $(\mathcal{R}_1), \ (\mathcal{R}_2)$  and $(\mathcal{R}_3)$. This decomposition allows us to estimate each term separately. 
First, we analyze the term $(\mathcal{R}_1)$. This term involves a singular integral operator defined as $	\mathscr{S} f$. By applying the Cancellation lemma (Lemma \ref{prop-S}), we effectively handles the singularity and proves  $ ( \mathcal{R}_1 ) \in L^\infty \left( [0,T ] ; L^1 \left( \R ^3_x \times B_R ( v) \right)\right), \ \forall R >0$. 
Subsequently,  to address the critical case,   we employ test functions in the Hölder space $C^\alpha_c(\mathbb{R}^3_v)$ with $\nu < \alpha < 1$. Utilizing the $\alpha$-Hölder regularity, we control the singularity of the collision kernel and estimate the bilinear term $(\mathcal{R}_2)$ directly.  This leads to a more regular dual space compared to the $W^{1,\infty}_0$ framework. Specifically, we show that $  (\mathcal{R}_2) \in L^\infty \left( [0,T]; L^1 \left( \mathbb{R}^3_x; B^{-\alpha}_{1,1} ( B_R (v)) \right) \right).$
Finally, we handle the integrability of $(\mathcal{R}_3 )$. This term arises from the convexity of the renormalization function $\beta_\delta(f)$ and is strictly non-negative.
We multiply the renormalized Boltzmann equation by a test function $\varphi$ and integrate it over $t,x,v$. Together with the bounds for $(\mathcal{R}_1)$ and $(\mathcal{R}_2)$ obtained earlier,   we deduce that $(\mathcal{R}_3 )$ must also be finite in $  L^1 \left( [0,T ] \times \R ^3_x \times B_R ( v)  \right) $. Thus, the renormalized collision operator is well-defined as a distribution. 

{\bf Step 3. Strong compactness of approximated solution.} By the uniform bound \eqref{FB2} and \eqref{FWC2} of $f^n $ and the Dunford-Pettis theorem (Lemma \ref{theorem-dunford}), it follows directly that  $f^n$ converges weakly to $f$ in $L^1 ( [0,T] \times \R^3_x \times \R^3_v )$.
To establish that the approximated solution $f^n$ converges strongly to $f$ in $L^1([0,T] \times \R^3_x  \times \R^3_v )$,  we proceed in several steps. First, based on the preliminary lemmas,  we establish the functional inequality \eqref{F-Zn}, which shows that the 
high-frequency components of $\sqrt{f^n}$ in the velocity variable are controlled by the entropy dissipation.

 Next, applying the averaged velocity lemma (Lemma \ref{theo-6.2.1}), we derive that  $\int_{\mathbb{R}^3_v } f^n  dv$ is compact in the space $L^1([0,T] \times \mathbb{R}^3_x )$.
Subsequently, combining the reflexivity of $L^p (1<p<\infty)$ spaces, we consider the weak limit $g$ of $g^n = \sqrt{f^n}$ in $L^2([0,T] \times B_R (x)\times B_R(v))$. The discussion is divided into two cases. For almost every $t,x,v$ satisfying $g(t,x,v)=0$, the weak convergence property and the non-negativity of $  f^n $ imply that $\sqrt{f^n}$ converges strongly to zero. For almost all points satisfying $g(t,x,v)>0$, the smoothing estimate \eqref{F-Zn} implies that the high-frequency components of the solution vanish uniformly. Accordingly, combined with the properties of the Fourier transform and the M. Riesz-Fr\'{e}chet-Kolmogorov Theorem (Lemma \ref{MRFK}),  we prove that $\sqrt{f^n}$ converges strongly in $L^2([0,T] \times B_R (x)\times B_R(v))$. This implies that $f^n$ converges in measure. Combining the weak compactness of $f^n$, the uniform bound \eqref{FB2} and the Vitali Convergence Theorem (Lemma \ref{L-VC}), we prove that $f^n$ converges strongly to $f$ in $L^1([0,T] \times \mathbb{R}^3_x \times \mathbb{R}^3_v )$. Furthermore, we give a counterexample showing that the above method is not applicable to soft potentials with $ -3 < \gamma < 0 $. Even if finite entropy and the standard conservation laws hold, the distribution function may still develop local concentration. When this behavior interacts with the singular factor $ |v - v_*|^\gamma $, the collision frequency can diverge. Consequently, the corresponding upper bound becomes infinite, i.e., the argument fails in the soft potential case.

{\bf Step 4. The convergence  of renormalized collision operator.}   Since renormalized collision operator $\beta_\delta ' (f^n ) \tilde{Q}_n (f^n, f^n)$ can be decomposed into three sidtinct parts, denoted by $( \mathcal{R} _1) ^n, ( \mathcal{R} _2) ^n$ and $( \mathcal{R} _3 ) ^n$, verifying the convergence of $ \beta_\delta ' (f^n ) \tilde{Q} _n (f^n, f^n) $ reduces to verifying the convergence of the decomposed terms $( \mathcal{R} _1)^n, \ (\mathcal{R} _2) ^n$  and $(\mathcal{R} _3) ^n$. To study the convergence of the renormalized collision operator $(\mathcal{R}_1)^n$, we   truncate the velocity variable $v_*$. This allows us to split the estimates into two parts: one on compact sets and one at infinity.  
Next, by applying the M. Riesz–Fréchet–Kolmogorov theorem (Lemma \ref{MRFK}), we prove that the convolution term $f^n *_v S^n$ is relatively compact in $L^1([0,T]\times \mathbb{R}^3_x \times B_R(v))$, thereby deriving its strong convergence.  
Finally, using the Product Limit Lemma (Lemma \ref{product limit}), we successfully establish the weak convergence \eqref{cnv-R1} of the term  $(\mathcal{R}_1)^n$, i.e.,
   \begin{align*} 
  	\lim_{n \to \infty}  \int_{0}^{T} dt \iint_{ \mathbb{R}^3_x \times \mathbb{R}^3_v } \left( 1+ \frac{1}{n} \int_{\R^3_v } f^n d v \right)^{-1}  ( \mathcal{R}_1 )^n \varphi d x d v =  \int_{0}^{T} dt \iint_{ \mathbb{R}^3_x \times \mathbb{R}^3_v } ( \mathcal{R}_1 )  \varphi d x d v .
  \end{align*}

 For the limit of the term $(\mathcal{R}_2)^n$, we split the integral according to the size of the relative velocity $|v - v_*|$.
 On the truncated region, we use the strong convergence of $f^n$ and $\beta \in C^2(\R_+)$ to obtain weak compactness of  $ \left( 1+ \frac{1}{n} \int_{\R^3_v} f^n d v \right)^{-1}  f^n_* \beta_\delta ( f^n  ) $. Together with the distributional convergence of the operator $\mathcal{T}^n$, this allows us to identify the limit. 
 For large relative velocities, the contribution is controlled by uniform bound \eqref{FB2} and therefore goes to zero as the truncation radius goes to infinity. Combining the two cases above, the convergence \eqref{cnv-R2} holds, i.e., 
 \begin{align*} 
	\lim_{n \to \infty} \int_{0}^{T} dt \iint_{  \mathbb{R}^3_x \times \mathbb{R}^3 _v }  \left( 1+ \frac{1}{n} \int_{\R^3_v } f^n d v \right)^{-1}  ( \mathcal{R}_2 )^n \varphi d x d v =\int_{0}^{T} dt \iint_{  \mathbb{R}^3 _x \times \mathbb{R}^3_v  }    ( \mathcal{R}_2 )  \varphi d x d v .
\end{align*}

Finally, we address the convergence of $(\mathcal{R}_3)^n$. Specifically, we expand the squared term and combine it with the entropy inequality to derive the average time-velocity estimate \eqref{R3n-tv} for $  (\mathcal{R}_3)^n  $.  
Subsequently, by integrating the renormalized approximate Boltzmann equation \eqref{eq6.75} with respect to $t$, $x$, and $v$, and using the established weak compactness of $f^n$ along with that of the terms $\int_{\mathbb{R}^3} \big(1+ \frac{1}{n} \int_{\mathbb{R}^3_v } f^n dv \big)^{-1}(\mathcal{R}_1)^n \varphi dv$ and $\int_{\mathbb{R}^3_v } \big(1+ \frac{1}{n} \int_{\mathbb{R}^3_v } f^n dv \big)^{-1}(\mathcal{R}_2)^n \varphi dv$, we obtain the equi-integrability of  $
	\int_{0}^{T} dt \int_{\mathbb{R}^3_v } \big( 1+ \frac{1}{n} \int_{\mathbb{R}^3_v } f^n dv \big)^{-1} (\mathcal{R}_3)^n \varphi dv $
on $B_R(x)$. 

Based on these facts,  to address the singularity of the collision kernel at $ \theta = 0 $,  we decompose  $ (\mathcal{R}_3)^n$  into two parts, denoted by  
$ (\mathcal{R}_3)^{n,k} $ and  $ (\mathcal{R}_3)^{n}_k $. The term $ (\mathcal{R}_3)^{n,k} $ corresponds to the truncated collision kernel  $B_{n}^k (|z|,\cos \theta ) = B(|z|,\cos \theta )  \mathbf{1}_{ \frac{1}{k} \leq \theta \leq \frac{\pi}{2} }$, while  $ (\mathcal{R}_3)^{n}_k $ accounts for the singular part of the kernel, defined by  $B_{n,k} (|z|,\cos \theta ) = B_{n} (|z|,\cos \theta )  \mathbf{1}_{ 0 \leq \theta \leq \frac{1}{k}}$. Furthermore, we define the functionals $M^n(x)$ and $ M(x)$ and the   sets $N^n_L$ and $ N_L$, in \eqref{ML}.  Consequently,  the average time-velocity estimate \eqref{R3n-tv}  implies that the integral $ \int_0^T dt \int_{  {\R}^3_v }     (\mathcal{R}_3) ^{n }    dv$ is uniformly bounded on the set $(N_L^n)^c$. Combining this bound with the convergence $f^n \to f$ a.e. in $(t,x,v)$ and $B^k_n \to B^k$ a.e. in $(v_*,\sigma)$, we deduce the convergence of the non-singular part $ (\mathcal{R}_3)^{n,k} $, namely 
\begin{align*} 
	\lim\limits_{n\to \infty} \int_0^T dt \iint_{\mathbb{R}^3_x \times \mathbb{R}^3_v }  \left(1 + \frac{1}{n} \int_{\mathbb{R}^3_v} f^n \, dv \right)^{-1}  (\mathcal{R}_3) ^{n, k} \varphi \mathbf{1}_{ (N_L^n)^c}  dx dv = \int_0^T dt \iint_{\mathbb{R}^3_x \times \mathbb{R}^3_v}  (\mathcal{R}_3) ^k \varphi \mathbf{1}_{(N_L)^c} \, dx dv. 
\end{align*}

For the  singular part, we aim to prove that the corresponding integral converges uniformly to zero. To this end, we established an estimate \eqref{Hnu} in Lemma \ref{LBep}. This estimate indicates that the angular singularity of the collision kernel at $\theta \approx 0$ exhibits fractional derivative behavior. Furthermore, the corresponding integral contribution is controlled by the fractional Sobolev norm $\|g\|_{H^{\nu/2}(B_R(v))}$ of the distribution function $g$ and tends uniformly to zero as the truncation angle decreases (i.e., $k \to \infty$). Recall that the estimate \eqref{F-Zn} implies that for hard potentials, the $H^{\nu/2}(B_R(v))$-norm of $\sqrt{f^n} $ is bounded by 
$$
 \frac{  C_0  J^2   (2\pi )^3   }{  \|f ^n \|_{L^1( B_R(v)) } \left( r_0 /2 \right)^\gamma    }  \left\{   \int_{  \mathbb{R}^{3}_v } e^n  ( t,x,v) \,   dv   +   \left( C_{\nu }  + \frac{ C_1 }{ r ^2 _0 } \right)  \|f ^n\|^2_{L^1_2 ( \R^3_v ) }   \right\}.
$$
Furthermore, utilizing the sets $W_\epsilon$ and $U_\delta$ defined in Lemma  \ref{v-regular}, we decompose the domain  $[0,T] \times B_R (x)$ into three parts: $W_\epsilon  \setminus (U_\delta \cup U_L^n)$, $U_\delta \cup U_L^n$ and $([0,T] \times B_R (x))  \setminus W_\epsilon  $. By virtue of the boundedness of $\|\sqrt{f^n} \|_{L^2 ([0,T] \times B_R (x) ; H^{\nu/2} (B_R(v)))}$ on the set $W_\epsilon \setminus (U_\delta \cup U_L^n)$, combined with the uniform lower bound of $\int_{B_R(v)} f^n \, dv$ on this set for sufficiently large $n$, we deduce that
$$
\lim_{k \to \infty}	\sup_{n>N} \int_0^T dt \iint_{\mathbb{R}^3_x \times \mathbb{R}^3_v} ( \mathcal{R}_3 )_k^n \varphi \mathbf{1}_{W_\epsilon \setminus (U_\delta \cup U_L^n)} \, dx dv=0.
$$
The small measure of the set $U_\delta \cup U_L^n$ implies that
\begin{equation*}
	\sup_{n>N} \int_0^T dt \iint_{\mathbb{R}^3_x \times \mathbb{R}^3_v} \left(1 + \frac{1}{n} \int_{\mathbb{R}^3_v} f^n \, dv \right)^{-1} (\mathcal{R}_3)_k^n \varphi \mathbf{1}_{(U_\delta \cup U_L^n)} \, dx dv \to 0 \quad ( \delta \to 0 , L \to \infty ).
\end{equation*}
Finally, the estimate $  \|f ^n \|_{L^1( B_R(v))} \leq 2\epsilon \ $  almost everywhere on $ ([0,T] \times B_R(x)) \setminus W_\epsilon$, together with the average time-velocity estimate for  $  ( \mathcal{R}_3 ) ^n$, implies that
\begin{equation*}
\lim_{\epsilon \to 0}\sup_{n>N} \int_0^T dt \iint_{\R^3_x \times \mathbb{R}^3_v} ( \mathcal{R}_3 )_k^n  \varphi \mathbf{1}_{([0,T] \times B_R(x)) \setminus W_\epsilon} \, dx dv =0.
\end{equation*}
Moreover, for any $1 \leq n \leq N$, it is straightforward to verify that
\begin{align*}  
	\lim_{k \to \infty} \int_0^T dt \iint_{\mathbb{R}^3_x \times \mathbb{R}^3_v} \left(\left(1 + \frac{1}{n} \int_{ {\R}^3_v} f^n \, dv \right)^{-1}(\mathcal{R}_3)^n_k +( \mathcal{R}_3 )_k\right)    \varphi   dx dv = 0, \quad \forall   1 \leq n \leq N. 
\end{align*} 
These convergence results imply that
\begin{align*} 
	\lim_{k \to \infty } \sup_{n} \int_0^T dt 	\iint_{\R^3_x \times \mathbb{R}^3_v}  \left( \left(1 + \frac{1}{n} \int_{ {\R}^3_v} f^n \, dv \right)^{-1}( \mathcal{R}_3 )_k^n   + ( \mathcal{R}_3 )_k \right)\varphi dxdv =0.
\end{align*}
Consequently, combining this with the convergence of the non-singular part $ (\mathcal{R}_3)^{n,k} $, we obtain the convergence of $(\mathcal{R}_3) ^{n } $, i.e., 
\begin{equation*} 
	\lim \limits_{n\rightarrow \infty} \int_{0}^{T} dt \iint_{ \R ^3 _x \times \R ^3_v  } \left( 1+ \frac{1}{n} \int_{\R^3_v } f^n d v \right)^{-1}( \mathcal{R} _3)^n \varphi  d x d v = \int_{0}^{T} dt \iint_{ \R ^3 _x \times \R ^3_v  } ( \mathcal{R} _3) \varphi  d x d v.
\end{equation*}

{\bf Step 5. Existence and entropy inequality.}  
By combining the convergence properties \eqref{cnv-R1}, \eqref{cnv-R2}, and \eqref{cnv-R3}, for any test function $\varphi(t,x,v) \in L^\infty([0,T] \times \mathbb{R}^3_x; C^\alpha_c(\mathbb{R}^3_v))$ with compact support in $[0,T] \times B_R(x) \times B_R(v)$, taking the limit $n \to \infty$ in the approximating equation yields: 
\begin{equation*} 
	\iint_{\mathbb{R}^3_x \times \mathbb{R}^3_v} (\beta_\delta(f)(T,\cdot) - \beta_\delta(f)(0,\cdot))\varphi \, dxdv = \int_0^T dt \iint_{\mathbb{R}^3_x \times \mathbb{R}^3_v} \beta_\delta'(f) Q(f,f) \varphi \, dxdv. 
\end{equation*} 
This verifies that $f$ is indeed a renormalized solution of the Boltzmann equation \eqref{Boltz}. Furthermore, by invoking the absolute continuity of the integral and the convergence relationship between $\beta_\delta(f)$ and $f$ as $\delta \to 0$, it can be deduced that the limit function is continuous in time with respect to the weak topology, i.e., $f \in C\left(\mathbb{R}_+; w\text{-}L^1(\mathbb{R}^3_x; B^{-\alpha}_{1,1}(\mathbb{R}^3_v))\right)_+$. In addition, by applying the Monotone Convergence Theorem and the estimates for the $H$-functional $H(f)$, it is shown that $f$ satisfies the following uniform bound: 
\begin{equation*} 
	\sup_{t \in [0,T]} \iint_{\mathbb{R}^3_x \times \mathbb{R}^3_v} f(1 + |x|^2 + |v|^2 + |\log f|) \, dxdv \leq C_T. 
\end{equation*} 

Finally, we establish the entropy inequality \eqref{entropy-th}. To derive the desired entropy inequality from the approximated entropy identity \eqref{e na}, it suffices to verify the bound \eqref{ena}, i.e., $ \liminf_{n \to + \infty} \int_{   \R ^3 _v} \tilde{e} ^n  d v  \geq \int_{   \R ^3 _v} e d v  $. We first utilize the convergence \eqref{fn-f} and the bound $f^n \log f^n \in L^\infty ((0,T); L^1(\mathbb{R}^3_x \times \mathbb{R}^3_v ))$ to establish that the sequences ${f ^n}' {f_{*}^n}' B^R_n$ and $f^n f^n_{ * } B^R_n$ are relatively weakly compact in $L^1(E_R,d \Theta  )$ for almost all $(t,x) \in (0,R) \times B_R(x)$.  We then can verify that for all $\varphi \in L^\infty (E_R)$,
\begin{align*}
	\int_{E_R}   N(f^n)^{-1}   f^n  f^n_{ * }      B^R_n   \varphi d \Theta   \to&  \int_{E_R} N(f)^{-1}  f   f _{ * }       B^R     \varphi d \Theta, \\
		\int_{E_R}  N(f^n)^{-1} {f^n}' {f_{  *}^n}'   B^R_n   \varphi   d\Theta  \to& \int_{E_R}   N(f )^{-1} {f }' f_{*}'   B^R    \varphi  d\Theta
\end{align*}
strongly in $L^1((0,R)\times B_R(x))$, where $ N (f) = 1+\int_{   \R ^3_v } f dv $. Together with the fact that for any Borel subset $A$ of $E_R$,  there exists a sequence of measurable subsets $\{ A_k\}_{k=1}^\infty $ such that $\mathbf{1}_{A_k} \to \mathbf{1}_A $ in $L^1( E_R)$. We can  obtain that for almost all $(t,x)\in (0,R)\times B_R(x)$,
\begin{equation*} 
	\frac{ f^n  f^n_{ * }   }{1+ \frac{1}{n} \int_{\R^3_v } f^n d v}  B^R_{  n}  \rightarrow {f  }   {f_* }  B^R, \quad \frac{ {f ^n}'  {f_*^n}'   }{1+ \frac{1}{n} \int_{\R^3_v } f^n d v}	   B^R_n   \rightarrow \ f'   f'_*   B^R\quad  \text{weakly in } L^1(E_R, d \Theta).
\end{equation*} 
  At the end, since the function 
\begin{equation*}
	j(a,b)=\left\{\begin{array}{rl}
		(a-b) \log \frac{a}{b}, & \quad \text{for } a, b>0 \,  \\
		+\infty, & \quad \text{for } a \text { or } b \leq 0 \, 
	\end{array}\right.
\end{equation*}
is convex, the property of weak lower semicontinuity implies that  $\liminf\limits_{n\rightarrow \infty} \int_{\R^3 _v} \tilde{e}^n d v \geq \int_{\R^3_v } e d v$.
So the entropy inequality \eqref{entropy-th} in Theorem \ref{MainThm} is obtained.

\subsection{Organizations of current paper}

In the next section, we begin by providing several fundamental results from functional analysis that serve as the technical foundation for our subsequent proofs. In  Section \ref{Sec-GEAE}, we construct the  approximate equations, establish the existence and uniqueness of their smooth solutions, and demonstrate that the resulting sequence of solutions satisfies the  conservation laws and possesses certain  uniform bounds.  Section \ref{Sec-integrable}  is devoted to verifying that the renormalized collision terms $(\mathcal{R}_1)$, $(\mathcal{R}_2)$  and $(\mathcal{R}_3)$ are well-defined within the appropriate space of distribution functions.  Section \ref{Sec:Con} establishes the strong convergence of the solution sequence $f^n$ to $f$ in $L^1([0,T] \times \R^3_x \times \R^3_v)$. The convergence of the renormalized collision operators $(\mathcal{R}_1)^n$ and $(\mathcal{R}_2)^n$  is then addressed in Section \ref{Sec-cnv}, while Section \ref{Sec:Cnv-R3} is devoted to the analysis of  $(\mathcal{R}_3)^n$. In Section \ref{Sec-ee}, we prove that $f$ is a renormalized solution to the Boltzmann equation \eqref{Boltz} and show that it satisfies the entropy inequality. Finally, in Section \ref{sec-cnv-soft} we provide a counterexample that there exists a positive function $f(v)$ with finite mass, energy, and entropy such that the collision integral diverges for the soft potential case $-3 < \gamma < 0$.

\section{Toolbox}
First, we present some commonly used conclusions in functional analysis, which play a crucial role in verifying convergence later.
\begin{definition}\label{def-equi}
	Let $\mathscr{F} $ be a subset of $L^1(X, \mathcal{M},\mu )$. We say that $\mathscr{F}$ is equi-integrable if for every $\epsilon >0 $ there exists $\delta >0 $ such that for any measurable set $E \subset X $ with $\mu (E ) < \delta $ and for all $f\in \mathscr{F}$,
	\begin{equation*}
		\int_{E} |f|dx < \epsilon .
	\end{equation*}
\end{definition}

\begin{lemma}[Dunford-Pettis, \cite{Dunford-Schwartz-1958}] \label{theorem-dunford}
	Suppose that $\mathscr{F}$ is a Bounded subset of $ L^1(X, \mathcal{M},\mu )$. $\mathscr{F}$ is weakly relatively compact if and only if 
	\begin{enumerate}
		\item {$\mathscr{F} $ is equi-integrable};
		\item {for any $\epsilon>0, $ there exists compact set $K \subset X$ such that 
			\begin{equation*}
				\int_{K^c} |f(x) | dx <\epsilon , \quad \forall f \in \mathscr{F}.
		\end{equation*}}
	\end{enumerate}
\end{lemma}
\begin{lemma}[Vitali Convergence Theorem, \cite{Dunford-Schwartz-1958,Folland-Real}]\label{L-VC} 
Suppose $1\leq p< \infty $ and $\{ f^n\}_{n\geq 1 } \subset L^p(X)$, then $\lim_{n\to \infty}\| f^n - f\|_{L^p(X)}=0$ if and only if 	 
 \begin{enumerate}
 \item {$f^n \to f $ in measure;}\label{measure}
 \item { $ \{f^n\} $  is equi-integrable;}
	\item {for any $\epsilon>0, $ there exists compact set $K \subset X$ such that 
	\begin{equation*}
	\sup_{n \geq 1}	\int_{K^c} |f^n (x)| dx  <\epsilon .
	\end{equation*}}
 \end{enumerate}
\end{lemma}
In fact, if a sequence converges in measure, there exists a subsequence that converges almost everywhere. In a finite measure space, convergence almost everywhere implies convergence in measure. Consequently, in the case of a finite measure space (where $X$ is a compact set), condition \eqref{measure} can be strengthened to the requirement that $f^n $ converges to $f$ almost everywhere. The detailed proof can be found in \cite{Folland-Real}.
 \begin{lemma}[Monotone Convergence Theorem, \cite{Brezis-2010}]\label{MCT}
 Let $\{f^n\}_{n\geq 1}$ be a sequence of measurable functions on $X$ and suppose that 
 \begin{enumerate}
 \item  {$  f_1 (x)  \leq f_2 (x)\leq \cdots  \leq f_n (x)  \leq f_{n+1} (x) \leq \cdots$ a.e.  on  $  X$;}
 \item  { $\sup_{n} \int_{   \R ^3 _x} f^n (x) d x < \infty $.}
 \end{enumerate}
 Then $f^n$ converges a.e. on $X$ to a finite limit, which we denote by $f$; the function $f$ belongs to $L^1( \R ^3_x)$ and 
 \begin{align*}
 \lim\limits_{n \to \infty} \int_{   \R ^3_x }| f^n  (x) -f(x) |dx =0 .
 \end{align*}
 \end{lemma}
 \begin{lemma}[Lebesgue's Dominated Convergence Theorem, \cite{Brezis-2010}]\label{LCT}
Suppose that $\{f^n\}_{n\geq 1}$ is a sequence of measurable functions on $X$ such that 
\begin{enumerate}
	\item  {$ 	f^n (x) \rightarrow f  (x)$, a.e. on  $X$;}
	\item  {there is a function $g \in L^1 (X)$ such that for all $n \in \mathbb{N}$, $|f^n(x) | \leq g(x) $ , a.e. on  $X$.}
 \end{enumerate}
Then $f \in L^1 (X)$ and 
\begin{align*}
\lim\limits_{n \to \infty}  \int_{   \R ^3_x}| f^n  (x) -f(x) |dx =0 .
\end{align*} 
\end{lemma}
\begin{lemma}[Fatou's lemma, \cite{Brezis-2010}]\label{fatou}
	If $ f^n : X \rightarrow [0,\infty ]$ is measurable, for each positive integer $n$, then  
	\begin{align*}
	\int_{   \R ^3 _x}\liminf_{  n \rightarrow \infty } f^n (x)dx \leq \liminf_{  n \rightarrow \infty } \int_{   \R ^3 _x} f^n (x) dx
	\end{align*}
\end{lemma}
\begin{lemma}[Egorov's Theorem, \cite{Brezis-2010}]\label{egorov}
If a sequence of measurable functions $f^n$ converges to a measurable function $f$ almost everywhere, then for any $\delta > 0$ there exists a measurable set $\Omega_\delta \subseteq \Omega$ such that $\mu(\Omega_\delta) \geq \mu(\Omega) - \delta$ and $f_n$ converges to $f$ uniformly on $\Omega_\delta$. 
\end{lemma}
\begin{lemma}[Product Limit Theorem, \cite{BGL-limit2-1993,Diperna-Lions}]\label{product limit} 
	Let $\mu$ be a finite, positive Borel measure on a Borel subset $X$ of $\R^3$. Assume that $f^n,f\in L^1(d \mu)$ and $g^n,g\in L^\infty(d \mu)$ satisfy
	\begin{enumerate}
		\item {$w-\lim \limits_{n\rightarrow \infty} f^n = f$ in $L^1(d \mu)$;}
		\item {$\{g^n\}_{n\geq 1}$ is bounded in $L^\infty (d \mu)$ and $\lim \limits_{n\rightarrow \infty} g^n=g, \text{a.e.}$}
	\end{enumerate}
	Then $w-\lim \limits_{n \rightarrow \infty} f^n g^n= fg$ in $L^1(d \mu)$, where “$w-\lim$” means weak limit.
\end{lemma}

\begin{lemma}[M. Riesz-Fr\'{e}chet-Kolmogorov Theorem, \cite{Brezis-2010}]\label{MRFK}
Suppose $\mathcal{F}$ is a bounded subset of $L^p ( \R^3_x) , \ 1 \leq p < +\infty $. Then $\mathcal{F}$ is relatively compact  if and only if the following statements hold:
\begin{enumerate}
\item { for any $\epsilon >0$, there exists $\delta >0$ such that for any function $f \in \mathcal{F} $ and any $ h\in \R^3$ with $ |h | < \delta$ one has 
\begin{equation*}
	\int_{ \R ^3 _x} | f( x+ h) - f( x ) |^p d x < \epsilon ;
	\end{equation*}}
\item { for any $\epsilon >0$, there exists a bounded measurable set $ K \subset \R^3$ such that for any $ f \in \mathcal{F} $,
\begin{equation*}
	\int_{\R^3 _x \setminus K } |f( x ) |^p d x < \epsilon.
	\end{equation*} }
\end{enumerate}
\end{lemma}
Next, we present the weak lower semicontinuity of entropy, See Lemma 2.6 in \cite{Elmroth-H}.
\begin{lemma}\label{Lemma 4.6}
	Assume that $f^n (x,v ) \in L^1(\R^3 _x \times \R^3_v )_+$ satisfies
	\begin{enumerate}
		\item {there exists a constant $M >0$ such that 
			$$\Vert f^n \Vert_{L^1_2( \R^3_x \times \R^3_v )} \leq M ,\ f^n   \log f^n  \in L^1(\R^3_x  \times \R^3_v  ); $$ }
		\item the sequence {$\{ f^n    \}_{n\geq 1}  $ converges to $f   $ weakly in $ L^1(\R^3_x  \times \R^3_v  ) $. }
	\end{enumerate} 
	Then 
	\begin{equation}\label{eq4.28}  
\iint_{\R^3_x \times \R^3_v } f  \log f  dxdv  	\leq   \liminf \limits_{n\rightarrow \infty }   \iint_{\R^3 _x\times \R^3 _v} f^n  \log f^n  dxdv  . 
	\end{equation} 
\end{lemma}
\begin{lemma}[Averaged velocity lemma, \cite{Glose-velocity,Glose-perthame-velocity}]\label{theo-6.2.1}
	Let $g^n (t,x, v )$ converge weakly to $g (t, x, v )$ in $L^1_{loc} ( (0,T) \times \R^3 _x\times \R^3_v )$, and $ f^n (t, x, v ) $ converge weakly to $f (t, x, v )$ in $L^1 ( (0,T) \times \R^3_x \times \R^3 _v)$. Moreover, they satisfy
	\begin{equation*}
		\tfrac{\partial f^n}{\partial t} + v \cdot \nabla_x f^n = g^n \quad \text{in}\ \mathscr{D}' ((0,T) \times \R^3 _x\times \R^3 _v) .
	\end{equation*}
	Set $\{\phi^n (t,x, v )\}_{n\geq 1} \subset L^\infty ((0,T)\times \R^3_x \times \R^3_v )$ converge to $ \phi (t, x, v ) $ a.e. $(0,T) \times \R^3_x \times \R^3_v$. Then the sequence $\int_{\R^3_v} f^n\phi^n dv  $ converges strongly to $ \int_{\R^3_v } f \phi dv $ in $L^1((0,T)\times \R^3_x)$.
\end{lemma}
Based on Lemma \ref{theo-6.2.1}, one can generalize the previous result. Let $(E,\mu )$ denote the given measure space.
\begin{lemma}[Generalized averaged velocity lemma, \cite{Glose-velocity,Glose-perthame-velocity}]\label{theo-6.4.1}
Suppose that the bounded sequence $\{	\phi^n (t, x, v , e)\}_{n\geq 1}$ in $L^\infty ((0,T)\times \R^3_x \times \R^3 _v; L^1(E))$ converges to a function $\phi \in L^\infty ((0,T)\times \R^3_x \times \R^3 _v; L^1(E))$ such that the following limit holds 
\begin{equation*}  
	\lim \limits_{n\rightarrow \infty} \Vert \phi^n -\phi \Vert_{L^1(E)}(t,x,v )=0 \text{ a.e. } (0,T)\times \R^3_x \times \R^3_v. 
\end{equation*} 
Moreover, assume that $f^n (t, x, v )$ and $g^n (t, x, v )$ satisfy the conditions of Lemma \ref{theo-6.2.1}. It follows that the sequence $\int_{\R^3_v } f^n\phi^n dv  $ converges strongly to $ \int_{\R^3 _v} f \phi dv $ in $L^1((0,T)\times \R^3_x\times E)$.
\end{lemma}

\section{global existence of approximation equation}\label{Sec-GEAE}
In this section, we focus on constructing the following approximate problem, which is given by  
\begin{equation} \label{approximation} 
	\tfrac{\partial f^n}{\partial t} + v \cdot \nabla_x f^n =\tilde{Q}_n(f^n, f^n) \,, \ f^n|_{t=0}=f^n_{0} \,, 
\end{equation}  
where $f^n_{0} = f^n( 0 ,x , v )$ denotes the approximate initial data and $\tilde{Q}_n (f^n, f^n)$ represents the nonlinear collision operator which will be explicitly constructed later.    Subsequently, we establish the existence, positivity  and uniqueness of the approximate problem. In addition, we verify the regularity of the solution to equation \eqref{approximation}. Finally, we derive the conservation laws, the entropy inequality, and the uniform-in-$n$ bounds associated with the approximate problem \eqref{approximation}.     
 
\begin{lemma}\label{Lemma 4.8}
	Assume that $f_0  \in L^1_2 ( \R^3_x \times \R^3_v )_+ $   satisfies \eqref{f0-p}, i.e.,
	\begin{equation*}
		\iint_{\R^3 _x \times \R^3_v } f_{ 0 } (x , v  )   (1+|x|^2 +|v|^2 + |\log  f_{ 0 } |)  dxdv  <+ \infty .
	\end{equation*}
	Then, there exists a sequence $\{f^n_{ 0 }\}_{n\geq 1} \subset \mathscr{S} (\R^3_x \times \R^3_v )$ (where $\mathscr{S}$ denotes the Schwartz space over $\R^3_x \times \R^3_v $) and a constant $C > 0$ such that 
	\begin{equation}\label{Ap-fs0}
		\begin{aligned}
			& \lim \limits_{n\rightarrow \infty}  \iint_{\R^3_x \times \R^3_v } (1+|x|^2 +| v |^2) |f^n_{ 0 } - f_{ 0 }| d x d v =0 \,, \ f^n_{ 0 } \geq \tfrac{1}{n} \exp \left(- \tfrac{|x|^2+ |v|^2 }{2} \right) \,, \\
			& \iint_{\R^3_x \times \R^3_v } f^n_{ 0 }  (1+|x|^2 +|v|^2 + |\log f^n_{ 0 }|)d x d v   \leq C \,, \\
			& \lim \limits_{n\rightarrow \infty} \iint_{\R^3_x \times \R^3_v} f^n_{ 0 } \log f^n_{ 0 }  d x d v = \iint_{\R^3_x \times \R^3_v } f_{ 0 } \log  f_{ 0 } dxdv \,.
		\end{aligned}
	\end{equation}
\end{lemma}
 \begin{proof}
	The proof is almost the same as Lemmas 3.1 and 3.3 of \cite{LN-BSMP}. We omit the details here for
	simplicity.
\end{proof}
\begin{lemma}\label{L-Bn}
Consider a smooth  cutoff  function $\phi \in C^\infty _c \left( B_{\frac{4}{3}} \right)$ satisfying $ 0 \leq \phi \leq 1 $ and $\phi \equiv 1 $ on $B_{ \frac{3}{4}}$. We introduce the truncated angular kernel
	\begin{align}\label{bn}
		b_n ( \cos \theta  ) = b ( \cos \theta ) \left[ 1- \phi \left( n \sin \left( \frac{\theta }{2 } \right) \right) \right],
	\end{align}
	and define the collision kernel as
	\begin{align}\label{d-Bn}
	B_n ( z, \sigma ) = |z |^\gamma b_n ( \cos \theta ), \quad \text{where } z \cdot \sigma = |z| \cos \theta .
	\end{align} 
 Consequently, the non-negative term $B_{n } (z,\sigma )$  fulfills the properties listed below.
\begin{enumerate}
{\item $ B_{ n }  (z, \sigma ) =0$ whenever   $ \frac{z}{|z|} \cdot \sigma  \geq  1- \frac{2}{ n},$  and the limit $ \lim \limits_{n\rightarrow \infty } B_{n }  (z,\sigma ) =B (z,\sigma ) $ holds for  all $(z,   \sigma  )\in \R^3 _z     \times \mathbb{S}^2 $  ;  }
{ \item The convergence results $ \lim \limits_{n\rightarrow \infty } M_{n }  ( |z | ) =M  ( |z | )$ and $ \lim \limits_{n\rightarrow \infty } M'_{n }  (|z| ) =M ' ( |z| ) $ are valid for all   $z\in \R^3_z$. Here $M_n ( |z| ) $ is defined as $M( |z|)$ with its collision kernel replaced by $B_n$, and likewise $ M'_n ( |z|)$ is defined as $M'( |z|)$ with its collision kernel replaced by $B_n.$  } 
\end{enumerate} 
\end{lemma}
\begin{remark}
The angular cutoff    $b_n ( \cos \theta )$ can be found in  \cite{HYZ-collision-2021, HZ-collision-2022}.
\end{remark}
\begin{proof}
	 Based on the definition of the smooth function $\phi $,   we  conclude that if $\sin \left( \frac{\theta }{2 } \right) \leq \frac{3}{4n} $ or equivalently  $ \cos \left( \frac{\theta }{2 } \right)  \geq \frac{\sqrt{16n^2 - 9 }}{4n}$, then $ b_n ( \cos \theta )= 0. $ Subsequently, by applying the double-angle formula, we obtain that if 
\begin{equation*}
	\cos \theta \geq 1- \frac{9}{ 8 n^2 } \geq 1 - \frac{2}{n}, \quad b_n ( \cos \theta )= 0.
\end{equation*} 
Furthermore, for any fixed $\theta \in (0, \frac{\pi}{2})$, $ n \sin \left( \frac{\theta }{2 } \right)  \to +\infty$ as $n \to \infty$, which implies that 
\begin{align*}
\lim \limits_{  n \rightarrow \infty }	\phi \left( n \sin \left( \frac{\theta }{2 } \right)  \right) = 0  .
\end{align*} 
The definition of the approximate kernel $B_n$ in \eqref{d-Bn} implies that whenever $\frac{z}{|z|}  \cdot \sigma \geq 1 - \frac{2}{n}$, 
\begin{align*}
\lim\limits_{n\rightarrow \infty} B_{n}(z,\sigma) = B(z,\sigma),  \quad \  \text{ and }\ B_n ( z, \sigma ) \leq B( z, \sigma ) \text{ for all }  (z, \sigma) \in \R^3_z \times \mathbb{S}^2.
\end{align*}
Thus, by the Lebesgue's
Dominated Convergence Theorem (Lemma \ref{LCT}), we obtain that 
\begin{align*}
\lim\limits_{n\rightarrow \infty } M_n ( |z|) = \lim\limits_{n\rightarrow \infty } \int_{ \S ^{   2 } } B_n ( z, \sigma ) ( 1- k \cdot \sigma ) d \sigma = M(|z| ) , \quad \text{ for all } z \in \R^3_z.
\end{align*} 
Furthermore, this convergence implies that $M_n ( |z| ) \leq M ( |z| ) $. We can similarly analyze $$M'_n( |z|) = \int_{\S^2}  B'_n (|z| ) ( 1- k \cdot \sigma  ) d \sigma  ,$$ where 
\begin{align*}
	B_n'(z, \sigma) = \sup_{1 < \lambda \leq \sqrt{2}} \frac{|B_n(\lambda z, \sigma) - B_n(z, \sigma)|}{(\lambda - 1)|z|} = \sup _{1 < \lambda \leq \sqrt{2}} \frac{   | \lambda z |^\gamma  -|z |^\gamma      }{ ( \lambda -1 ) |z| } b_n ( \cos \theta ) \leq B' ( z, \sigma ) .
\end{align*} 
This implies that $M '_n ( |z| ) \leq M ' ( |z| )$. Applying the Lebesgue's Dominated Convergence Theorem (Lemma \ref{LCT}) again, we obtain 
\begin{align*}
\lim\limits_{n\rightarrow \infty } M'_n ( |z|) = \lim\limits_{n\rightarrow \infty } \int_{ \S ^{   2 } } B'_n ( z, \sigma ) ( 1- k \cdot \sigma ) d \sigma = M'(|z| ) , \quad \text{ for all } z \in \R^3_z.
\end{align*}
This completes the proof of Lemma \ref{L-Bn}.
\end{proof}

The above lemmas provide the basis for constructing the approximation equation. Let  
 \begin{align}\label{Q L}
 \tilde{Q}_n (f,f) = \left( 1+ \frac{1}{n } \int_{\R ^3_v } f d v \right)^{-1} Q_n (f,f)  ,
 \end{align}
 where
 \begin{align}
 	Q_n (f,f) = \iint_{\mathbb{R}^3 _{v_*}\times \S^2} B_n ( v-v_* , \sigma ) ( f' f'_* - f f_* ) d v_* d \sigma .
 \end{align}
 It is evident that the approximate collision operator $ \tilde{Q}_n (f,f)$ is a quasilinearization of $Q_n (f,f)$ and admits a linear upper bound in $L^1 \cap L^\infty$ as described in the following lemma.
 \begin{lemma}\label{1-inf}
There exists a positive constant $C_{ n}$ depending on $n$ such that for all $f , g  \in L^1 ( \R^3_v )$, one has 
 		\begin{equation}
 		\begin{aligned}
 			& \Vert \tilde{Q}_n (f,f)  \Vert_{L^1 ( \R^3 _v)}\leq C_n \Vert f \Vert_{L^1 ( \R^3_v ) }, \\
 			& \Vert \tilde{Q}_n (f,f)   -\tilde{Q}_n (g,g )  \Vert_{L^1( \R^3_v )} \leq C_n   \Vert f    -g    \Vert_{L^1( \R^3_v )}.
 		\end{aligned}
 	\end{equation} 
 Furthermore, if $f \in L^\infty \cap L^1 ( \R^3_v )$, one has 
 \begin{equation}
 	 \Vert \tilde{Q}_n (f,f)  \Vert_{L^\infty ( \R^3 _v)}\leq C_n \Vert f \Vert_{L^\infty ( \R^3_v ) }.
 \end{equation}
 \end{lemma}
 \begin{proof}
  The proof is almost the same as Lemmas 3.7-3.8 of \cite{LN-BSMP}. We omit the details here for
 simplicity.
 \end{proof}
Furthermore, by employing arguments similar to those in the Lemma 3.9 of \cite{LN-BSMP}, 
we establish the existence and uniqueness of the non-negative smooth  solution $f^n$ to the approximate problem \eqref{approximation}. Specifically, we have the following lemma.
\begin{lemma}[Existence, Uniqueness and Positivity]\label{Lemma 4.14}
	Consider the approximated initial data $f_0^n$ constructed in Lemma \ref{Lemma 4.8} and the approximated collision operator $\tilde{Q}_{n}  (f , f )$ constructed in \eqref{Q L}. Then the approximated problem \eqref{approximation} admits a unique distributional solution $f^n$ satisfying 
	\begin{enumerate}
		\item{ for any fixed $T>0$, $  f^n \in C([0,+\infty ); L^1(\R^3_x \times \R^3 _v))$ and
			\begin{equation*}
				f^n (t,x, v ) \geq \tfrac{1}{n} \exp [ -c_0 (|x|^2+|v |^2 ) ], \quad \textrm{for all } \ (t,x,v ) \in [0, T] \times \R^3_x \times \R^3_v ;
		\end{equation*}}
	 \item { For   any multi-index $ \boldsymbol{\zeta} = ( \zeta_0, \zeta_1, \cdots, \zeta_6 ) \in \mathbb{N}^7$ with $| \boldsymbol{\zeta} | = \sum_{i=0}^{6} \zeta_i \leq m$, we have 
	 		$$D^{ \boldsymbol{\zeta} } f^n  \in L^\infty ((0,T) \times \R^3 _x\times \R^3_v ) \cap L^\infty ( (0, T); L^1 ( \R^3_x \times \R^3_v ; (1 + |x|^k + | v |^k) d x d v ) ) \,, $$
	 		and 
	 		\begin{equation*}
	 			\sup_{ (t , x , v ) \in (0,T) \times \R^3 _x\times \R^3_v } | D^{ \boldsymbol{\zeta} } f^n  | + \sup_{t \in (0, T)} \iint_{\R^3 _x\times \R^3_v } (1+| x |^k+| v |^k ) | D^{ \boldsymbol{\zeta} } f^n |d x d v \leq C \,.
	 	\end{equation*}   	
 		Here the constant $c_0 = c_0 (T)>0 $. }
	\end{enumerate} 
\end{lemma}
\begin{lemma}\label{Lemma 4.17}
	Assume that $f^n$ is the distributional solution to the approximated problem \eqref{approximation} constructed in Lemma \ref{Lemma 4.14}. Then it further enjoys the following properties:
	\begin{enumerate}
		\item {Mass, momentum, energy and angular momentum are all conserved. Hence, for all $t \geq 0$, we have
			\begin{equation*}
				\iint_{\R^3 _x\times \R^3_v } f^n ( t,\cdot )
				\left (\begin{array}{c}
					1 \\[2mm]
					v  \\ [2mm]
					|v |^2 \\ [2mm]
					|x-t v |^2 \\[2mm]
				\end{array}\right ) dxdv  
				= \iint_{\R^3 _x\times \R^3_v } f_0^n 
				\left (\begin{array}{c}
					1 \\[2mm]
					v \\ [2mm]
					|v |^2 \\ [2mm]
					|x|^2 \\[2mm]
				\end{array}\right )dxdv  ;
			\end{equation*} }\label{1 of Lemma 4.17}
		\item{The entropy identity holds.  Hence for all $t \geq 0$,
			\begin{equation}\label{entro}
				\begin{aligned}
					H(f^n_{0})   = H(f^n)(t) 
				  +\int_{0}^{t} d \tau \iint_{ \R^3 _x\times \R^3_v } \tilde{e} ^n(\tau,x, v ) dxdv   ,
				\end{aligned}
			\end{equation}
			where the density of entropy dissipation is given by
			\begin{equation}\label{e na} 
				\begin{aligned} 
		\tilde{e} ^n ( t ,x , v )=& \frac{1}{4} \left( 1+ \frac{1}{n } \int_{\R ^3_v } f^n d v \right)^{-1}    \iint_{ \R^3 _{v_*}\times \S^2 } B_n ( v-v_* , \sigma ) \\
		& \times  ( { f^n}' {f_*^n}' - f^n f^n_* ) \log \left( \frac{ {f^n}'  {f_*^n}' }{f^n f^n_*} \right) d v_* d \sigma,  
		\end{aligned}
	\end{equation}
and	the $H$-functional $H(f)$ is defined as in \eqref{Hf}; }\label{2 of Lemma 4.17}
		
		\item{For all   $t \geq 0$, we have
			\begin{equation}\label{FB2}{\small
					\begin{aligned}
						\iint_{ \R^3 _x\times \R^3_v } f^n  ( t , x , v )(1+|x|^2+| v |^2)d x d v \leq & \iint_{\R^3 _x\times \R^3_v } f^n_{ 0}(x,v)(1+2|x|^2+(2t^2+1)|v|^2)d x d v,
				\end{aligned}}
			\end{equation}
			and
			\begin{equation}\label{FWC2}
				\begin{aligned}
					\sup \limits_{t\geq 0}&\iint_{ \R^3 _x\times \R^3_v } f^n ( t , x , v ) |\log f^n | d  x d v + \frac{1}{4}\int_{0}^{+\infty} dt \iint_{ \R^3 _x\times \R^3_v } \tilde{e}^n  d x d v \\
					\leq & 2 \iint_{ \R^3 _x\times \R^3_v } f^n_{ 0}(x,v ) (|\log f^n_0|+ |x|^2+ | v | ^2 )d x d  v + C_1,
				\end{aligned}
			\end{equation}		 
			where the constant $C_1>0$ is independent of $n$.}\label{3 of Lemma 4.17}
	\end{enumerate}
\end{lemma}
 \begin{proof}
	The proof is almost the same as Lemma 3.11 of \cite{LN-BSMP}. We omit the details here for
	simplicity.
\end{proof}
\section{integrability of the renormalized collision operator}\label{Sec-integrable}
In this paper, we consider 
\begin{align}\label{beta-delta}
\beta_\delta ( s ) = \frac{s}{ 1+ \delta s },\  \text{ for any  }\ \delta >0.
\end{align}
    Obviously,  $\beta_\delta \in C^2 (\R_+, \R_+) $ satisfies condition \eqref{beta}. Moreover, for any fixed $\delta > 0$, $\beta_\delta$ is bounded. Consequently, for any non-negative functions $f$ and $f'$, we define
 \begin{align}\label{Gamma}
	\Gamma ( f, f' ) = \beta_\delta ( f' ) - \beta_\delta ( f ) - \beta_\delta ' ( f ) ( f' - f ).
\end{align}
Since $\beta_\delta$ is concave, it follows that $\Gamma(f, f') \leq 0$. Following the argument in Section 3 of \cite{Alex-Vill-2002-CPAM}, for almost all $( t,x , v)$, the renormalized collision operator $\beta_\delta'(f) Q(f, f)$ can be decomposed as
\begin{align*}
	\beta_\delta ' (f ) Q ( f,f ) =  (\mathcal{R}_1) + (\mathcal{R}_2) +  (\mathcal{R}_3), 
\end{align*}
where
\begin{equation}\label{R123}
\begin{aligned}
	(\mathcal{R}_1)  & = \left[ f \beta_\delta' ( f )  - \beta_\delta ( f ) \right] \iint_{ \mathbb{R}^3_{v_*}  \times \S^2 }  B(f'_* - f_*)  dv_* \, d\sigma \,, \\
	(\mathcal{R}_2) &= \iint_{\mathbb{R}^3_{v_*} \times \S^2 }  B \left[f'_*\beta_\delta(f') - f_*\beta_\delta(f)\right]dv_* \, d\sigma \,, \\
	(\mathcal{R}_3) &= -\iint_{\mathbb{R}^3_{v_*} \times \S^{2}}  B f'_* \Gamma(f, f')dv_* \, d\sigma \,.
\end{aligned}
\end{equation}
In particular, we define
\begin{equation}
	\mathscr{S}  f \equiv \iint_{\mathbb{R}^3_{v_*} \times \S^2 }  B(v - v_*, \sigma) (f'_* - f_*) dv_* \, d\sigma .
\end{equation}
Thus, we have 
\begin{align*}	(\mathcal{R}_1 )= \left[ f \beta_\delta ' ( f ) - \beta_\delta ( f  ) \right] \mathscr{S}  f .
\end{align*}

\begin{remark}
Since we consider the non-cutoff Boltzmann equation, the classical DiPerna Lions theory based on Grad's angular cutoff assumption cannot be applied. Specifically, we cannot follow the approach in \cite{Diperna-Lions} to obtain local integrability of the gain and loss terms of the collision operator by using Lemmas \ref{theo-6.2.1}-\ref{theo-6.4.1}(averaged velocity lemma) and the Arkeryd-type inequality.

To overcome this difficulty, we adopt the renormalization method introduced by Alexandre and Villani in Section 3 of \cite{Alex-Vill-2002-CPAM}. In this approach, the renormalized collision term $\beta'_\delta (f) Q(f,f)$ is decomposed into three parts, denoted by $(\mathcal{R}_1)$, $(\mathcal{R}_2)$ and $(\mathcal{R}_3)$. The main strategy is to first show that $(\mathcal{R}_1)$ and $(\mathcal{R}_2)$ are well defined as distributions in $\mathcal{D}'([0,T] \times \mathbb{R}^3_x \times \mathbb{R}^3_v )$. Then, by using the renormalized Boltzmann equation together with the non-negativity of these terms, we derive suitable bounds for $(\mathcal{R}_3)$, which allows us to give a rigorous meaning to the renormalized collision operator.
\end{remark}

\subsection{Cancellation lemma}

The main purpose of this section is to analyze the integral term $\iint_{\mathbb{R}^3_{v_*} \times \mathbb{S}^2} B(f'_* - f_*) \,dv_* \,d\sigma$.
Since the collision kernel possesses a singularity at the angle $ \theta = 0 $, we cannot directly show that this integral is finite. However, we observe that $v'_* \simeq v_*$ holds when the angle $\theta$ is sufficiently close to $0$.  This observation leads us to expect that the integral is still convergent. The purpose of the Cancellation lemma is to justify this fact.
The idea of the Cancellation lemma first appeared in the work of Villani \cite{Villani-regular-1999}. It was later further developed in \cite{ADVW-bound-2000, Alex-Vill-2002-CPAM}. In this work, we mainly follow the formulation given in \cite{Alex-Vill-2002-CPAM}.

\begin{lemma}[Cancellation lemma, \cite{Alex-Vill-2002-CPAM}]\label{prop-S}
	Let $-3 < \gamma \leq 1$. For  the collision kernel $ B $ and for almost everywhere $ v \in \mathbb{R}^3_v $, we have the following conclusion:
	\begin{align}\label{Sf}
		\mathscr{S}   f \equiv \iint_{\mathbb{R}^3_{v_*}  \times \S^2 } \, B (v - v_*, \sigma)(f'_* - f_*) dv_* d\sigma  =\int_{\mathbb{R}^3_{v_*}} f _*   S (|v-v_*|) dv_*,
	\end{align}
	where the kernel $S$ is given by
	\begin{align}\label{S}
		S  (|z|) = 2\pi  \int_0^{\frac{\pi}{2}}\sin   \theta \left[ \frac{1}{\cos^3 \left(\frac{\theta}{2}\right)} B  \left( \frac{|z|}{\cos \left(\frac{\theta}{2}\right)}, \cos \theta \right) - B  (|z|, \cos \theta) \right] d\theta .
	\end{align}	
	In particular, $ S $ is a locally integrable function satisfying the estimate 
	\begin{align*}
		|S (|z|)|  \leq \frac{1  }{\cos^2 \left( \frac{\pi}{8} \right)} \left[ 3 M (|z|) + |z| M '(|z|)\right].
	\end{align*}
\end{lemma}
 \begin{proof}
  	The proof is almost the same as Proposition 3.1 of \cite{Alex-Vill-2002-CPAM}. We omit the details here for
  simplicity.
 \end{proof}
 
\begin{lemma}\label{R1}
Assume that  $f$ satisfies 
	\begin{align*}
		\sup_{ t \in [0, T]} \iint_{\mathbb{R}^3 _x \times \mathbb{R}^3_v } f ( t ,x , v) \left( 1+ |v|^2 \right) d x d v < + \infty .
	\end{align*}
If $-3 < \gamma \leq 1$,	then it follows that $ ( \mathcal{R}_1 ) \in L^\infty \left( [0,T ] ; L^1 \left( \R ^3_x \times B_R ( v) \right)\right), $ for any $R >0.$
\end{lemma}

\begin{proof}
	
	We first recall the definition of $\beta_\delta (s) $ in \eqref{beta-delta}, which  implies that both $f \beta'_\delta(f)$ and $\beta_\delta(f)$ belong to $L^\infty([0,T]\times \mathbb{R}^3_x \times \R^3_v ))$ for any fixed $\delta>0$. Thus, to complete the proof of the lemma, it suffices to show that
$$
	f *_v S \in L^\infty\big([0,T]; L^1(\mathbb{R}^3_x \times B_R(v))\big).
$$
	Next, using the expression for the collision kernel $B$ 	given in \eqref{B-gamma}, we observe that there exists a constant $C_\nu$ such that 
$$
	S(z) \le C\big(M(|z|) + |z| M'(|z|)\big) \le C_\nu |z|^\gamma .
$$
If $0 \le \gamma \leq 1$, we obtain
\begin{align*}
	\|f * S\|_{L^1(\mathbb{R}_x^3 \times B_R(v))} 
\leq &C_\nu \iiint_{ \R ^3 _x \times \R ^3_{v_*}  \times B_R(v) } f ( t ,x , v_* )   |v-v_*|^\gamma    dx dv dv_* \\
	\leq &C_\nu \iiint_{ \R ^3 _x \times \R ^3_{v_*}  \times B_R (v)} f ( t ,x , v_* ) (  |v-v_*|^2  +1                                                                                                                                                                                                                                           ) dx dv dv_* \\
	\leq & C_{\nu  ,R} \iint_{ \R ^3 _x\times \R ^3 _v  } f ( t ,x , v_* )( 1+ | v_*|^2  )   dx  dv_*  .
\end{align*}
If $-3 < \gamma < 0 $, one has
\begin{align*}
		\|f * S\|_{L^1(\mathbb{R}_x^3 \times B_R(v))}  
 \leq & C_\nu \iiint_{ \R ^3 _x\times \R ^3 _v \times B_R(v) } f ( t ,x , v_* )   |v-v_*|^\gamma    dx dv dv_* \\
  \leq &  C_\nu \iint_{ \R ^3_x \times \R ^3 _v  } f ( t ,x , v_* )     dx  dv_* \int_{B_{2R+1}(z)}  |z|^\gamma d z  +  C_{\nu, R} \iint_{ \R ^3_x \times \R ^3 _{v_*}  } f ( t ,x , v_* )     dx  dv_*  .
		\end{align*}  
Based on the local integrability of the term $|z|^{\gamma} $ with $\gamma \in (-3, 0)$, we conclude that
$f *_v S \in L^1 \left( \R ^3_x  \times B_R (v) \right)  $, which means $ ( \mathcal{R}_1 ) \in L^\infty \left( [0,T ] ; L^1 \left( \R ^3  _x \times B_R  (v) \right)\right)  $  for all  $  \gamma \in (-3,1] $.
\end{proof}

\subsection{The integrability of $(\mathcal{R}_2)$}

When dealing with the key bilinear term $(\mathcal{R}_2)$, our method differs from that presented in Section 3.2 of  \cite{Alex-Vill-2002-CPAM}. In their work, Alexandre and Villani used a second-order Taylor expansion to control the strong angular singularity and obtained boundedness of an operator from $W^{2,\infty}$ to $L^\infty$ in a dual sense. As a result, the upper bound  in inequality \eqref{T-var} reduces to $\frac{1}{2} |v-v_*| \left( 1+ \frac{|v-v_*| }{2} \right) M( |v-v_*|)$.  However, for hard potentials ($0\leq \gamma \leq 1$), the term $ |v-v_*|^2M( |v-v_*|)  $  cannot be controlled by  the available a priori bounds on mass and energy, i.e.,
$$
\iint_{\R^3_x  \times \R^3_v  } f (1 + |x|^2 + |v|^2) dx  dv .
$$
 
 In fact, we have verified that a first-order Taylor expansion is sufficient to control the singularity induced by the collision kernel. In this case, choosing test functions  
 \begin{equation*} 
 	\varphi \in L^\infty \left( [0,T] \times \mathbb{R}^3_x ; W^{1,\infty}_0(\mathbb{R}^3_v)\right) 
 \end{equation*} 
 ensures that the term $(\mathcal{R}_2)$ is well-defined. Here, the space $W^{1,\infty}_0(\mathbb{R}^3)$ is defined as the closure of $C_c^\infty(\mathbb{R}^3)$ in $W^{1,\infty}(\mathbb{R}^3)$. However, in the critical case $\gamma = 1$, the linear term $|v-v_*|$ arising from the Taylor expansion couples with the factor $|v-v_*|^\gamma$ in the collision kernel, leading to the breakdown of convergence for the renormalized operator $(\mathcal{R}_2)^n$. A detailed analysis of this issue is provided in Section \ref{Subsec:Cnv-R2}.

To solve this difficulty, we use test functions from the H\"older space $C_c^\alpha(\mathbb{R}^3_v)$ instead of $W^{1,\infty}_0(\mathbb{R}^3_v)$. We set the parameter $\alpha$ between $\nu$ and $1$, and the detailed derivation is given in Lemma~\ref{R2}. Although this change relaxes the regularity requirements for the test functions, it actually leads to a more regular dual space from the perspective of duality. Specifically, the dual space $B^{-\alpha}_{1,1}$ is more regular than $(W^{1,\infty}_0)^*$ because the index $-\alpha$ is strictly larger than $-1$ when $\alpha < 1$. This more refined estimate effectively compensates for the lack of Lipschitz continuity and ensures that the operator converges effectively even in the critical case.

 To this end, let  $\varphi ( v ) $ be a  test function depending only on the velocity variable. Then we have
\begin{align*} 
	\int_{\R ^3 _v } (\mathcal{R}_2) \varphi(v)   dv 
	&= \iiint_{ \mathbb{R}^{3} _v \times  \mathbb{R}^{3}_{v_*}   \times \S^{2} } 
	B \left[ f' _* \beta_\delta  ( f ' ) - f _* \beta_\delta (f) \right] \varphi(v) dv   dv_*  d\sigma \,  \\
	&= \iint_{ \mathbb{R}^{3} _v \times  \mathbb{R}^{3}_{v_*}  } f_* \beta_\delta (f)  dv  dv_*   \int_{\S^{2}} B ( v -  v _* , \sigma ) ( \varphi' - \varphi ) d\sigma . 
\end{align*}
For a fixed $v_* $, we define the operator $\mathcal{T}$ as 
\begin{align}\label{T}
	\mathcal{T} \colon \varphi \longmapsto 
	\int_{\S^{2}} B ( v  - v  _* , \sigma) (\varphi' - \varphi) \, d\sigma.
\end{align}

We shall show that $\mathcal{T}$ is essentially bounded from $C_c^\alpha$ to $L^\infty_{loc}$ for $\alpha \in (\nu, 1)$, which differs from Proposition 3.3 of \cite{Alex-Vill-2002-CPAM}.

\begin{lemma}[$C_c^\alpha \to L^\infty_{loc}$ bound for $\mathcal{T}$]\label{Lmm-T}
  Let $- 3 < \gamma \leq 1$ and $\alpha \in (\nu, 1)$ with $\nu$ defined in \eqref{nu}. Then for all $\varphi \in C_c^\alpha (\R_v^3)$,
  \begin{equation}\label{T-M}
    \begin{aligned}
      | \mathcal{T} \varphi (v) | \leq C_{\nu, \alpha} \| \varphi \|_{C_c^\alpha} |v - v_*|^{\alpha + \gamma} \leq C'_{\nu, \alpha} \| \varphi \|_{C_c^\alpha} |v - v_*|^{\alpha} M (|v - v_*|)
    \end{aligned}
  \end{equation}
  for some constant $C_{\nu, \alpha}, C'_{\nu, \alpha} > 0$.
\end{lemma}

\begin{proof}
  In view of the definition \eqref{B-gamma} of the collision kernel $B$ and the range of parameters, we observe that for $\nu < \alpha < 1$, the integral satisfies 
 \begin{align*} 
 	\int_{0}^{\frac{\pi}{2} } b ( \cos \theta) \sin ^\alpha \left(\frac{\theta}{2}\right) \sin \theta d \theta < +\infty. 
 \end{align*}  
Moreover, we observe that
 \begin{align*}
 	|v -v'|^2 = \frac{1}{2} |v-v_*|^2 ( 1- k \cdot \sigma )= |v-v_*|^2 \sin ^2 \left(\frac{\theta}{2}\right).
 \end{align*}
 Consequently, we obtain the estimate 
 \begin{equation}\label{T-var} 
 	 \begin{aligned} 
 		\left| \int_{ \S ^{ 2 } } B ( v-v_* , \sigma ) ( \varphi ' - \varphi ) d \sigma \right| 
 		&\leq \| \varphi \|_{C_c^\alpha} \int_{ \S ^{ 2 } } B ( v-v_* , \sigma ) |v' -v |^\alpha d \sigma \\ 
 		&\leq 2 \pi \| \varphi \|_{C_c^\alpha} |v-v_*|^{\alpha} \int_{0}^{\frac{\pi}{2} } B ( |v-v_*|, \cos \theta) \sin ^\alpha \left(\frac{\theta}{2}\right) \sin \theta d \theta \\ 
 		&\leq C_{\nu ,\alpha} \| \varphi \|_{C_c^\alpha} |v-v_*|^{\alpha+\gamma} . 
 	\end{aligned} 
 \end{equation}
This implies
\begin{align*}
	| \mathcal{T} \varphi ( v) | =	\left| \int_{ \S ^{ 2 } } B ( v-v_* , \sigma ) ( \varphi ' - \varphi ) d \sigma \right|  \leq  C_{\nu ,\alpha} \| \varphi \|_{C_c^\alpha} |v-v_*|^{\alpha+\gamma} .
\end{align*}  
Together with \eqref{M-gamma}, the proof of Lemma \ref{Lmm-T} is completed.
\end{proof}
 
\begin{lemma}\label{R2}
If  $-3 < \gamma \leq 1$,	then it follows that $ (\mathcal{R}_2 ) \in L^\infty \left( [0,T ] ; L^1 \left( \R ^3_x ; B^{-\alpha}_{ 1,1} ( B_R ( v)) \right)\right) $, for any $R>0$ and any $\nu < \alpha <1$.
\end{lemma}
\begin{proof}
	
First, by the definition of the dual norm, we have
\begin{align*}
	\| (\mathcal{R}_2) \|_{B^{-\alpha}_{1,1} (B_R(v))}  
= & \sup \left\{ 
\int_{\mathbb{R}^3_v } (\mathcal{R}_2) \varphi \, dv \;\Bigg|\; 
\varphi \in C^{\alpha}_c (B_R(v)), \, 
\|\varphi\|_{ C^{\alpha}_c(B_R(v))} \leq 1  \right\} \\
	\leq & \iint_{ \mathbb{R}^3_{v_*}   \times B_R(v) } \beta_\delta (f) f_* | \mathcal{T} \varphi |  d v  d v _* ,
\end{align*}
 where $\alpha \in (\nu,1)$. In the case where $\gamma \in [-1,1] $, Lemma \ref{Lmm-T} yields
	\begin{align*}   
			\| (\mathcal{R}_2) \|_{B^{-\alpha}_{ 1,1} (B_R(v))}  
				\leq & C_{\nu ,\alpha}  \iint_{ \mathbb{R}^3_{v_*}   \times B_R(v) }  \beta_\delta (f) f_*   |  v-v_*  |^{ \alpha+ \gamma }  d v  d v_* \\  
		\leq & C_{\nu ,\alpha}  \iint_{ \mathbb{R}^3 _{v_*}  \times B_R(v) }  \beta_\delta (f) f_*  ( 1 +  |  v-v_*  |^2 )  d v  d v_* \\  
		\leq & C_{\nu, \alpha , R } \int_{ \R ^3_{v_*}  } f_* ( 1+ |v_*|^2 ) d v_*.
	\end{align*}
 On the other hand, if  $\gamma \in (-3,-1)$, we have
 \begin{align*}   
 	\| (\mathcal{R}_2) \|_{B^{-\alpha}_{ 1,1} (B_R(v))}  
 	\leq & C_{\nu ,\alpha}  \iint_{ \mathbb{R}^3 _{v_*}  \times B_R(v) }  \beta_\delta (f) f_*   |  v-v_*  |^{\alpha+ \gamma  }  d v  d v_* \\  
 	\leq &   C_{\nu ,\alpha} \int_{ \R ^3 _{v_*}   } f_*     dv_* \int_{B_{2R+1}(z)}  |z|^{\gamma +1} d z  +  C_{\nu,\alpha, R} \int_{  \R ^3  _{v_*}  } f_*      dv_*  .
 \end{align*}
In conclusion, $(\mathcal{R}_2 ) \in L^\infty \left( [0,T ] ; L^1 \left( \R ^3_x ; B^{-\alpha}_{ 1,1} (B_R(v))  \right)\right) $  for all  $  \gamma \in (-3,1] $.
\end{proof}
 
 \subsection{The integrability of $(\mathcal{R}_3)$}

\begin{lemma}\label{p-R3}
	Assume that $f$ satisfies 
	\begin{align*}
		\sup_{ t \in [0, T]} \iint_{\mathbb{R}^{3 }_x  \times \mathbb{R}^3_v } f ( t ,x , v) \left( 1+ |v|^2 \right) d x d v < + \infty.
	\end{align*}
If  $-3 < \gamma \leq 1$,	then it follows that $ ( \mathcal{R}_3 ) \in L^1 \left( [0,T ] \times \R ^3_x \times B_R ( v)  \right)  $, for any $ R >0.$
\end{lemma}
\begin{proof}
This lemma essentially serves as an a priori estimate, based on which we regard the function $f$ as a renormalized solution of the Boltzmann equation \eqref{Boltz} in the sense of Definition \ref{def-re}. Multiplying both sides of the following equation by a test function $\varphi$,
	\begin{align*}
	\frac{\partial \beta_\delta  ( f ) }{ \partial t } + v \cdot \nabla_x \beta_\delta ( f ) = \beta_\delta  ' ( f ) Q ( f , f ),
\end{align*}
and integrating over $t,x,v$,  we obtain
	\begin{equation*}
	\begin{aligned}
		\iint_{ \mathbb{R}^3 _x\times \mathbb{R}^3 _v} \beta_\delta ( f )( T, \cdot ) \varphi d x d v = &\iint_{ \mathbb{R}^3_x \times \mathbb{R}^3 _v} \beta_\delta ( f ) ( 0, \cdot ) \varphi d x d v \\
		& + \int_{0}^{T} d t \iint_{ \mathbb{R}^3_x \times \mathbb{R}^3 _v} \left(  (\mathcal{R}_1 ) + (\mathcal{R}_2 ) + (\mathcal{R}_3 ) \right) \varphi d x d v.
	\end{aligned}
\end{equation*}
Here, the test function $\varphi $ satisfies
\begin{align*}
	\varphi(v) \geq 0 ,\quad \varphi \equiv 1 \  \text{on} \  B_R(v),\quad \varphi \equiv 0  \  \text{on} \   \mathbb{R}^3 _v  \setminus B_{2R}(v) ,  \quad \|\varphi\|_{C^{\alpha}_c (\R^3_v)} \leq C.
\end{align*} 
Consequently,
	\begin{equation}\label{R3-bdd}
		\begin{aligned}
			\int_{0}^{T} d t \iint_{ \mathbb{R}^3_x \times\mathbb{R}^3_v } (\mathcal{R}_3 ) \varphi d x d v \leq &  \left|  \iint_{ \mathbb{R}^3_x \times \mathbb{R}^3_v }  (\beta_\delta ( f )( T,\cdot ) + \beta_\delta ( f ) ( 0, \cdot )) \varphi d x d v \right|  \\
			& + \left|  \int_{0}^{T} d t \iint_{ \mathbb{R}^3_x \times \mathbb{R}^3_v }  ( (\mathcal{R}_1 ) + (\mathcal{R}_2 )) \varphi dx dv  \right|.
		\end{aligned}
	\end{equation}
Given that $\beta_\delta ( f ) = f / ( 1+ \delta f) \leq C_\delta f$, we have
	\begin{equation*}
	\left|	\iint_{ \mathbb{R}^3 _x \times \mathbb{R}^3_v  } \left( \beta_\delta ( f ) ( T,\cdot ) - \beta_\delta ( f ) ( 0 ,\cdot ) \right) d x d v \right| \leq C_\delta \| f  \|_{L^1 ( \R ^3_x \times \R ^3_v )}.
	\end{equation*}
Lemmas \ref{R1} and \ref{R2}  ensure that the second term on the right-hand side of inequality \eqref{R3-bdd} is finite. 
Moreover, by the non-negativity of  $( \mathcal{R}_3 ) $, we conclude that 
$$ (\mathcal{R}_3 ) \in L^1 \left( [0,T] \times \R ^3 _x \times B_R(v )\right) \,.$$  
This completes the proof of Lemma \ref{p-R3}.
\end{proof}

\section{Strong compactness of approximated solution}\label{Sec:Con}
 Based on the uniform bounds \eqref{FB2} and \eqref{FWC2} of $f^n $ and the Dunford-Pettis theorem (Lemma \ref{theorem-dunford}), it follows directly that  $f^n$ converges weakly to $f$ in $L^1 ( [0,T] \times \R^3 _x \times \R^3_v )$. This section is devoted to upgrading the weak convergence results to strong convergence in $L^1$, specifically showing that $f^n$ converges strongly to $f$ in $L^1 ( [0,T]  \times \R^3 _x \times \R^3_v ) $. We first present several lemmas to establish an inequality which demonstrates that the high-frequency components of $\sqrt{f^n} $ in the velocity variable are controlled by the entropy dissipation.  By combining the weak compactness of $f^n$, the uniform bound \eqref{FB2} and the Vitali Convergence Theorem (Lemma \ref{L-VC}), we prove that $f^n$ converges strongly to $f$ in $L^1([0,T] \times \mathbb{R}^3 _x \times \mathbb{R}^3_v )$.  We note that this method yields strong convergence results only in the hard potential case. For soft potentials, We illustrate this limitation through a concrete counterexample (see Example \ref{soft-f} in Section \ref{sec-cnv-soft}), where we construct a specific distribution function that satisfies finite entropy and moment bounds but results in an infinite collision frequency.
\subsection{A functional inequality}\label{Subsec: inq}
 In this subsection, we establish a functional inequality for the non-cutoff model, demonstrating that the regularity of the distribution function with respect to the velocity variable is controlled by the entropy dissipation. This result is adapted primarily from \cite{ADVW-bound-2000}.

 Let 
 \begin{align}\label{d-en}
 	e^n ( t,x,v)   =  \iint_{  \mathbb{R}^{3}_{v_*} \times \S^{2}} B_n( f^n  f^n_* - {f^{n}} '  f_*^{ n}{'}  ) \log  f^n_*    \, dv_* \, d\sigma ,  
 \end{align}
 and
 \begin{align}\label{d-Zn}
 	Z_n(a)  =2\pi \int_{a}^{\frac{\pi}{2}} b_n (\cos\theta) \sin  \theta \, d\theta,
 \end{align}
 where $b_n$ is  defined in \eqref{bn}.

 We now proceed to prove the following inequality 
 \begin{equation}\label{F-Zn}
 \begin{aligned}
 	 &  	\int_{|\xi |\geq 1 } | \mathcal{F} (\sqrt{ f^n } \chi_R) |^2    Z_n\left( \frac{1}{|\xi | } \right)  d \xi   \\
 	 \leq & \frac{  C_0 C_2 J^2   (2\pi )^3   }{  \|f ^n \|_{L^1( B_R(v )) }  }   \left\{   \int_{  \mathbb{R}^{3}_v } e^n  ( t,x,v)   dv   + \frac{C_1 }{r_0 ^2 }  \|f ^n\|^2_{L^1_2 ( \R^3_v ) }  +\iiint_{ \mathbb{R}^{ 3}_v \times \mathbb{R}^{ 3} _{v_*}\times  \S^2 } B_n f^n ( {f^n_*}' -f^n_*)   d vdv_* d \sigma  \right\}   .
 \end{aligned}
 \end{equation} 
Here, $\mathcal{F}$ denotes the Fourier transform with respect to the velocity variable $v$, and  $\chi_R$ is a smooth cutoff function satisfying $0 \leq \chi_R \leq 1 $ with $\chi_R \equiv 1 $ on $B_R (v)$ and  $\supp(\chi_R) \subset B_{ R+1}(v)$. Furthermore, the term $C_0 C_2  J^2 $ is a constant, and $r_0, R > 0$ are sufficiently small and large constants, respectively. The proof is detailed in the following series of lemmas.

\begin{lemma}\label{lemma-Zn}
Assume that $\|f \|_{L^1 _1 ( \R^3_v  ) } = \int_{ \R ^3 _v } f( v ) ( 1+ |v| ) d v >0$. Then for all $\xi \in \mathbb{R}^3 $ such that if $|\xi| \geq 1$, we have
	\begin{align}\label{Zn}
			Z_n   \left( \frac{1}{|\xi  | } \right)   
	 \leq C_0 \|f\|^{-1}_{L^1(\mathbb{R}^3_v )}  
		\int_{ \S ^2 } b_n \left( \frac{\xi    }{| \xi  | } \cdot \sigma \right) \left( \hat{f } ( 0 ) - |\hat{ f } ( \xi  ^{- } )  | \right) d \sigma  , 
	\end{align}
	where $\xi^{- } = \frac{\xi - |\xi| \sigma}{2}$, $b_n$ is  defined in \eqref{bn}, and the constant $C_0 > 0$ is independent of $n$ and $f$.
\end{lemma}
\begin{remark}
This lemma is primarily based on Section 6 of \cite{ADVW-bound-2000}. However, the constant $C_0$   on the right-hand side of inequality  \eqref{Zn} originally  depends on $\| f\|_{L^1_1(\R^3_v ) }$ and  $\|f\|_{ L \log L } $. Here, we refine the constant by factoring out the  term $\|f\| _{L^1(\mathbb{R}^3_v )} $,  so that $C_0$ becomes independent of both $n$ and $f$. 
\end{remark}
\begin{proof}
	From the definition of $\xi^{- }$, it follows that
	\begin{align*}
		\left|\xi ^{- } \right| ^2 = \frac{1}{2} \left( |\xi |^2 - |\xi | \xi \cdot \sigma \right) = \frac{|\xi |^2 }{2 } \left( 1- \frac{\xi }{|\xi | } \cdot \sigma \right).
	\end{align*}
By using the half-angle identity $1-\cos \theta = 2 \sin ^2 \left(\theta /2\right)$, we obtain
	\begin{align*}
		\int_{ \S ^{2} }  b_n \left( \frac{\xi   }{| \xi | } \cdot \sigma \right) \left( |\xi ^{- } |^2 \wedge 1  \right) d \sigma = 2\pi  \int_{0}^{\frac{ \pi } { 2 } } b_n ( \cos \theta ) \left[ |\xi |^2 \sin ^ 2 \left( \frac{\theta }{2} \right) \wedge 1  \right]  \sin   \theta d \theta,
	\end{align*}
where we denote $a \wedge b = \min \{ a,b\}$.	Since $\sin x \geq \frac{2}{\pi} x$ for all $x \in \left[ 0, \frac{\pi}{2} \right]$, we have $|\xi|^2 \sin^2 \left( \frac{\theta}{2} \right) \geq |\xi|^2 \theta^2 / \pi^2$. Consequently, the following inequality holds 
	$$\left[ |\xi|^2 \theta^2 \wedge 1 \right] \leq \pi^2 \left[ |\xi|^2 \sin^2 \left( \frac{\theta}{2} \right) \wedge 1 \right] . 
	$$
This implies that for $|\xi| \geq 1$, we have
	\begin{align*}
		Z_n\left( \frac{1}{|\xi  | } \right)  = 2\pi  \int_{ \frac{1}{|\xi | } }^{\frac{\pi}{2}} b_n (\cos\theta) \sin  \theta \, d\theta &\leq  2\pi  \int_{0}^{\frac{ \pi } { 2 } } b_n ( \cos \theta ) \sin   \theta ( |\xi| ^ 2 \theta ^2 \wedge 1 ) d \theta \\
		& \leq \pi ^2 \int_{\S^2 } b_n \left( \frac{\xi   }{| \xi | } \cdot \sigma \right) \left( |\xi ^{- } | ^2 \wedge 1 \right) d \sigma.
	\end{align*}

	Next, we prove that for any $ \zeta  \in \R ^3_\zeta$, the following bound holds 
	$$ \hat{f}(0) - |\hat{f}( \zeta )| \geq C_{f}' \, (|\zeta |^{2} \wedge 1) .$$
	 We first observe that
	\begin{align*}
		\hat{f}(0) - |\hat{f}(\zeta )| = \int_{\mathbb{R}^3_\eta } f (\eta)(1 - \cos( \eta \cdot \zeta  + \alpha)) \, d\eta  
		= 2 \int_{\mathbb{R}^3_\eta } f (\eta) \sin^2 \left( \frac{\eta \cdot \zeta  + \alpha}{2} \right) \, d\eta,
	\end{align*}
	where $\hat{f}(\zeta )= |\hat{f} ( \zeta  )| e^{-i \alpha },$ with $\alpha \in ( -\pi , \pi ]$. Furthermore,   whenever $|\eta \cdot \zeta  + \alpha - 2 p \pi| \geq 2 \epsilon$ for any $p \in \mathbb{Z}$, the inequality $\sin^2 \left( \frac{\eta \cdot \zeta  + \alpha}{2} \right) \geq \sin^2 \epsilon$ holds. Consequently, we obtain
	\begin{align*}
		\hat{f}(0) - |\hat{f}(\zeta )| \geq  2 \sin^2 \epsilon \int_{\{|\eta| \leq r, \forall p \in \mathbb{Z}, |\eta \cdot \zeta  + \alpha - 2p \pi| \geq 2 \epsilon\}} f(\eta) \, d \eta .
	\end{align*}

By considering the decomposition of the space
	\begin{align*}\scalebox{0.95}{$\displaystyle
			\left\{ |\eta| \leq r, \forall p \in \mathbb{Z}, |\eta \cdot \zeta  + \alpha - 2p \pi| \geq 2 \epsilon   \right\}  \cup  \left\{ |\eta| \leq r, \exists p \in \mathbb{Z}, |\eta \cdot \zeta  + \alpha - 2p \pi| < 2 \epsilon   \right\} \cup  \left\{ | \eta | \geq r \right\} = \R ^3_\eta $}
	\end{align*}
	and the elementary tail estimate
	\begin{align*}
		\int_{ | \eta | \geq r  } | f( \eta ) |  d \eta \leq \frac{1}{r} \int_{\R ^3 _\eta } |f ( \eta ) | | \eta | d \eta ,
	\end{align*}
we derive the lower bound
	\begin{align*}
		\int_{\{|\eta| \leq r, \forall p \in \mathbb{Z}, |\eta \cdot \zeta  + \alpha - 2p \pi| \geq 2 \epsilon\}} f(\eta)   d \eta  
		\geq   \|f\|_{L^1(\mathbb{R}^3_\eta)} - \frac{\|f\|_{L^1_1(\mathbb{R}^3_\eta)}}{r} - \int_{ \{|\eta| \leq r, \exists p \in \mathbb{Z}, |\eta \cdot \zeta   +  \alpha  - 2 p \pi  \leq 2 \epsilon \}} f(\eta)   d\eta . 
	\end{align*}
	
We observe that for a fixed $p$, the condition
	\begin{align*}
		\left| \eta\cdot \frac{\zeta  }{|\zeta  | } + \frac{\alpha }{ | \zeta  | } - p \cdot \frac{2 \pi }{ |\zeta  | }  \right| \leq \frac{2 \epsilon }{ |\zeta  | }
	\end{align*}
	defines a thin slab region with width $\frac{4 \epsilon }{ |\zeta  | }$ and cross-sectional area $ 4r^2   $. 
	Next, we examine the range of $p$. Given that $ |\eta \cdot \zeta  + \alpha - 2p \pi| \leq 2 \epsilon $ and $|\eta \cdot \zeta   | \leq |\eta| \cdot  |\zeta  | \leq r |\zeta  |$, it follows that $|2 p \pi - \alpha | \leq 2 \epsilon + r |\zeta  |$. This inequality can be rewritten as 
	\begin{align*}
	  - 2 \epsilon - r |\zeta  |+ \alpha \leq 2 \pi p \leq 2 \epsilon + r |\zeta  | + \alpha,
	\end{align*}
	which implies 
	\begin{align*}
		p \in \left[ \frac{\alpha - 2 \epsilon }{2 \pi } - \frac{r |\zeta  | }{ 2 \pi } , \frac{ r |\zeta  | }{ 2 \pi }  +  \frac{\alpha + 2 \epsilon }{2 \pi } \right] \cap  \mathbb{Z} .
	\end{align*}
	
	For sufficiently small $\epsilon$, we have $ | \alpha \pm 2 \epsilon |  \leq  \pi$. This implies that the number of such integer values $p$ is about $\frac{r |\zeta |}{\pi} + 1$. Therefore, the total volume of the integration region 
	\begin{align*}
	A =	\left\{|\eta| \leq r, \exists p \in \mathbb{Z},  \left| \eta\cdot \frac{\zeta  }{|\zeta  | } + \frac{\alpha }{ | \zeta  | } - p \cdot \frac{2 \pi }{ |\zeta  | }  \right| \leq \frac{2 \epsilon }{ |\zeta  | } \right\}
	\end{align*}
	does not exceed that of region $ \frac{16 \epsilon }{|\zeta  | }   r ^{2} \left(\frac{r |\zeta  | }{\pi} + 1 \right)$. 
	
	Next, we discuss the value of the constant $C_f$ for the two cases where $|\zeta  |\geq 1 $ and $|\zeta  |\leq 1 $, respectively.
	
 \underline{\em  Case 1. $|\zeta  |\geq 1 $.} We note that the volume of the region $A$  satisfies
	\begin{align*}
		|A | \leq  \frac{16 \epsilon }{|\zeta  | }   r ^{2} \left( \frac{r |\zeta  |}{\pi } + 1 \right) \leq 16 \epsilon   r ^2 + \frac{16 \epsilon}{\pi }  r  ^3.
	\end{align*}
Let the constant $C'_f$ be defined as
	\begin{align*}
		C'_f = 2 \sin^2 \epsilon \left\{ \|f\|_{L^1(\mathbb{R}^3_\eta)} - \frac{\|f\|_{L^1_1(\mathbb{R}^3_\eta)}}{r} - \sup_{|A| \leq 16 \epsilon   r ^2 + \frac{16 \epsilon}{\pi }  r  ^3 } \int_A f(\eta) \, d\eta \right\}.
	\end{align*}
By choosing appropriate values of $\epsilon$ and $r$, we can ensure that $C'_f \geq  \sin^2 \epsilon \|f\|_{L^1(\mathbb{R}^3_\eta)}$.
	
 \underline{\em  Case 2. $|\zeta  |\leq 1 $.} Let $\delta = \frac{\epsilon }{|\zeta |}$. Then we have
	\begin{align*}
		\sin ^2 \epsilon = \frac{\sin ^ 2 ( \delta |\zeta |  )}{ \delta ^2 |\zeta |^2 } \delta ^2 |\zeta | ^2 \geq \inf_{|\zeta  | \leq 1 } \left| \frac{\sin ^2 ( \delta |\zeta |  )}{ \delta ^2 |\zeta |^2 } \right| \delta^2 |\zeta |^2.
	\end{align*}
Furthermore, we set
	\begin{align*}
		C'_f = 2 \delta^2 \inf_{|\zeta | \leq 1} \left| \frac{\sin^2 (\delta |\zeta |)}{\delta^2 |\zeta |^2} \right| \times \left\{ \|f\|_{L^1(\mathbb{R}^3_\eta )} - \frac{\|f\|_{L^1_1(\mathbb{R}^3_\eta )}}{r} - \sup_{|A| \leq 16 \delta  r ^2 \left(1+\frac{r}{\pi}\right)} \int_A f(\eta) \, d\eta \right\}.
	\end{align*}
Similarly, we can choose appropriate values of $\delta$ and $r$ such that 
$$C'_f \geq  \delta^2 \inf \limits_{|\zeta | \leq 1} \left| \frac{\sin^2 (\delta |\zeta |)}{\delta^2 |\zeta |^2} \right| \|f\|_{L^1(\mathbb{R}^3_\eta)} \,.$$ 	
Consequently, there exists a positive constant $C_0$, independent of $f$, such that 
\begin{align*}
 |\zeta |^{2} \wedge 1 \leq C_0 \|f\|^{-1}_{L^1(\mathbb{R}^3_\eta)} \left( \hat{f}(0) - |\hat{f}( \zeta )| \right) , \quad \forall \zeta  \in \R ^3_\zeta .
\end{align*}
  In particular, we have $| \xi^{- }  |^{2} \wedge 1 \leq C_0  \|f\|^{-1}_{L^1(\mathbb{R}^3)} \left(   \hat{f}(0) - |\hat{f}( \xi^{- } )|  \right) $. It follows that if $|\xi| \geq 1 $,  
	\begin{align*}
	Z_n   \left( \frac{1}{|\xi  | } \right)  
	&\leq \pi^2 	\int_{ \S ^2 } b_n \left( \frac{\xi    }{| \xi  | } \cdot \sigma \right) \left( |\xi ^{- } |^2 \wedge 1 \right) d \sigma \\
	&\leq C_0 \|f\|^{-1}_{L^1(\mathbb{R}^3_\eta)}  
		\int_{ \S ^2 } b_n \left( \frac{\xi    }{| \xi  | } \cdot \sigma \right) \left( \hat{f } ( 0 ) - |\hat{ f } ( \xi  ^{- } )  | \right) d \sigma  .
	\end{align*}
This completes the proof of Lemma \ref{lemma-Zn}. 
\end{proof}

\begin{lemma}\label{Lemma-F}
	For any functions $ f \in L^1 ( \R ^3 _v), F \in L^2 ( \R ^3_v )$, the following inequality holds:
	\begin{align*}
		& \iiint_{ \mathbb{R}^{3 }_v \times  \mathbb{R}^{3 }_{v_*} \times \S^2 } b(k \cdot \sigma) f_* (F' - F)^2 dv\, dv_*\, d\sigma \\
		\geq & \frac{1}{(2\pi)^3} \int_{\R ^3_\xi } | \hat{F}(\xi ) | ^2  \left\{ \int_{ \S ^2 } b\left( \frac{\xi}{|\xi | } \cdot \sigma \right) \left( \hat{ f } ( 0 ) - |\hat{ f } ( \xi ^- ) | \right) d \sigma  \right\} d \xi,
	\end{align*}
	where
	\begin{align*}
		\xi^+ = \frac{\xi + |\xi| \sigma}{2}, \quad \xi^- = \frac{\xi - |\xi| \sigma}{2}.
	\end{align*} 
\end{lemma}

\begin{proof}
First, we recall the Plancherel-type identity established in Property 1, Section 5 of \cite{ADVW-bound-2000}, which is expressed as
		\begin{equation}\label{Plancherel}
		\begin{aligned}
			&  \iiint_{ \mathbb{R}^{ 3} _v \times \mathbb{R}^{ 3} _{v_*} \times  \S^{2} } b(k \cdot \sigma) f_*\left(F^{\prime}-F\right)^2 d v d v_* d \sigma \\
			= &\frac{1}{(2 \pi)^3} \iint_{\mathbb{R}^3_\xi  \times \S^{2}} b \left( \frac{ \xi }{ | \xi | } \cdot \sigma \right) \left[ \hat{ f } (0 ) | \hat{ F } ( \xi ) |^2 + \hat{ f } ( 0 ) | \hat{ F } \left( \xi ^{ + } \right) |^2 \right.  \\
			& \left.\qquad \qquad  - \hat{ f } \left( \xi ^{ -} \right) \hat{F} \left( \xi ^{+}\right) \overline{ \hat{ F } } ( \xi )- \overline{ \hat{ f } } \left( \xi ^{ - }\right) \overline{\hat{ F } } \left( \xi ^{ + } \right) \hat{ F }( \xi ) \right] d \xi d \sigma.
		\end{aligned}
	\end{equation}
The proof of this identity is omitted here. By observing the inequality
	\begin{align*}
		\hat{f} ( \xi ^- ) \hat{ F } ( \xi ^+ ) \overline{ \hat{ F } ( \xi ) } + \overline{ \hat{ f } ( \xi ^- ) }  \overline{ \hat{ F } ( \xi ^+  ) } \hat{ F } ( \xi ) \leq | \hat{ f } ( \xi ^- ) | \left( | \hat{ F } ( \xi ^+ ) |^2 + |\hat{ F } ( \xi ) | ^2 \right)  ,
	\end{align*}
we deduce that
	\begin{align*}
		&\hat{ f } ( 0 ) |\hat{ F } ( \xi ) |^2 + \hat{ f } ( 0 ) |\hat{ F } ( \xi ^+ ) |^2 - \hat{f}( \xi ^- ) \hat{ F } ( \xi ^+) \overline{ \hat{ F } ( \xi   ) } - \overline{ \hat{ f } ( \xi ^- ) }   \overline{ \hat{ F } ( \xi ^+  ) } \hat{ F } ( \xi  )\\
		\geq & \left( \hat{ f } ( 0 ) - |\hat{ f } ( \xi ^- ) | \right) \left( |\hat{ F } ( \xi ^+ ) |^2 + |\hat{ F } (\xi ) | ^2 \right) \geq \left( \hat{f } ( 0 ) - |\hat{ f } ( \xi ^- ) | \right)|\hat{ F } ( \xi ) |^2.
	\end{align*}
	Together with  \eqref{Plancherel}, one derives 
	\begin{align*}
		&  \iiint_{ \mathbb{R}^{ 3}_v \times \mathbb{R}^{ 3}_{v_*} \times  \S^{2} } b(k \cdot \sigma) f_*\left(F^{\prime}-F\right)^2 d v d v_* d \sigma \\
	= &\frac{1}{(2 \pi)^3} \iint_{\mathbb{R}^3_\xi \times \S^{2}} b \left( \frac{ \xi }{ | \xi | } \cdot \sigma \right) \left[ \hat{ f } (0 ) | \hat{ F } ( \xi ) |^2 + \hat{ f } ( 0 ) | \hat{ F } \left( \xi ^{ + } \right) |^2 \right.  \\
	& \left.\qquad \qquad  - \hat{ f } \left( \xi ^{ -} \right) \hat{F} \left( \xi ^{+}\right) \overline{ \hat{ F } } ( \xi )- \overline{ \hat{ f } } \left( \xi ^{ - }\right) \overline{\hat{ F } } \left( \xi ^{ + } \right) \hat{ F }( \xi ) \right] d \xi d \sigma \\
	\geq & \frac{1}{(2\pi)^3} \int_{\R ^3_\xi } | \hat{F}(\xi ) | ^2  \left\{ \int_{ \S ^2 } b\left( \frac{\xi}{|\xi | } \cdot \sigma \right) \left( \hat{ f } ( 0 ) - |\hat{ f } ( \xi ^- ) | \right) d \sigma  \right\} d \xi.
	\end{align*}
 	Therefore, the proof of Lemma \ref{Lemma-F} is finished. 
\end{proof}

\begin{lemma}\label{Lemma-Phi}
Let $r_0 $ be a small positive constant, and let $R > 0$ be a sufficiently large constant.  For any $v_j \in \mathbb{R}^3$ with $|v_j| < R$, we define the sets
\begin{align}\label{Aj-Bj}
	A_j  = \left\{ v \in \mathbb{R}^3_v : |v-v_j| \leq \frac{r_0}{4}, \ |v_j|<R \right\},  \quad
	B_j  = \left\{ v \in B_R(v) : |v-v_j| \geq r_0  \right\}.
\end{align}
 Then for all $\gamma \in (-3, 1]$, the following inequality holds:
	\begin{equation*}
	\begin{aligned} 
		& \iiint_{\mathbb{R}^{3}_v  \times \mathbb{R}^{3}_{v_*} \times \S^{2}} (b (k  \cdot \sigma )(f\chi_{B_j}) _* \left( (F \chi_{A_j} )' - (F \chi_{A_j} ) \right) ^2   dv   dv_* d\sigma \\
		\leq  & C_2 \iiint_{\mathbb{R}^{3}_v \times \mathbb{R}^{3}_{v_*} \times \S^2 } B (k \cdot \sigma )\left(\sqrt{f' }  - \sqrt{f }\right)^2 \, dv \, dv_* \, d\sigma +  \frac{C_1  C_2 }{r_0 ^2 }  \|f\|_{L^1_2 ( \R^3_v ) } ^2,
	\end{aligned}
	\end{equation*}
	where $F  = \sqrt{f ( v )}$,  $F ' = \sqrt{f (  v ' ) } $, $ 	C_2 = \min       \left\{  \left(\frac{r_0}{2}\right)^\gamma ,  (2\sqrt{2}R ) ^\gamma \right\} ^{-1} $, and $ \chi_{A_j}, \chi_{B_j}$ denote smooth cutoff functions supported on the sets $A_j$ and $B_j$, respectively.
\end{lemma}
\begin{remark}
	This lemma is mainly based on Lemma 2 in \cite{ADVW-bound-2000}. For the reader's convenience, we provide the proof here. 
\end{remark}
\begin{proof}
 It is evident that the inequality
	\begin{align*}
		f_* ( F' - F )^2 \geq (f\chi_{B_j}) _*   ( F' -F )^2 \chi_{A_j} ^2  
	\end{align*}
holds. Furthermore, the relation
	\begin{align*}
		(F ' \chi _{A_j} ' - F \chi_{A_j} )^2 = ( F' ( \chi_{A_j} ' - \chi_{A_j} ) + ( F' -F ) \chi_{A_j} )^2 \leq 2( F') ^2 ( \chi_{A_j} ' - \chi_{A_j} ) ^2 + 2 ( F' -F )^2 \chi_{A_j} ^2
	\end{align*}
implies that for any $\gamma \in [0,1]$, we have the lower bound
	\begin{align*}
	  |v-v_*|^\gamma   b( k \cdot \sigma ) f_* ( F' - F )^2 
		\geq & \min \{ 1, |v-v_*|^\gamma ) \} b ( k \cdot \sigma ) (f\chi_{B_j}) _* ( F ' - F )^2 \chi_{A_j} ^2 \\
		\geq & \frac{1}{2} \min \{ 1, |v-v_*|^\gamma \} b ( k \cdot \sigma ) (f\chi_{B_j}) _*   (F ' \chi _{A_j} ' - F \chi_{A_j} )^2 \\
		& - b ( k \cdot \sigma ) (f\chi_{B_j}) _*   ( F') ^2 ( \chi_{A_j} ' - \chi_{A_j} ) ^2,
	\end{align*}
	where we have used the estimate
	\begin{align*}
		( \chi_{A_j} ' - \chi_{A_j} ) ^2 \leq \| \nabla \chi_{A_j} \|^2 _{L^\infty } |v'-v |^2 = \| \nabla \chi_{A_j} \|^2 _{L^\infty } |v - v_* |^2 \sin ^2 \left( \frac{\theta}{ 2 } \right).
	\end{align*}
	By the definition of the set $A_j$, we assume the gradient bound $|\nabla \chi_{A_j}|^2_{L^\infty} \leq C_1 r_0^{-2}$. Consequently, estimating the   term $ ( F') ^2  ( \chi ' _{A_j} - \chi_{A_j} )^2 $  yields 
	\begin{align*}
		& \iiint_{ \mathbb{R}^{3}_v \times \mathbb{R}^{ 3 }_{v_*} \times  \S^{2} } (f\chi_{B_j}) _* b ( k \cdot \sigma ) f ' ( \chi ' _{A_j} - \chi_{A_j} )^2 dv dv_* d \sigma \\
		\leq & \frac{C_1 }{r_0 ^2 } \iint_{ \R ^3_v \times \R ^3_{v_*} } d v dv_* \int_{0}^{ \frac{\pi }{2 } } f ' f_* |v-v_*|^2 b ( \cos \theta ) \sin ^2 \left( \frac{\theta }{ 2 } \right) \sin   \theta d \theta  \\
		= &  \frac{C_1 }{r_0 ^2 }  \iint_{ \R ^3_{{v}'} \times \R ^3_{v_*} } d v' dv_* \int_{0}^{ \frac{\pi }{2 } } f ' f_* |v ' - v _*|^2  \cdot 4 \cos ^{ - 2 } \left( \frac{\theta }{ 2 } \right)   b ( \cos \theta ) \sin ^2 \left( \frac{\theta }{ 2 } \right)   \sin   \theta d \theta \\
		\leq &  \frac{C_1 }{r_0 ^2 } \int_{ \R ^3 _{{v}'} }  f' ( 1+ |v'|^2 ) d v' \int_{ \R ^3_{v_*}  } f_* ( 1+ |v_* | ^2 )  d v_* = \frac{C_1 }{r_0 ^2 } \Vert f \Vert^2_{L^1_2 ( \R^3_v ) }.
	\end{align*}
 We observe that the term $|v - v_*|^\gamma$ vanishes as the distance $|v - v_*|$ tends to 0 for $\gamma \in (0,1]$. To address this singularity, we construct the sets $A_j$ and $B_j$ such that they are initially separated by a distance of $3r_0/4$. Consequently, even after mollification, their supports remain separated by at least $r_0/2$. This separation ensures that if $v_* \in B_j$ and either $v \in A_j$ or $v' \in A_j$, we have the bounds
\begin{align*} 
	r_0/2 <  |v-v_*|  \leq \sqrt{2} |v'-v_* | \leq 2\sqrt{2} R. 
\end{align*}
Let us define the constant $C_2$ by $	C_2 = \left(   \min_{r_0 /2 \leq |z| \leq 2\sqrt{2}R } \{ 1, |z|^\gamma   \} \right)^{-1}$.
Since $r_0 > 0$ is sufficiently small and $R$ is sufficiently large, this constant simplifies to 
$$C_2 = \min \{ (r_0/2)^\gamma, (2\sqrt{2}R)^\gamma \}^{-1}\,. $$ 
Finally, combining these estimates leads to
	\begin{equation*}
		\begin{aligned} 
			& \iiint_{\mathbb{R}^{3}_v \times \mathbb{R}^{3}_{v_*} \times \S^{2}} (b (k  \cdot \sigma )(f\chi_{B_j}) _* \left( (F \chi_{A_j} )' - (F \chi_{A_j} ) \right) ^2   dv   dv_* d\sigma \\
			\leq  & C_2 \iiint_{\mathbb{R}^{3}_v \times \mathbb{R}^{3}_{v_*} \times \S^2 } B (k \cdot \sigma )\left(\sqrt{f' }  - \sqrt{f }\right)^2 \, dv \, dv_* \, d\sigma +  \frac{C_1  C_2}{r_0 ^2 }  \|f\|_{L^1_2 ( \R^3_v ) } ^2    .
		\end{aligned}
	\end{equation*}
	Therefore, the proof of Lemma \ref{Lemma-Phi} is finished.
\end{proof}
Based on the above lemmas, we are now in a position to provide the detailed proof of inequality \eqref{F-Zn}.
\begin{lemma}\label{L-FZn}
Let $\chi_R$ be a smooth cutoff function satisfying $0 \leq \chi_R \leq 1 $ with $\chi_R \equiv 1 $ on $B_R (v)$ and  $\supp(\chi_R) \subset B_{ R+1}(v)$. Then  for all $\gamma \in (-3, 1]$, the following inequality holds:
\begin{equation*} 
	\begin{aligned}
		&  	\int_{|\xi |\geq 1 } | \mathcal{F} (\sqrt{ f^n } \chi_R) |^2    Z_n\left( \frac{1}{|\xi | } \right)  d \xi   \\
		\leq & \frac{  C_0 C_2 J^2   (2\pi )^3   }{  \|f ^n \|_{ L^1( B_R(v ) ) }  }   \left\{   \int_{  \R^{3}_v } e^n  ( t,x,v)   dv   + \frac{C_1 }{r_0 ^2 }  \|f ^n\|^2_{L^1_2 ( \R^3_v ) } +  \iiint_{ \mathbb{R}^{ 3}_v \times \R^{3}_{v_*} \times  \S^2 } B_n f^n ( {f^n_*}' -f^n_*)   d vdv_* d \sigma \right\} .
	\end{aligned}
\end{equation*} 
Here, $\mathcal{F}$ denotes the Fourier transform with respect to the velocity variable $v$,  $C_0 C_2 J^2$ is a positive constant, while  $r_0, R > 0$ are sufficiently small and large constants, respectively. 
\end{lemma}
\begin{proof}
Applying Lemma \ref{lemma-Zn} to the function  $f ^n \chi_{B_j}$, where  $\chi_{B_j}$ denotes a smooth cutoff function supported on $B_j $ as defined in \eqref{Aj-Bj}, establishes that for  $|\xi| \geq 1$  the function $Z_n$ satisfies the estimate
\begin{align*} 
		Z_n   \left( \frac{1}{|\xi  | } \right)   
	\leq C_0 \|f ^n \chi_{B_j} \|^{-1}_{L^1(\mathbb{R}^3_v)}    \int_{ \S ^{ 2 } } b_n \left( \frac{\xi   }{| \xi | } \cdot \sigma \right) \left( \mathcal{F}  ({f }^n\chi_{B_j})   ( 0 ) - |  \mathcal{F}  ({f }^n\chi_{B_j} )  ( \xi ^{- } )  | \right) d \sigma,
\end{align*}
where we choose an appropriate $r_0$ such that for any $|v_0| < R$, the condition
\begin{equation*}
  \int_{|v-v_0|<r_0} f^n ( v ) d v \leq \frac{1}{2} \int_{ B_R(v) } f^n ( v ) d v 
\end{equation*} 
holds. Combined with the definition of the set $B_j$, this implies the bound $\|f^n \chi_{B_j} \|^{-1}_{L^1(\mathbb{R}^3_v)} \leq 2 \|f^n \|^{-1}_{L^1(B_R(v))}$. Consequently, for $|\xi| \geq 1$, we obtain the inequality   
\begin{align}\label{Zn-1}
	Z_n   \left( \frac{1}{|\xi  | } \right)   
	\leq C_0 \|f ^n \|^{-1}_{L^1( B_R(v)) }    \int_{ \S ^{ 2 } } b_n \left( \frac{\xi   }{| \xi | } \cdot \sigma \right) \left( \mathcal{F}  ({f }^n\chi_{B_j})   ( 0 ) - |  \mathcal{F}  ({f }^n\chi_{B_j} )  ( \xi ^{- } )  | \right) d \sigma.
\end{align}

Based on the definition of the set  $A_j$ in \eqref{Aj-Bj}, the finite covering theorem implies the existence of a finite collection of points $\{v_j\}_{j=1}^J$ such that $B_R \subset \cup_{j}^J A_j$. We then construct a partition of unity $\{\phi_j\}^J_{j=1}$ (for instance, by setting $\phi_j = \chi_{A_j} / \sum^J_{k=1} \chi_{A_k}$) satisfying the conditions  
\begin{align*}
\supp \phi_j   \subset B_R , \quad 0 \leq \phi_j \leq 1 , \quad  \sum^J_{j=1} \phi_j ( v )=1,\   \text{ for }v \in B_R.
\end{align*}
We decompose the function as $\sqrt{ f^n } \chi_R = \sum^J_{j=1} h_j$, where $h_j(v)= \sqrt{ f^n } \chi_R \phi_j$. The linearity of the Fourier transform and the Cauchy–Schwarz inequality imply 
\begin{align*}
 |\mathcal{F}( \sqrt{ f^n } \chi_R) ( \xi )  |^2 =  \left| \sum^J_{j=1}  \mathcal{F} h_j ( \xi )  \right|^2 \leq J \sum^J_{j=1} \left| \mathcal{F} h_j ( \xi )  \right|^2 \leq J \sum^J_{j=1} \left| \mathcal{F} (\sqrt{ f^n }   \chi_{A_j}) ( \xi )  \right|^2 ,
\end{align*}
where we have used the pointwise bound $0 \leq h_j ( v ) \leq \sqrt{ f^n } \chi_R \chi_{A_j} = \sqrt{ f^n } \chi_{A_j}$. Combining this result with Lemma \ref{Lemma-F} yields the estimate 
\begin{equation}\label{Zn-2} 
\begin{aligned}
		&	\int_{ \mathbb{R}^{3}_\xi  } |\mathcal{F}( \sqrt{ f^n } \chi_R) |^2 \int_{ \S ^{2 } } b_n \left( \frac{\xi   }{| \xi | } \cdot \sigma \right) \left(  \mathcal{F}  ({f }^n\chi_{B_j})  ( 0 ) - |  \mathcal{F}  ({f }^n\chi_{B_j})   ( \xi ^{- } )  | \right) d \sigma  d \xi \\
\leq 	&J \sum^J_{j=1}	\int_{ \mathbb{R}^{3}_\xi } \left| \mathcal{F} (\sqrt{ f^n }   \chi_{A_j})  \right|^2 \int_{ \S ^{2 } } b_n \left( \frac{\xi   }{| \xi | } \cdot \sigma \right) \left(  \mathcal{F}  ({f }^n\chi_{B_j})  ( 0 ) - |  \mathcal{F}  ({f }^n\chi_{B_j})   ( \xi ^{- } )  | \right) d \sigma  d \xi \\
	\leq & (2\pi )^3  J \sum^J_{j=1}  \iiint_{ \mathbb{R}^{3}_v \times  \mathbb{R}^{3}_{v_*} \times \S^{2 } } b_n \left( \frac{\xi   }{| \xi | } \cdot \sigma \right)    ({f }^n \chi_{B_j})_*  \left( (\sqrt{ f^n }   \chi_{A_j}){'}     - (\sqrt{ f^n }   \chi_{A_j}) \right)^2 dv \, dv_* \, d\sigma .
\end{aligned} 
\end{equation} 
Furthermore, Lemma \ref{Lemma-Phi} provides the upper bound
\begin{equation}\label{Zn-3}
	\begin{aligned} 
		&\iiint_{ \mathbb{R}^{3}_v \times  \mathbb{R}^{3}_{v_*} \times \S^{2} } b_n \left( \frac{\xi   }{| \xi | } \cdot \sigma \right)  ({f }^n \chi_{B_j})_*  \left( (\sqrt{ f^n }   \chi_{A_j}){'}     - (\sqrt{ f^n }   \chi_{A_j}) \right)^2 dv \, dv_* \, d\sigma  \\
		\leq & C_2  \iiint_{\mathbb{R}^{3}_v \times \mathbb{R}^{3}_{v_*} \times \S^{2}} B_n f^n_* (\sqrt{ f^{ n}{'} }  - \sqrt{f ^{n } })^2 \, dv \, dv_* \, d\sigma + \frac{C_1 C_2}{r_0 ^2 }  \|f ^n\|^2_{L^1_2 ( \R^3_v )}   .
	\end{aligned}
\end{equation}
Combining estimates \eqref{Zn-1}-\eqref{Zn-3}, we deduce that
\begin{align*} 
&	\int_{|\xi |\geq 1 } | \mathcal{F} (\sqrt{ f^n } \chi_R) |^2    Z_n\left( \frac{1}{|\xi | } \right)  d \xi  \\
\leq &  \frac{  C_0  C_2   J^2   (2\pi )^3   }{  \|f ^n \|_{L^1( B_R(v)) }   }    \left\{    \iiint_{\mathbb{R}^{3}_v  \times \mathbb{R}^{3}_{v_*} \times \S^{2}} B_n f^n  \left(\sqrt{f_*^{ n}{'} }  - \sqrt{f ^{n }_* } \right)^2 \, dv \, dv_* \, d\sigma + \frac{C_1  }{r_0 ^2 }  \|f ^n\|^2_{L^1_2 ( \R^3_v )}    \right\}.   
 \end{align*}
Applying the elementary inequality $(\sqrt{x} - \sqrt{y})^2 \leq x \log \frac{x}{y} - x + y$, we bound the first term as
\begin{align*}
	& \iiint_{\mathbb{R}^{3}_v \times \mathbb{R}^{3}_{v_*} \times \S^{2}} B_n f^n  \left(\sqrt{ f_*^{ n}{'}  }  - \sqrt{f ^{n }_* } \right)^2 \, dv \, dv_* \, d\sigma  \\
	\leq & \iiint_{\mathbb{R}^{3}_v \times \mathbb{R}^{3}_{v_*} \times \S^{2}} B_n f^n  f^n_*  \log \frac{f^n_*}{ {f_*^{n}} '  } dv \, dv_* \, d\sigma  
	+  \iiint_{\mathbb{R}^{3}_v \times \mathbb{R}^{3}_{v_*} \times \S^{2}} B_n f^n ( {f_* ^{n}} '   -  f^n_* ) dv \, dv_* \, d\sigma ,
\end{align*}
where we observe that
\begin{align*}
	  \iiint_{\mathbb{R}^{3}_v  \times \mathbb{R}^{3}_{v_*} \times \S^{2}} B_n f^n  f^n_*  \log \frac{f^n_*}{ f_*^{n' }  } dv \, dv_* \, d\sigma   
	=   \iiint_{\mathbb{R}^{3}_v \times \mathbb{R}^{3}_{v_*} \times \S^{2}} B_n( f^n  f^n_* - {f^{n}} '  f_*^{ n}{'}  ) \log  f^n_*   dv \, dv_* \, d\sigma.
\end{align*} 
Consequently, recalling the definition of $e^n$ in \eqref{d-en}, we arrive at the inequality
\begin{align*}
&  	\int_{|\xi |\geq 1 } | \mathcal{F} (\sqrt{ f^n } \chi_R) |^2    Z_n\left( \frac{1}{|\xi | } \right)  d \xi   \\
\leq & \frac{  C_0 C_2 J^2   (2\pi )^3   }{  \|f ^n \|_{L^1( B_R(v)) }   }  \left\{   \int_{  \mathbb{R}^{3}_v } e^n  ( t,x,v) \,  dv   +  \frac{C_1 }{r_0 ^2 }  \|f ^n\|^2_{L^1_2 ( \R^3_v ) } +\iiint_{ \mathbb{R}^{ 3}_v \times \mathbb{R}^{ 3}_{v_*}\times  \S^2 } B_n f^n ( {f^n_*}' -f^n_*)   d vdv_* d \sigma  \right\} .
\end{align*}
	Therefore, the proof of Lemma \ref{L-FZn} is finished. 
\end{proof}
\begin{corollary}
If $ 0 \leq  \gamma \leq 1$, then	
\begin{align*}
	& 	\int_{|\xi |\geq 1 } | \mathcal{F} (\sqrt{ f^n } \chi_R) |^2    Z_n\left( \frac{1}{|\xi | } \right)  d \xi   \\
	\leq &  \frac{  C_0  J^2   (2\pi )^3   }{  \|f ^n \|_{L^1( B_R(v)) } \left( r_0 /2 \right)^\gamma    }  \left\{   \int_{  \mathbb{R}^{3}_v } e^n  ( t,x,v) \,   dv   +   \left( C_{\nu }  + \frac{ C_1 }{ r ^2 _0 } \right)  \|f ^n\|^2_{L^1_2 ( \R^3_v ) }   \right\}.
\end{align*}
Here, $\mathcal{F}$ denotes the Fourier transform with respect to the velocity variable $v$, and $C_0 J^2$ represents a positive constant. Moreover, $r_0$ denotes a sufficiently small constant.
\end{corollary}
\begin{proof}
First, we obtain the inequality \eqref{F-Zn} in Lemma \ref{L-FZn},  then we only need to handle the term   $B_n f^n ( {f_* ^{n}} ' - f^n_* )$.  Replacing $S$ with $S^n$ in Lemma \ref{prop-S} yields $$  
S^n ( | z | ) \leq C  \left( M_n ( |z| )  + |z | M'_n ( | z | )  \right) .
$$

Moreover, according to the construction of $M_n$ in Lemma \ref{L-Bn}, there exists a constant $C_{\nu }$ independent of $n$ such that $
M_n ( |z| ) + |z| M'_n ( | z | ) \leq C_{\nu } |z|^\gamma $. Then for any $\gamma \in [0,1]$,
 \begin{equation*}
 	\begin{aligned}
 	 	\iiint_{\mathbb{R}^{3}_v \times \mathbb{R}^{3}_{v_*} \times \S^{2}} B_n f^n ( {f_* ^{n}} '   -  f^n_* ) dv \, dv_* \, d\sigma 
 		\leq & C_{\nu }	 \iint_{ \R ^3 _v\times \R ^3_{v_*}  }  f^n  f^n_*   |v-v_*|^\gamma d v d v_* \\ 
 		 \leq & C_{\nu }	    \iint_{ \R ^3_v \times \R ^3_{v_*}  }  f^n  f^n_* ( |v-v_*|^2 +1)  d v d v_* \\
 		\leq & C_{\nu }   \left(\int_{\R ^3 _v} f^n  (1+|v|^2)d v \right)^2    .
 	\end{aligned}
 \end{equation*}
 Substituting the above result into \eqref{F-Zn} shows that the conclusion holds.
\end{proof}
 \begin{remark} 
We remark that the upper bound established above is restricted to the range $\gamma \in [0,1]$. For the case of soft potentials where $-3 < \gamma < 0$, the estimate corresponding to \eqref{F-Zn} takes the form
$$
\frac{ C_0 (2\pi )^3 }{ \|f ^n \|_{L^1( B_R (v )) } (2 \sqrt{2 }R )^\gamma } \left\{ \iint_{\mathbb{R}^{3}_v \times \mathbb{R}^{3}_{v_*} } f^n f^n_* |v-v_*|^\gamma \, dv \, dv_* + \|e^n \|_{L^1( \mathbb{R}^3_v )} + C_1 \|f ^n\|^2_{L^1_2 ( \mathbb{R}^3_v ) } \right\} .
$$
When mass concentration couples with the singularity of the soft potential kernel $|v - v_*|^\gamma$, the collision integral term $\iint_{\mathbb{R}^{3}_v \times \mathbb{R}^{3}_{v_*} } f^n f^n_* |v-v_*|^\gamma \, dv \, dv_* $ may diverge to infinity. We illustrate through Example \ref{soft-f} that this case can indeed occur.
 \end{remark}

 \subsection{Convergence of approximated solution $f^n $}\label{Subsec:Cnv-f}
 In this subsection, we investigate the convergence of the positive solutions $f^n(t,x,v)$ to the approximated equation \eqref{approximation}, relying on the uniform estimates established in Lemma \ref{Lemma 4.17}.
 
 First, by the Averaging velocity lemma (Lemma \ref{theo-6.2.1}), we conclude that for any bounded sequence $\{\varphi^n\} \subset L^\infty( [0,T ] \times \mathbb{R}^3_x \times \mathbb{R}^3_v)$ converging to $\varphi$ almost everywhere, the sequence of integrals $\int_{\mathbb{R}^3_v} \beta_\delta(f^n) \varphi^n \, dv$ converges strongly in $L^1([0,T ] \times \mathbb{R}^3_x)$ to $\int_{\mathbb{R}^3_v} \beta_\delta \varphi \, dv$, where $\beta_\delta$ denotes the weak limit of $\beta_\delta(f^n)$ in $L^1([0,T ] \times \mathbb{R}^3_x \times \mathbb{R}^3_v)$. Consequently, combining the convergence properties of $\beta_\delta(f^n)$ and $f^n$, we deduce that for any $1 \leq p < \infty$, the convergence holds  
 \begin{equation}\label{fn-f}
 \int_{\mathbb{R}^3_v } f^n \varphi^n  dv \rightarrow  \int_{\mathbb{R}^3_v } f \varphi  dv \quad \text{ in } \ L^p(0,T; L^1(\mathbb{R}^3_x)).
 \end{equation}
For detailed proofs, we refer the reader to Section 5 of \cite{LN-BSMP}.

Furthermore, utilizing the uniform bounds \eqref{FB2}-\eqref{FWC2} on $f^n$ and the properties of $L^p$ spaces for $1 < p < \infty$, we infer that for any sufficiently large $R>0$, there exists a function $g \in L^2([0,T] \times B_R(x) \times B_R(v))$ such that the sequence satisfies
 \begin{align}\label{c-sqfn}
 	g^n = \sqrt{ f^n } \rightharpoonup g \quad \text{ in } \ L^2([0,T ] \times B_R (x)\times B_R(v)) .
 \end{align}
 \begin{lemma}\label{v-regular}
 For any $\epsilon>0$, let
\begin{align}\label{W-epsion}
	W_\epsilon = \left\{  (t,x ) \in [0,T] \times B_R(x): \int_{B_R(v)} g^2  d v >\epsilon \right\}.
\end{align}
We thus obtain the velocity regularity estimate, which is given by
\begin{align*}
	\lim_{A \to +\infty} \lim_{n \to \infty} \int_{W_\epsilon} dt \, dx \int_{|\xi| \geq A}|\mathcal{F} \sqrt{f_R^n}|^2   d\xi  = 0,
\end{align*} 
where $f_R^n = f^n \chi_R $. Here $\chi_R (v)$  is a smooth cutoff function supported in $B_{3R}(v) $ and equal to $1$ on $B_{2R}(v) $.
 \end{lemma}
 \begin{proof}
On the set $W_\epsilon$, the weak lower semicontinuity of the norm, combined with the $L^1$ convergence of $\int_{\mathbb{R}^3_v} f^n \, dv$ to $\int_{\mathbb{R}^3_v} f \, dv$, implies  
 \begin{align*}
 	\epsilon< \int_{B_R(v)} g^2 d v \leq \liminf \limits_{n\rightarrow \infty } \int_{B_R(v)} (g^n)^2 d v =\liminf \limits_{n\rightarrow \infty } \int_{B_R(v)}  f^n  d v = \int_{B_R(v)} f dv .
 \end{align*}
By Egorov's Theorem (Lemma \ref{egorov}), for any $\delta > 0$, there exists a Borel set $U_\delta$ with measure less than $\delta$ such that the uniform convergence
 \begin{align*}
 	\int_{B_R(v) } f^n \,  d v \rightrightarrows 	\int_{B_R(v) } f \ d v  
 \end{align*}
holds on $W_\epsilon \setminus U_\delta$. Consequently, for sufficiently large $n \in \mathbb{N}_+$, we obtain the lower bound $\int_{B_R(v) } f^n \, dv \geq \epsilon /2$ on the set $W_\epsilon \setminus U_\delta$.

Next, we define the functional $H^n$ and the set $V^n_L$ by 
\begin{align*} 
	H^n (t,x) = \int_{\mathbb{R}^3_v } f^n (1 + |v|^2 + \log f^n(t, x, v) ) \, dv , \quad V^n_L = \{(t, x) : H^n(t, x) > L \} .
\end{align*} 
By Chebyshev’s inequality, we obtain the measure estimate
\begin{align*} 
	\text{meas} ( V^n_L ) \leq \frac{1}{L} \int_{0}^{T} dt \int_{ \mathbb{R}^3_x } H^n ( t, x ) \, dx \leq \frac{C}{L}. 
\end{align*} 
Finally, inequality \eqref{F-Zn}, together with the upper bound on the set $W_\epsilon \setminus V^n_L$ and the lower bound on the set $W_\epsilon \setminus U_\delta$, implies that for any $A \geq 1$ and sufficiently large $n \in \mathbb{N}_+$, we have
 \begin{align*}
 	&	\int_{W_\epsilon \setminus (U_\delta \cup V^n_L)} dt \, dx \int_{|\xi| \geq A} |\mathcal{F} \sqrt{f^n_R}|^2 Z_n(1/A) d\xi \\
 	\leq & 	\int_{W_\epsilon \setminus (U_\delta \cup V^n_L)} dt \, dx \int_{|\xi| \geq A} |\mathcal{F} \sqrt{f^n_R}|^2 Z_n(1/|\xi |) d\xi \\
 	\leq & \frac{1}{\epsilon } \tilde{C} L    \left( C_{\nu }	  + \frac{ C _1 }{ r ^2 _0 } \right)   \int_{0}^{T}  dt  \iint_{ \R^3_x \times \R^3_v  }  f^n ( 1+ |v|^2) dx d v    \\
 	& + \tilde{C } \int_{0}^{T}  dt  \iint_{ \R^3_x \times \R^3_v  }  e^n( t,x,v) d x d v \leq C_{L , \epsilon , \delta}    ,
 \end{align*}
 where $\tilde{C} =  16 C_0 C_2 J^2  \pi ^3 $ is a constant. Taking the limit superior with respect to $n$ leads to the estimate
 \begin{align*}
 	\limsup_{n \rightarrow \infty} \int_{W_\epsilon \setminus (U_\delta \cup V_L^n)} dt \, dx \int_{|\xi| \geq A} |\mathcal{F} \sqrt{f_R^n}|^2 \, d\xi \leq \frac{C_{L, \epsilon, \delta} }{Z_{\infty}(1/A)},
 \end{align*}
 where the function $Z_{\infty}$ satisfies
 \begin{align*}
 	Z_{\infty}(a) = 2\pi \liminf_{n \rightarrow \infty} \int_a^{\frac{\pi}{2}} b_{n}(\cos \theta) \sin \theta \, d\theta
 	\geq 2\pi \int_a^{\frac{\pi}{2}}\liminf_{n \rightarrow \infty} b_{n}(\cos \theta) \sin \theta \, d\theta.
 \end{align*}
 Consequently, we have $\lim_{a \to 0^+ } Z_{\infty} (a) = +\infty$. This implies that for any fixed parameters $L, \epsilon, \delta >0$, the vanishing limit
 \begin{align*}
 	\lim_{A \to \infty} \limsup_{n \rightarrow \infty} \int_{W_\epsilon \setminus (U_\delta \cup V_L^n)} dt \, dx \int_{|\xi| \geq A} |\mathcal{F} \sqrt{f_R^n}|^2 \, d\xi = 0
 \end{align*}
 holds. Since the measure bound $|U_\delta \cup V_L^n| \leq \delta + \frac{C}{L}$ holds and the sequence $f^n$ is uniformly integrable, we deduce that
 \begin{align*}
 	\lim_{\substack{\delta \rightarrow 0 \\ L \rightarrow+\infty}} \sup_{n \in \mathbb{N}} \int_{U_\delta \cup V_L^n} dt \, dx \int_{\mathbb{R}^3_\xi } \left|\mathcal{F} \sqrt{f_R^n}\right|^2 \, d\xi = 0.
 \end{align*}
 Combining these results, we conclude that
 \begin{align*}
 	\lim_{A \to +\infty} \lim_{n \to \infty} \int_{W_\epsilon} dt \, dx \int_{|\xi| \geq A}|\mathcal{F} \sqrt{f_R^n}|^2 \, d\xi = 0.
 \end{align*}
 This completes the proof of Lemma \ref{v-regular}.
 \end{proof}
\begin{lemma}
	Let $\{f^n\}_{n\geq 1}$ be a sequence of approximate solutions satisfying the uniform bounds \eqref{FB2}-\eqref{FWC2}, and let $f$ be its weak $L^1$-limit. Then, $f^n$ converges strongly to $f$ in $L^1( [0,T] \times \R^3_x \times \R^3_v)$.
\end{lemma}
\begin{proof}
	To establish the strong convergence of $f^n$, we first recall the weak convergence result
	\begin{align*}
		g^n = \sqrt{ f^n } \rightharpoonup g \quad \text{in } L^2([0,T ] \times B_R (x)\times B_R(v)),
	\end{align*}
	where the limit function satisfies $g \in L^2([0,T ] \times B_R(x) \times B_R(v))$. We then proceed by considering the cases where $g$ is either zero or positive.
	
 \underline{\em  Case 1. For almost every $(t,x,v) \in (0,T) \times B_R (x)\times B_R(v), \ g(t,x,v) = 0 $.} 

The weak convergence of $\sqrt{f^n}$ implies
	\begin{align}\label{lim-fn-0}
		\lim_{n \rightarrow \infty } \int_{0}^{T} dt \iint_{ B_R(x) \times B_R (v) } \sqrt{ f^n } \, dx \, dv = \int_{0}^{T} dt \iint_{ B_R(x) \times B_R (v) } g \, dx \, dv = 0.
	\end{align}
	Moreover, the uniform integrability of $f^n$ (derived from the uniform bounds) ensures that for every $\epsilon > 0$, there exists a constant $M > 0$ such that
	\begin{align*}
		\sup_{n} \int_{0}^{T} dt \iint_{ B_R(x) \times B_R (v) } f^n \mathbf{1}_{ \{f^n \geq M\} } \, dx \, dv < \epsilon /2 .
	\end{align*}
	For this fixed $\epsilon$, the limit \eqref{lim-fn-0} guarantees the existence of $N \in \mathbb{N}_+$ such that for all $n > N$, we have the bound
	\begin{align*}
		\int_{0}^{T} dt \iint_{ B_R(x) \times B_R (v) } \sqrt{f^n } \, dx \, dv < \frac{\epsilon}{2 \sqrt{M} } .
	\end{align*}
	Consequently, for any $\epsilon > 0$ and $n > N$, the decomposition of the integral yields
	\begin{align*}
		\int_{0}^{T} dt \iint_{ B_R(x) \times B_R (v) } f^n \, dx \, dv & \leq \int_{0}^{T} dt \iint_{B_R(x) \times B_R (v)} \left( f^n \mathbf{1}_{ \{f^n \geq M\} } + f^n \mathbf{1}_{ \{f^n \leq M\} } \right) \, dx \, dv \\
		& \leq \sqrt{M} \int_{0}^{T} dt \iint_{B_R(x) \times B_R (v)} \sqrt{f^n } \, dx \, dv + \frac{\epsilon}{2} 
		 < \epsilon.
	\end{align*}
	Therefore, given the non-negativity of $f^n$, we conclude that $f^n$ converges strongly to $0$ in $L^1 ( [0,T] \times B_R(x) \times B_R (v))$. 

 \underline{\em  Case 2. For almost every $(t,x,v) \in (0,T) \times B_R (x)\times B_R(v), \ g(t,x,v) > 0 $.} 
 
 We have obtained a smoothing estimate for the velocity variable in Lemma \ref{v-regular}. Next, assume that $\rho_\delta= \delta^{-3} \rho \left( \frac{\cdot }{ \delta } \right)$ is a mollifier in the velocity variable. By Parseval's identity, one can derive
 \begin{align}\label{F-sqfn}
 	\| \sqrt{f^n_R} - \sqrt{f^n_R} \ast \rho_\delta \|_{L^2 ( B_R (v) ) } = 	\|  \mathcal{F} (\sqrt{f^n_R} ) - \mathcal{F} ( \sqrt{f^n_R} \ast \rho_\delta ) \|_{L^2 ( \R ^3 _v)},
 \end{align}
 where $\mathcal{F} (\sqrt{f^n_R} \ast \rho_\delta) = \mathcal{F} (  \sqrt{f^n_R})\cdot \mathcal{F} ( \rho_\delta )$.
Given that $ \mathcal{F} ( \rho_\delta ( v ) )  = \frac{1}{\delta ^3 } \mathcal{F} \left( \rho \left( \frac{ v  }{\delta } \right) \right) = \hat{ \rho } ( \delta \xi )$ and $\hat{ \rho }  ( 0 )   =1$, we observe that $ \hat{ \rho } (  \delta  \xi  ) \approx  1$ whenever $ | \xi | \ll    \delta   $. Therefore, as $\delta \rightarrow 0_+$, for any fixed $A$, $|1 - \mathcal{F}(\rho_\delta)| \rightarrow 0$ uniformly for $|\xi| \leq A$. Consequently,  
\begin{align*}
\lim_{\delta \to 0} 	\int_{ |\xi| \leq A } | \mathcal{F} (\sqrt{f^n_R} ) |^2 ( 1- \mathcal{F} (\rho_\delta ) )^2 d \xi  =0.
\end{align*} 
From this, we obtain the estimate
 \begin{align*}
 \|  \sqrt{f^n_R} - \sqrt{f^n_R} \ast \rho_\delta \|_{L^2(W_\epsilon \times B_R(v))}  
\leq &   \sup \limits_{n} \int_{W_\epsilon} dt \, dx  \int_{ |\xi| \leq A } | \mathcal{F} (\sqrt{f^n_R} ) |^2 ( 1- \mathcal{F} (\rho_\delta ) )^2 d \xi   \\
&+ \int_{W_\epsilon} dt \, dx   \int_{ |\xi| \geq A } | \mathcal{F} (\sqrt{f^n_R} ) |^2  d\xi .
 \end{align*}  
This implies
 \begin{equation*}
 	\lim_{\delta \to 0}\limsup_{  n \rightarrow \infty } \| \sqrt{f^n_R } - \sqrt{f^n_R } \ast \rho_\delta \|_{L^2(W_\epsilon \times B_R(v))} = 0.
 \end{equation*}
By the truncation method,  it follows that
\begin{align*}
\lim_{\delta \to 0}\limsup_{  n \rightarrow \infty } \| \sqrt{f^n } - \sqrt{f^n } \ast \rho_\delta \|_{L^2(W_\epsilon \times B_R(v))} = 0.
\end{align*}  
 According to \eqref{fn-f}, we have
 \begin{align*}
 	 \int_{\mathbb{R}^3_v } f^n   dv \rightarrow  \int_{\mathbb{R}^3_v} f  dv \quad \text{ in } \ L^1( [0,T  ] \times  \mathbb{R}^3_x ) .
 \end{align*}
 Consequently, for almost every $(t,x) \in [0,T] \times B_R(x)$, the convergence $ \int_{B_R(v)} f^n d v \rightarrow  \int_{B_R(v)} f d v$ holds. Specifically, for almost every $(t,x) \in [0,T] \times B_R(x)$, the bound $\| \sqrt{f^n} ( t,x, \cdot ) \|_{L^2 ( B_R(v)) } \leq C(t,x) $ is satisfied. Due to the reflexivity of the $L^2$ space, for almost all $(t,x) \in [0,T] \times B_R(x)$, we have the weak convergence
 \begin{align*}
 	\sqrt{f^n} ( t,x , \cdot ) \rightharpoonup g( t,x,\cdot ) \quad \text{ in } L^2 ( B_R(v)).
 \end{align*}
 Moreover, for almost all $(t,x) \in [0,T] \times B_R(x)$, the following estimate holds
  \begin{align*}
 	\| \sqrt{f^n} \ast \rho_\delta \|_{L^2 (B_R(v))} \leq \| \rho_\delta \|_{L^1( \R^3_v) )} \cdot \| \sqrt{f^n} \|_{L^2 ( B_R(v))} \leq C(t,x) ,
 \end{align*}
 which implies that for almost all $(t,x) \in [0,T] \times B_R (x)$, the sequence $\sqrt{f^n} \ast \rho_\delta $ is bounded in $L^2 ( B_R (v )) $. Besides, for any $h \in  \R^3$,
 \begin{align*}
 &	\| \tau_h (\sqrt{f^n} \ast \rho_\delta) - \sqrt{f^n} \ast \rho_\delta \|_{L^2 ( [0, T] \times B_R(x) \times B_R (v)) } \\
 \leq & \| \sqrt{f^n} \|_{L^2 ([0, T] \times B_R (x)\times B_R(v) ) } \| \tau_h \rho_\delta  - \rho_\delta \|_{L^1( \R^3_v )} \rightarrow 0 , \quad (h \rightarrow 0 ).
 \end{align*}
 By the M. Riesz-Fr\'{e}chet-Kolmogorov Theorem (Lemma \ref{MRFK}), for almost every $(t, x) \in [0, T] \times B_R(x)$, the sequence ${ \sqrt{f^n} \ast \rho_\delta }$ is relatively compact in $L^2(B_R(v))$. Therefore,
 \begin{align*}
 	\lim_{n \to \infty} \| \sqrt{f^n } \ast \rho_\delta - g \ast \rho_\delta \|_{L^2 ( B_R(v) ) } = 0 , \quad \text{ a.e. } (t,x) \in [0,T] \times B_R(x).
 \end{align*}
 Due to the equi-integrability of $\{f^n\}_{n \geq 1}$, for any $\epsilon>0$, there exists $\eta >0$ such that for  any measurable subset $E \subset W_\epsilon \times B_R(v)$ satisfying $|E|<\eta$, 
 $$
  \int_{E} f^n d t d x d v < \frac{\epsilon}{\| \rho_\delta \|_{L^1( \R^3_v )}  }.
  $$ 
Consequently, we obtain
 $$
 \int_{E} | \sqrt{f^n } \ast \rho_\delta |^2 d t d x d v \leq  \| \rho_\delta \|_{L^1( \R^3_v )}   \int_{E} f^n d t d x d v < \epsilon.
 $$
Furthermore, since
 \begin{align*}
 	\| \sqrt{f^n } \ast \rho_\delta -  g \ast \rho_\delta \|_{L^2 ( W_\epsilon \times B_R(v) ) } \leq \| \rho_\delta \|_{L^1( \R^3_v ) } \left( \|f^n \|_{L^1 ( W_\epsilon \times B_R(v) )} + \| g \|_{ L^2 ( W_\epsilon \times B_R (v) ) } \right) \leq C,
 \end{align*}
 the sequence $\left\{ \sqrt{f^n } \ast \rho_\delta\right\}_{n \geq 1}  $ is bounded.  
 Combining this with the estimate over small sets, we conclude that  $\left\{ \sqrt{f^n } \ast \rho_\delta \right\}_{n \geq 1} $  is equi-integrable in $ L^2 ( W_\epsilon \times B_R (v) ) $. Therefore,  by the Vitali Convergence Theorem (Lemma \ref{L-VC}), we obtain
 \begin{align*}
 	\lim_{n \to \infty}  \| \sqrt{f^n } \ast \rho_\delta -  g \ast \rho_\delta \|_{L^2 ( W_\epsilon \times B_R(v) ) } =0.
 \end{align*}
 We now establish that $\{\sqrt{f^n}\}_{n \geq 1} $ is a Cauchy sequence in $(W_\epsilon \times B_R(v))$. Fix any $\epsilon > 0$.  By the approximation properties derived earlier, we can choose a fixed $\delta > 0$ sufficiently small and find an integer $N_1 \in \mathbb{N}$ such that for all $n > N_1$, 
 \begin{equation*} 
 	\| \sqrt{f^n} - \sqrt{f^n} \ast \rho_\delta \|_{L^2(W_\epsilon \times B_R(v))} < \frac{\epsilon}{3}. 
 \end{equation*}
 For this fixed $\delta$, since the sequence $\{ \sqrt{f^n} \ast \rho_\delta \}_{n \geq 1}$ converges strongly, there exists $N_2 \in \mathbb{N}$ such that for all $m, n > N_2$, 
 \begin{equation*} 
 	\| \sqrt{f^n} \ast \rho_\delta - \sqrt{f^m} \ast \rho_\delta \|_{L^2(W_\epsilon \times B_R(v))} < \frac{\epsilon}{3}. 
 \end{equation*} 
 Let $N = \max \{ N_1, N_2 \}$. Then for any $n, m > N$, 
 \begin{align*} 
 	&\| \sqrt{f^n} - \sqrt{f^m} \|_{L^2(W_\epsilon \times B_R(v))} \\
 	\leq & \| \sqrt{f^n} - \sqrt{f^n} \ast \rho_\delta \|_{L^2(W_\epsilon \times B_R(v))}  
 	  + \| \sqrt{f^n} \ast \rho_\delta - \sqrt{f^m} \ast \rho_\delta \|_{L^2(W_\epsilon \times B_R(v))} \\ 
 	& + \| \sqrt{f^m} \ast \rho_\delta - \sqrt{f^m} \|_{L^2(W_\epsilon \times B_R(v))}   < \epsilon. 
 \end{align*} 
 Hence, $\{\sqrt{f^n}\}_{n \geq 1}$ is a Cauchy sequence in $L^2(W_\epsilon \times B_R(v))$.
 
 Letting $\epsilon \to 0$, the results from \emph{Case 1} and \emph{Case 2} jointly imply that $\sqrt{f^n}$ converges strongly in $L^2([0,T] \times B_R(x) \times B_R(v))$. Consequently, $f^n$ converges in measure to some function $\tilde{f}$ on $[0,T] \times B_R(x) \times B_R(v)$. The weak convergence $f^n \rightharpoonup f$ ensure that $\tilde{f} = f$ almost everywhere. Combining this convergence in measure with the equi-integrability of $\{f^n\}_{n\geq 1 }$ and the Vitali Convergence Theorem (Lemma \ref{L-VC}), we obtain that $f^n$ converges strongly to $f$ in $L^1([0,T] \times B_R(x) \times B_R(v))$.
 
 On the other hand, the uniform bound \eqref{FB2} on $f^n$ establishes tightness at infinity:
 \begin{align*} 
 	\int_0^T dt \iint_{\mathbb{R}^3_x \times \mathbb{R}^3_v} f^n \left( \mathbf{1}_{|x|\geq R } + \mathbf{1}_{|v|\geq R } \right) \, dx \, dv 
   \leq \frac{1}{R^2} \int_0^T dt \iint_{\mathbb{R}^3_x \times \mathbb{R}^3_v} f^n \left(|x|^2 + |v|^2 \right) \, dx \, dv 
   \leq \frac{C}{R^2}. 
 \end{align*}
 Finally, by combining the local strong convergence with the uniform decay at infinity, we conclude that $f^n$ converges strongly to $f$ in $L^1([0,T] \times \mathbb{R}^3_x \times \mathbb{R}^3_v)$.
 \end{proof}

 \section{Convergence of the renormalized collision operator $ ( \mathcal{R}_1 )^n $ and $ ( \mathcal{R}_2 )^n $ }\label{Sec-cnv}
 In Section \ref{Sec-integrable},  we show that the operators $ ( \mathcal{R}_1 )  $, $ ( \mathcal{R}_2 )  $, and $ ( \mathcal{R}_3 )  $ are well defined in the sense of distributions. To further verify that the strong  $L^1$-limit $f$   of the approximating sequence  $f^n$ is indeed a renormalized solution to the Boltzmann equation \eqref{Boltz}, it remains to study the convergence of the renormalized approximate operator $\beta_\delta '(f^n) \tilde{Q}_n (f^n ,f^n)$. This problem reduces to analyzing the convergence, in the distributional sense, of the corresponding terms  $ ( \mathcal{R}_1 )^n  $, $ ( \mathcal{R}_2 )^n   $, and $ ( \mathcal{R}_3 )^n   $. Here, each term $ ( \mathcal{R}_k )^n   $ (for $k =1,2,3$) is obtained from $ ( \mathcal{R}_k )  $ by replacing $f$ with the approximating sequence $f^n$ and the collision kernel $B$ with its approximation $B_n$.  Owing to the length of the argument, we first establish the convergence of $ ( \mathcal{R}_1 )^n   $ and $ ( \mathcal{R}_2 )^n   $ in this section. 

\subsection{Convergence of the renormalized collision operator $ ( \mathcal{R}_1 )^n $}\label{Subsec:Cnv-R1}  
 Given any function $\varphi ( t,x , v) \in L^\infty ( [0,T] \times \R^3_x \times  \R^3_v  )$ with compact  support in  $[0,T] \times B_R (x) \times  B_R (v)$, the integral satisfies 
 \begin{align*} 
 	\int_{ \R ^3_v } ( \mathcal{R}_1 )^n \varphi d v = \iiint_{ \mathbb{R}^{ 3}_v \times \mathbb{R}^{ 3}_{v_*} \times  \S^2 } \left[ f^n \beta_\delta ' ( f^n) - \beta_\delta ( f^n) \right] f^n_* S^n (| v-v_*|) \varphi ( v ) d v dv_* d \sigma . 
 \end{align*}  
 \begin{lemma} 
 	For any function $\varphi ( t,x , v) \in L^\infty ( [0,T] \times \R^3_x \times  \R^3_v  )$ with compact  support in  $[0,T] \times B_R(x) \times  B_R (v)$, we establish the convergence 
 	\begin{align}\label{cnv-R1} 
 		\lim_{n \to \infty}  \int_{0}^{T} dt \iint_{ \mathbb{R}^3_x \times \mathbb{R}^3_v } \left( 1+ \frac{1}{n} \int_{\R^3_v} f^n d v \right)^{-1}  ( \mathcal{R}_1 )^n \varphi d x d v =  \int_{0}^{T} dt \iint_{ \mathbb{R}^3_x \times \mathbb{R}^3_v } ( \mathcal{R}_1 )  \varphi d x d v . 
 	\end{align} 
 \end{lemma}
 \begin{proof} 
 Since $\beta_\delta  $ is a bounded function,  the term $f^n \beta_\delta ' ( f^n ) - \beta_\delta ( f^n )  $ converges in weak-star topology of $L^\infty $. Replacing $B $ with $B_n $ in Lemma \ref{prop-S} yields the estimate 
 \begin{align*}  
 	S_n ( |z| ) \leq C   \left( M_n ( |z| ) +|z| M_n' ( |z| ) \right),  
 \end{align*}
where the integral is defined as 
\begin{align*} 
	S_n  (|z|) = 2\pi \int_0^{\frac{\pi}{2}}\sin   \theta \left[ \frac{1}{\cos^3 \left(\frac{\theta}{2}\right)} B_n  \left( \frac{|z|}{\cos \left(\frac{\theta}{2}\right)}, \cos \theta \right) - B_n  (|z|, \cos \theta) \right] d\theta .  
\end{align*} 
Using the definition of the approximate collision kernel  $B_n$ in Lemma \ref{L-Bn} and the expression of the collision kernel $B$ in \eqref{B-gamma}, we find that there exists a constant $C_{\nu } $ such that for any $n\geq 1$ the uniform bound reads 
\begin{align}\label{Sn-gamma} 
	S_n ( |z| ) \leq C  \left( M_n ( |z| ) +|z| M_n' ( |z| ) \right) \leq C_{\nu } | z| ^\gamma, \quad \text{with } \gamma = 1- \frac{4}{s-1}. 
\end{align} 
Consequently, $ S_n ( |z| ) \in L^1_{\loc } ( \R^3_z)$.   Furthermore, applying the Lebesgue's Dominated Convergence Theorem (Lemma \ref{LCT}) yields that for any compact set $K \subset \R^3_z$,
\begin{align*} 
	\lim_{n \to \infty} \int_{ K }  S^n ( |z| ) d z = \int_{ K }  S ( |z| ) d z.
\end{align*} 
 
Now,  for any large enough constant $L$ (without loss of generality, assume that $L\geq R+1$), we consider the following two cases.

 \underline{\em Case 1.  $|v_*| \leq L$. }   Let $R' = R + L$. Then by applying the properties of the convolution, we find that there exists a  constant $ C( R, L )>0$,  such that
 \begin{align*}
 	\|( f^n 1_{|v | \leq L }) \ast_v S_n \|_{L^1 ( [0,T] \times \R ^3_x \times B_R (v) ) } &= \iint_{B_R (v)\times B_L (v_*)} f^n_* S^n ( |v-v_*| ) d v dv_* \\
 	& \leq \| f^n \|_{ L^1 (  [0,T] \times \R ^3_x \times \R ^3_v ) } \| S^n \|_{L^1 ( B_{R'} ) }\leq C( R,L ). 
 \end{align*}
Furthermore, for any $h = (h_t , h_x , h_v) \in \R_+  \times \R^3_x \times \R^3_v $, with $|h|<1$, we have 
 \begin{align*}
 	&	\| \tau_h \left( \left( f^n 1_{|v | \leq L } \right) \ast_v S^n \right) - \left( f^n 1_{|v | \leq L } \right) \ast_v S^n \|_{L^1 ( [0,T] \times \R ^3_x \times B_R (v)) }\\
 	\leq & \| \tau_h f^n - f^n \|_{ L^1 ( [0,T] \times \R ^3  _x \times B_L (v))  } \| S^n \|_{L^1 ( B_{R'(v)})}.
 \end{align*}
Due to the compactness of $f^n$, it follows that the right-hand side converges  to $0$ as $|h|\to 0$.

 \underline{\em  Case 2.  $|v_*| \geq L $.}  This implies that $|v-v_*|\geq L-R$ and $|v-v_*|^2 \leq C_R ( 1+ |v_*|^2)$. Since $0 \leq  \gamma \leq 1$, we conclude that
 \begin{equation}\label{fS}
 \begin{aligned}
 	\| \left( f^n 1_{|v | \geq L  } \right) \ast_v S^n \|_{L^1 ( B_R (v))} &\leq \frac{C_{\nu } }{( L-R  )^{2-\gamma } } \iint_{|v-v_*| \geq L-R  } f^n_* |v-v_*|^2 d v dv_* \\
 	& \leq  \frac{ C_{\nu  , R} }{( L-R  )^{2-\gamma } }  \int_{\R ^3_{v_*} } f^n_*  ( 1+ |v_*|^2 ) dv_* . 
 \end{aligned} 
 \end{equation} 
 Furthermore, for any $h = (h_t , h_x , h_v) \in \R_+  \times \R^3_x \times \R^3_v $, with $|h|<1$, we have
 \begin{equation}\label{tfM-S}
 	\begin{aligned} 
&  \| \tau_h \left( \left( f^n 1_{|v | \geq L } \right) \ast_v S^n \right) - \left( f^n 1_{|v | \geq L } \right) \ast_v S^n \|_{  L^1 (  [0,T] \times \R ^3_x \times B_R (v) )}\\
 \leq &     \frac{2 C_{\nu , R} }{( L-R -1 )^{2-\gamma } }   \int_{0}^{T} d t \iint_{\R ^3_x \times \R^3_{v_*} } f^n_*   ( 1+ |v_*|^2 ) d x d v_* .  
 \end{aligned}
 \end{equation}   
By applying the uniform bound \eqref{FB2} on $f^n$,  we conclude that the right-hand side converges to zero as $L \to +\infty $.

Finally, combining the results of {\em Case 1} and {\em Case 2}, we find that 
 \begin{align*}
 &	\|   f^n  \ast_v S^n \|_{L^1 ([0,T]\times \R^3 _x \times  B_R (v) ) } \\ \leq &  \| \left( f^n 1_{| v | \leq L } \right) \ast_v S^n \|_{L^1  ([0,T]\times \R^3_x \times  B_R (v)) } + \| \left( f^n 1_{|v | \geq L } \right) \ast_v S^n \|_{L^1  ([0,T]\times \R^3_x \times  B_R (v)) } \leq C( R,L ),
 \end{align*}
 and that for any $ h = (h_t , h_x , h_v) \in \R_+  \times \R^3_x \times \R^3_v $, with $|h|<1$, one has 
 \begin{align*}
 	& \| \tau_h  \left( f^n   \ast_v S^n \right) - \left( f^n   \ast_v S^n \right) \|_{  L^1 (  [0,T] \times \R ^3_x \times B_R(v) )} \\ 
 	\leq &  \| \tau_h \left( \left( f^n 1_{|v | \geq L } \right) \ast_v S^n \right) - \left( f^n 1_{|v | \geq L } \right) \ast_v S^n \|_{  L^1 (  [0,T] \times \R ^3 _x\times B_R (v)) } \\
 	& + \| \tau_h \left( \left( f^n 1_{|v | \leq L } \right) \ast_v S^n \right) - \left( f^n 1_{|v | \leq L } \right) \ast_v S^n \|_{  L^1 (  [0,T] \times \R ^3_x \times B_R (v) ) }  
 	\rightarrow   0 \quad ( h \rightarrow 0 , L \rightarrow \infty ).
 \end{align*}
Furthermore, for any $K>0$, the following inequality holds
 \begin{align*}
 \| \left( f^n 1_{|x | \geq K } \right) \ast_v S^n \|_{  L^1 (  [0,T] \times \R ^3_x \times B_R (v))}  
 \leq &  \frac{  C_{\nu  , R} }{K }   \int_{0}^{T} d t \iint_{\R ^3_x \times \R^3 _v } f^n_*   |x |  \cdot  ( 1+ |v_*|  )    d x d v_*  \\
   \leq &    \frac{  C_{\nu  , R} }{K }  \int_{0}^{T} d t \iint_{\R ^3 _x \times \R^3_{v_*}  } f^n_*   ( 1+ |v_*|^2+ |x| ^2 )    d x d v_*  .
 \end{align*} 
 Consequently, we apply the M. Riesz–Fréchet–Kolmogorov Theorem (Lemma \ref{MRFK}) together with the uniform bound \eqref{FB2} on $f^n$   to conclude that $  \left\{ f^n \ast_v S^n\right\}_{n \geq 1 }  $ is relatively compact in $L^1 (  [0,T] \times \R ^3 _x \times B_R (v) ) $. 
 
We now proceed to establish the convergence $ f^n \ast_v S^n \rightarrow f \ast_v S $ in $L^1 (  [0,T] \times \R ^3 _x \times B_R(v) )$. First, we observe that the term associated with the truncated velocity domain satisfies
 \begin{align*}
 	& \int_{0}^{T} dt \iint_{ \mathbb{R}^3 _x \times B_{R} (v) } \left|  \left( f^n 1_{|v | \leq L } \right) \ast_v S^n  - \left( f  1_{|v | \leq L } \right) \ast_v S  \right|  dx d v \\
 	\leq & \| f^n - f \|_{L^1 ( [0,T] \times \R ^3_x  \times \R ^3_v ) } \| S \|_{L^1 ( B_L (z ))} + \sup_{n  } \| f^n \|_{ L^1 ( [0,T] \times \R ^3_x \times  \R ^3 _v) }  \left|  \int_{B_{L}(z)}  ( S^n ( |z| ) - S( |z| ) ) dz   \right| \\
 	\rightarrow & 0 \quad ( n \rightarrow \infty ) .
 \end{align*}
Additionally, for the tail contribution where $|v| > L$ the estimate is given by
 \begin{align*}
 	& \int_{0}^{T} dt \iint_{ \mathbb{R}^3_x \times B_{R}(v) } \left|  \left( f^n 1_{|v | > L } \right) \ast_v S^n  - \left( f  1_{|v | > L } \right) \ast_v S  \right|  dx d v \\
 	\leq &\frac{ C_{\nu  , R} }{( L-R  )^{2-\gamma } }  \sup_{n\geq 1 } \int_{0}^{T} dt \iint_{ \mathbb{R}^3_x \times \mathbb{R}^3_{v_*}  }  \left[ f^n_* ( 1+|v_*|^2 ) + f_* ( 1+|v_*|^2 )   \right] d x d v_* \rightarrow 0 ,
 \end{align*}
as $L \rightarrow \infty$. Consequently, combining these results yields  
 \begin{align*}
 	& \| f^n \ast_v S^n - f \ast_v S \|_{L^1 ( [0,T] \times \R ^3_x \times B_R (v) )} \\
 	\leq & \int_{0}^{T} dt \iint_{ \mathbb{R}^3_x \times B_{R} (v) } \left|  \left( f^n 1_{|v | \leq L } \right) \ast_v S^n  - \left( f  1_{|v | \leq L } \right) \ast_v S  \right|  dx d v \\
 	&+ \sup_{n\geq 1 }  \int_{0}^{T} dt \iint_{ \mathbb{R}^3_x \times B_{R} (v) } \left|  \left( f^n 1_{|v | > L } \right) \ast_v S^n  - \left( f  1_{|v | > L } \right) \ast_v S  \right|  dx d v   \to   0, 
\end{align*} 
as $n \to \infty$ and $L \to \infty$. This establishes the desired strong convergence.

Furthermore, since $ \left( 1+ \frac{1}{n} \int_{\R^3_v } f^n d v \right)^{-1}  (f^n\beta_\delta ' ( f^n ) - \beta_\delta ( f^n )) $ is bounded in $L^\infty ( [0,T] \times \R^3_x \times \R^3_v   )$ and satisfies
 \begin{align*}
 \left( 1+ \frac{1}{n} \int_{\R^3_v } f^n d v \right)^{-1}  (f^n\beta_\delta ' ( f^n ) - \beta_\delta ( f^n )) \rightarrow  f \beta_\delta ' ( f  ) - \beta_\delta ( f  ) \quad \text{ a.e.} \ t,x,v,
 \end{align*}
   we invoke the Product Limit Theorem (Lemma \ref{product limit}) to deduce that $\left( 1+ \frac{1}{n} \int_{\R^3_v } f^n d v \right)^{-1}  ( \mathcal{R}_1 )^n$ converges weakly to $( \mathcal{R}_1 )$. Specifically, for any $\varphi \in L^\infty ( [0,T] \times \R^3_x \times \R^3 _v )$, with compact  support in  $[0,T] \times B_R (x)\times  B_R (v)$,  we have
 \begin{align*}
 	\lim_{n \to \infty}  \int_{0}^{T} dt \iint_{ \mathbb{R}^3_x \times \mathbb{R}^3_v } \left( 1+ \frac{1}{n} \int_{\R^3_v } f^n d v \right)^{-1}  ( \mathcal{R}_1 )^n \varphi d x d v =  \int_{0}^{T} dt \iint_{ \mathbb{R}^3_x \times \mathbb{R}^3_v } ( \mathcal{R}_1 )  \varphi d x d v .
 \end{align*}
 Thus, the convergence \eqref{cnv-R1} holds.
 \end{proof}

\subsection{Convergence of the renormalized collision operator $ ( \mathcal{R}_2 )^n $}\label{Subsec:Cnv-R2}  
Next, we examine the term $(\mathcal{R}_2)^n$. Let $\varphi(t,x,v) \in L^\infty([0,T] \times \mathbb{R}^3_x; C^{\alpha}_c(\mathbb{R}^3_v))$ be a function with compact support in $[0,T] \times B_R(x) \times B_R(v)$. Recalling the definition of $(\mathcal{R}_2)^n$, which is given by 
\begin{equation*}
	\int_{\mathbb{R}^3_v} (\mathcal{R}_2)^n \varphi \, dv = \iint_{\mathbb{R}^3_v \times \mathbb{R}^3_{v_*}} f^n_* \beta_\delta(f^n) \left( \mathcal{T}^n \varphi \right) \, dv \, dv_*,
\end{equation*}
where the operator $\mathcal{T}^n $ is defined as
\begin{equation*}
	\mathcal{T}^n \varphi(t,x,v) = \int_{\mathbb{S}^2} B_n(v - v_*, \sigma) (\varphi' - \varphi) \, d\sigma.
\end{equation*} 
\begin{lemma}
For any function $\varphi ( t,x , v) \in L^\infty([0,T] \times \mathbb{R}^3_x; C^{\alpha}_c(\mathbb{R}^3_v))$ with compact  support in  $[0,T] \times B_R(x) \times  B_R (v)$, and assuming $\nu < \alpha <1$, we  establish the convergence
\begin{align}\label{cnv-R2}
	\lim_{n \to \infty}  \int_{0}^{T} dt \iint_{ \mathbb{R}^3_x \times \mathbb{R}^3_v } \left( 1+ \frac{1}{n} \int_{\R^3_v} f^n d v \right)^{-1}  ( \mathcal{R}_2 )^n \varphi d x d v =  \int_{0}^{T} dt \iint_{ \mathbb{R}^3_x \times \mathbb{R}^3_v } ( \mathcal{R}_2 )  \varphi d x d v .
\end{align}
\end{lemma}
\begin{proof}  
Substituting $B_n $ for $B$ in Lemma \ref{R2}, we obtain the analogous estimate
 \begin{align}\label{Tn}
 	|\mathcal{T}^n \varphi( t,x, v)| 
 	\leq   C_{\nu, \alpha}  \|\varphi\|_{L^\infty _{t,x}  C{^{\alpha}_c} (v)} |v - v_*|^{\alpha+\gamma },
 \end{align}
where $\nu < \alpha <1$. This implies that $\mathcal{T}^n \varphi$ converges in the sense of distributions. Combining this with the density of $C^\infty _c $ in $L^1$, we conclude that for any $g \in L^1 ([0,T] \times \R^3_x \times \R^3_v)$,  we have
 \begin{align}\label{cnv-tn}
 \lim\limits_{  n \rightarrow \infty } \int_{0}^{T} d t \iint_{\R^3_x \times  \R ^3_v } \mathcal{T}^n \varphi \cdot g \,d x  d v = \int_{0}^{T} d t \iint_{\R^3_x \times  \R ^3 _v}  \mathcal{T} \varphi \cdot g \,d x  d v
 \end{align}
 
Since $M ( |z|) $ is continuous on $\left\{ z \in \R ^3:      |z| \leq L   \right\}$, it follows that whenever $   |v-v_*| \leq L$, the term $ \mathcal{T}^n \varphi $ forms a bounded sequence of functions in $L^\infty ( [0,T] \times \R^3_x \times  B_R (v) ) $. Moreover, since $f^n$ converges strongly to $f$  in $L^1( [0,T] \times \R^3 _x \times \R^3_v    )$, and  $\beta_\delta \in C^2 $ is bounded, we deduce that the sequence $  \left( 1+ \frac{1}{n} \int_{\R^3_v} f^n d v \right)^{-1} f^n_* \beta_\delta ( f^n  ) $ is weakly relatively compact in $L^1 ( [0,T] \times \R^3_x \times \R^3_{v_*} \times B_R(v)  ) $. On the other hand, the strong convergence of $f^n$ in $L^1( [0,T] \times \R^3_x \times \R^3 _v   )$ implies  
\begin{align*}
 \left( 1+ \frac{1}{n} \int_{\R^3_v} f^n d v \right)^{-1} f^n_* \beta_\delta ( f^n  ) \to f _* \beta_\delta ( f   )  \quad \text{a.e.}\ t,x,v,v_*.
\end{align*}
This pointwise convergence identifies the weak limit. Specifically, $ \left( 1+ \frac{1}{n} \int_{\R^3_v } f^n d v \right)^{-1}  f^n_* \beta_\delta ( f^n  ) $ converges weakly to $f _* \beta_\delta ( f   ) $ in $L^1 ( [0,T] \times \R^3_x \times \R^3_{v_*} \times B_R(v)  ) $. Therefore, by the Product Limit Theorem (Lemma \ref{product limit}), one obtains  
  \begin{align}\label{cnv-fnb}
 	\lim_{n \to \infty} \int_{0}^{T} dt \iiint_{ \mathbb{R}^3_x \times \mathbb{R}^3_v \times \mathbb{R}^3_{v_*} } \left(  \frac{f^n_* \beta_\delta  ( f^n )   }{1+ \frac{1}{n} \int_{\R^3} f^n d v}  - f _* \beta_\delta ( f  )   \right) ( \mathcal{T}^n \varphi ) \mathbf{1}_{\{   |v-v_*| \leq L \} }d x d v d v_* = 0 .
 \end{align}
 In view of the convergence result \eqref{cnv-tn} and the fact that $f _* \beta_\delta ( f  ) \in L^1 ( [0,T] \times \mathbb{R}^3_x \times \mathbb{R}^3 _{v_*}\times B_R(v) )$, we have
 \begin{equation}\label{cnv-tnt}
 	\lim_{n \to \infty} \int_{0}^{T} dt \iiint_{ \mathbb{R}^3_x \times \mathbb{R}^3_v \times \mathbb{R}^3_{v_*} }  f _* \beta_\delta ( f  )  ( ( \mathcal{T}^n \varphi ) - ( \mathcal{T}  \varphi )  ) \mathbf{1}_{\{   |v-v_*| \leq L \} }d x d v d v_* = 0 .
 \end{equation}
 Combining the convergences \eqref{cnv-fnb} and \eqref{cnv-tnt} yields
 \begin{equation}
	\lim_{n \to \infty}\int_{0}^{T} dt \iiint_{\mathbb{R}^3_x \times \mathbb{R}^3_v \times \mathbb{R}^3_{v_*}  } \left\{ \frac{f^n_* \beta_\delta  ( f^n )   }{1+ \frac{1}{n} \int_{\R^3_v} f^n d v}  ( \mathcal{T}^n \varphi )    -  f _* \beta_\delta ( f  )  ( \mathcal{T} \varphi ) \right\}  \mathbf{1}_{\{   |v-v_*| \leq L \} } d x d v d v_*=0. 
 \end{equation}
 Using the estimates \eqref{T-M} and \eqref{Tn}, we obtain
 \begin{equation}\label{L-large} 
 \begin{aligned}
 	& \int_{0}^{T} dt \iiint_{ \mathbb{R}^3_x \times \mathbb{R}^3_v \times \mathbb{R}^3_{v_*} }\left\{ \left( 1+ \frac{1}{n} \int_{\R^3_v} f^n d v \right)^{-1}  f^n_* \beta_\delta  ( f^n ) (\mathcal{T}^n\varphi ) +  f _* \beta_\delta  ( f  ) (\mathcal{T} \varphi )  \right\} \mathbf{1}_{ |v-v_*| \geq L  } d x d v dv_*  \\
 	\leq & C_{\nu , \alpha  } \iiint_{ \mathbb{R}^3_x \times \mathbb{R}^3_v \times \mathbb{R}^3_{v_*} }  \left(f^n_*  \beta_\delta  ( f^n ) + f _* \beta_\delta  ( f  ) \right) |v-v_*|^{\alpha + \gamma } \mathbf{1}_{ |v-v_*| \geq L  } d x d v dv_*  \\
 	\leq & \frac{C_{\nu , \alpha, R}}{ L^{2- \alpha- \gamma } }  \sup_{n }  \int_{0}^{T} dt \iint_{   \mathbb{R}^3_x \times \mathbb{R}^3_{v_*} } (f^n_* + f_* ) ( 1+ |v_* |^2 ) d x dv_*  .
 \end{aligned} 
 \end{equation}
Since $0\leq \gamma \leq 1$ and $\nu < \alpha <1 $, the condition $2-\alpha -\gamma >0 $ always holds.  Therefore, combining this with the uniform bound for $f^n$ in \eqref{FB2}, the right-hand side of the above inequality vanishes as  $L \rightarrow \infty$.
 Therefore, we have
\begin{align*} 
	& \left| \int_{0}^{T} dt \iiint_{ \mathbb{R}^3_x \times \mathbb{R}^3_v \times \mathbb{R}^3_{v_*} }   \left\{ \frac{f^n_* \beta_\delta  ( f^n )   }{1+ \frac{1}{n} \int_{\R^3_v} f^n d v}  ( \mathcal{T}^n \varphi )    -  f _* \beta_\delta ( f  )  ( \mathcal{T} \varphi )  \right\}  d x d v d v_*  \right| \\
	\leq & \sup_{n }  \int_{0}^{T} dt \iiint_{ \mathbb{R}^3_x \times \mathbb{R}^3_v \times \mathbb{R}^3_{v_*} }  \left\{ \frac{f^n_* \beta_\delta  ( f^n )   }{1+ \frac{1}{n} \int_{\R^3_v} f^n d v}  ( \mathcal{T} ^n \varphi ) +  f _* \beta_\delta ( f  )  ( \mathcal{T} \varphi )    \right\} \mathbf{1}_{ \{  |v-v_*| > L    \}}  d x d v d v_* \\ 
	& + \int_{0}^{T} dt \iiint_{\mathbb{R}^3_x \times \mathbb{R}^3_v \times \mathbb{R}^3_{v_*} } \left\{ \frac{f^n_* \beta_\delta  ( f^n )   }{1+ \frac{1}{n} \int_{\R^3_v} f^n d v}  ( \mathcal{T} ^n \varphi )    -  f _* \beta_\delta ( f  )  ( \mathcal{T} \varphi ) \right\}  \mathbf{1}_{\{ |v-v_*| \leq L \} } d x d v d v_* 	\rightarrow  0  
\end{align*} 
as $n \rightarrow \infty$  and $ L \rightarrow + \infty$. Namely,  we obtain
 \begin{align*} 
 	\lim_{n \to \infty} \int_{0}^{T} dt \iint_{  \mathbb{R}^3_x \times \mathbb{R}^3 _v}  \left( 1+ \frac{1}{n} \int_{\R^3_v} f^n d v \right)^{-1}  ( \mathcal{R}_2 )^n \varphi d x d v =\int_{0}^{T} dt \iint_{  \mathbb{R}^3_x \times \mathbb{R}^3_v }    ( \mathcal{R}_2 )  \varphi d x d v .
 \end{align*}
  Thus, the convergence \eqref{cnv-R2} holds.
\end{proof} 
\begin{remark}
	For test functions $\varphi \in L^\infty \left( [0,T] \times \mathbb{R}^3_x ; W^{1,\infty}_0(\mathbb{R}^3_v)\right)$, proceeding in a manner similar to \eqref{fS}, the upper bound in inequality \eqref{L-large} becomes  
	\begin{equation*} 
		\frac{C_{\nu , R}}{ L^{1- \gamma } } \sup_{n } \int_{0}^{T} dt \iint_{ \mathbb{R}^3_x \times \mathbb{R}^3_{v_*} } (f^n_* + f_* ) ( 1+ |v_* |^2 ) \mathbf{1}_{ \{ | v_*| > L \}} d x dv_* . 
	\end{equation*}
Although uniform bounds  on the second moment are available in  \eqref{FB2}, the lack of estimates for moments higher than the second order prevents us from determining the decay rate of the tail integral in the upper bound. Consequently, the convergence of this term as $L \to \infty$ cannot be guaranteed in the critical regime $\gamma = 1$.
\end{remark}
\section{Convergence of the renormalized collision operator $ ( \mathcal{R}_3 )^n $}\label{Sec:Cnv-R3}  
 In Section \ref{Sec-cnv}, we prove the convergence of the renormalized approximate operators $ ( \mathcal{R}_1 )^n  $ and $ ( \mathcal{R}_2 )^n  $.  Therefore, the convergence of $ ( \mathcal{R}_3 )^n  $ remains the only unresolved issue. Establishing this convergence will complete the proof that the strong $L^1$-limit $f$   of the approximating sequence  $f^n$ is  a renormalized solution of the Boltzmann equation \eqref{Boltz}.
We first recall the definition of $ ( \mathcal{R}_3 ) $, which is given by
\begin{align*}
	(\mathcal{R}_3) = -\iint_{\mathbb{R}^3 _{v_*}\times \S^{2}}  B f'_* \Gamma (f, f')d v_*   d \sigma  = -\iint_{\mathbb{R}^3 _{v_*} \times \S^{2}}  B f'_* \left\{ \beta _\delta ( f' ) - \beta _\delta ( f ) - \beta _\delta ' ( f ) ( f' - f ) \right\}  d v_* d \sigma .
\end{align*} 
Since $\beta_\delta ( s ) =  s/ (1+\delta s )$ with $ \delta>0  $, it follows that  
\begin{align*}
	(\mathcal{R}_3) =  \iint_{\mathbb{R}^3  _{v_*}\times \S^{2}}  B f'_* \frac{ \delta (f'-f )^2   }{( 1+ \delta f ) ^2 ( 1+\delta  f')} d v_* d \sigma .
\end{align*}
For simplicity, we assume that $\delta = 1$ in this subsection. 

Since $f^n$ converges strongly to $f$ in $L^1([0,T] \times \mathbb{R}^3_x \times \mathbb{R}^3_v)$, Fatou's lemma  (Lemma \ref{fatou}) implies that
\begin{align*}
	\int_0^T dt \iint_{\mathbb{R}^3_x \times \mathbb{R}^3 _v } (\mathcal{R}_3) \, dx  dv \leq \liminf_{n \rightarrow \infty} \int_0^T dt \iint_{\mathbb{R}^3 _x \times \mathbb{R}^3_v }\left( 1+ \frac{1}{n} \int_{\R^3_v} f^n d v \right)^{-1}  (\mathcal{R}_3)^n \, dx   dv.
\end{align*}
This corresponds exactly to the convergence result for $(\mathcal{R}_3)^n$ established by Alexandre and Villani \cite{Alex-Vill-2002-CPAM}. This result also implies the existence of a non-negative defect measure. To eliminate the defect measures, we first establish the crucial  average time-velocity estimate for $(\mathcal{R}_3)^n$. 

\begin{lemma}[Average time-velocity estimate for $(\mathcal{R}_3)^n$]\label{Lmm-ATV} Let $0 \leq \gamma \leq 1$. Then, for almost every $x \in \mathbb{R}^3_x$, the term $(\mathcal{R}_3)^n$ satisfies the estimate 
	\begin{align}\label{R3n-tv}
		\int_{0}^{T} dt \int_{   \R ^3_v } ( \mathcal{R}_3)^n d v 
		\leq & C_{\nu ,R} \sup_{ t \in [0, T]} \int_{  {\R}^3_v } f^n dv  \int_{   \R ^3_{v_*}  } f^n ( 1+ |v_*|^2) dv_* + \frac{1}{2} \int_{0}^{T} dt \int_{\mathbb{R}^{ 3} _v } e^n dv  
	\end{align}
 where $C_{\nu,R} > 0$ is independent of $x$, and $e^n$ is defined in \eqref{d-en}.
\end{lemma}

\begin{proof} 
	First, the basic inequality $(x-y) \log \frac{x}{y} \geq 4(\sqrt{x} - \sqrt{y})^2$ shows that for almost all $(t,x) \in [0,T] \times B_R(x)$, we have
	\begin{equation}\label{sq-e}
		\begin{aligned}
			& \iiint_{ \mathbb{R}^{ 3} _v \times \mathbb{R}^{ 3}_{v_* } \times  \S^{2} } B_n ( v-v_*, \sigma ) \left( \sqrt{ {f ^n}' {f_*^n}'  } - \sqrt{ {f^n } {f^n_*}   } \right) ^2 d v dv_* d \sigma \\
			\leq & \frac{1}{4}   \iiint_{ \mathbb{R}^{ 3} _v \times \mathbb{R}^{ 3} _{v_* } \times  \S^{2} } B_n ( v-v_*, \sigma ) (  {f^n}' {f_*^n}'   -   {f^n } {f^n_*}  ) \log \frac{  {f ^n}' {f_*^n}'   }{{f^n } {f^n_*} } dv dv_* d \sigma = \int_{\mathbb{R}^{ 3}_v  } e^n dv.
		\end{aligned}
	\end{equation}
	Expanding the squared term on the left-hand side of the above inequality   yields 
	\begin{equation}\label{sqrt-fn} 
		\begin{aligned}
			\left( \sqrt{ {f ^n}' {f_*^n}'  } - \sqrt{ {f^n } {f^n_*}   } \right) ^2  = &\frac{1}{2} ( {f ^n}' -f^n   ) ( {f_*^n}'  - f^n_*  )  +  \frac{1}{4}    \left(\sqrt{{f ^n}' }  + \sqrt{f^n} \right) ^2  \left( \sqrt{ {f_*^n}'  } -\sqrt{ f^n_*}   \right) ^2        \\
			& + \frac{1}{4}    \left(\sqrt{ {f ^n}' }  - \sqrt{f^n} \right) ^2 \left( \sqrt{ {f_*^n}' } + \sqrt{ f^n_*}  \right) ^2 .
		\end{aligned}
	\end{equation}
	Combining this with \eqref{sq-e}, we derive that
	\begin{equation}\label{sq-e1}
		\begin{aligned}
			& \iiint_{ \mathbb{R}^{ 3}_v  \times \mathbb{R}^{ 3} _{v_* } \times  \S^{2} } B_n   \left(\sqrt{ {f ^n}' }  - \sqrt{f^n} \right) ^2 \left( \sqrt{ {f_*^n}' } + \sqrt{ f^n_*}  \right) ^2  d v dv_* d \sigma \\
			& +  \iiint_{ \mathbb{R}^{ 3}_v  \times \mathbb{R}^{ 3}_{v_* } \times  \S^{2} } B_n   ( {f ^n}' -f^n   ) ( {f_*^n}'  - f^n_*  )    d v dv_* d \sigma \leq \int_{\mathbb{R}^{ 3} _v } e^n dv,
		\end{aligned}
	\end{equation}
	where we have discarded the non-negative term $ \left(\sqrt{{f ^n}' }  + \sqrt{f^n} \right) ^2  \left( \sqrt{ {f_*^n}'  } -\sqrt{ f^n_*}   \right) ^2   $.
	
	For the second term on the left-hand side of the inequality \eqref{sq-e1}, by performing the change of variables $(v,v_*) \to ( v',v'_*)$ we obtain
	\begin{equation*}
		\begin{aligned}
			\iiint_{ \mathbb{R}^{ 3}_v  \times \mathbb{R}^{ 3}_{v_* } \times  \S^{2} } B_n   ( {f ^n}' -f^n   ) ( {f_*^n}'  - f^n_*  )    d v dv_* d \sigma  
			=  &-2 \iiint_{ \mathbb{R}^{ 3}_v  \times \mathbb{R}^{ 3}_{v_* } \times  \S^{2} } B_n f^n  ( {f_*^n}'  - f^n_*  )  d v dv_* d \sigma \\
			= &-2 \iint_{ \mathbb{R}^{ 3}_v  \times \mathbb{R}^{ 3}  _{v_* } } f^n  f^n_* S^n ( |v-v_*| ) dvdv_* .
		\end{aligned}
	\end{equation*}
	Next we address the first term on the left-hand side of inequality \eqref{sq-e1}. We recall the elementary inequalities
	\begin{align*}
		\left( \sqrt{ {f ^n}' } + \sqrt{ f^n }  \right) ^2   \leq 2 ( f^n + {f ^n}' ) \leq 2 ( 1+ {f ^n}'  ) ( 1+ f ^n )^2,
	\end{align*}
	together with the identity
	\begin{align*}
		({f ^n}'-f^n )^2  =  \left( \sqrt{ {f ^n}' } + \sqrt{ f^n }  \right) ^2  \left( \sqrt{ {f ^n}' } - \sqrt{ f^n }  \right) ^2.
	\end{align*}
	Combining these relations,  we obtain
	\begin{align}\label{fs}
		f^n_* \frac{  ({f ^n}'-f^n )^2   }{( 1+   f^n ) ^2 ( 1+  {f ^n}')} \leq &   \frac{1}{2}   f^n_*(\sqrt{ {f ^n}' }  - \sqrt{f^n} ) ^2 \leq   \frac{1}{2}  \left(\sqrt{ {f ^n}' }  - \sqrt{f^n} \right) ^2 \left( \sqrt{ {f_*^n}' } + \sqrt{ f^n_*}  \right) ^2.
	\end{align} 
	Thus, combining \eqref{Sn-gamma} with the condition $\gamma \in [0,1]$, we derive that
	\begin{equation}\label{R3n-1}
		\begin{aligned}
			\int_{0}^{T} dt \int_{   \R ^3 _v } ( \mathcal{R}_3)^n d v = & \int_0^T d t \iiint_{ \mathbb{R}^3 _v \times \mathbb{R}^3 _{v_* } \times \S^{2}} B_n f^n_* \frac{   ({f ^n}'-f^n )^2   }{( 1+   f^n ) ^2 ( 1+  {f ^n}')} d v dv_* d \sigma \\ 
			\leq  &\frac{1}{2} \int_{0}^{T} dt \iiint_{ \mathbb{R}^{ 3} _v \times \mathbb{R}^{ 3} _{v_* } \times  \S^{2} } B_n    \left(\sqrt{ {f ^n}' }  - \sqrt{f^n} \right) ^2 \left( \sqrt{ {f_*^n}' } + \sqrt{ f^n_*}  \right) ^2   d v dv_* d \sigma \\
			\leq &   \int_{0}^{T} dt \iint_{ \mathbb{R}^{ 3}_v  \times \mathbb{R}^{ 3}  _{v_* } } f^n  f^n_* S^n ( |v-v_*| ) dvdv_* + \frac{1}{2} \int_{0}^{T} dt \int_{\mathbb{R}^{ 3}_v  } e^n dv\\
			\leq &C _\nu  \int_{0}^{T} dt \iint_{ \mathbb{R}^{ 3}_v  \times \mathbb{R}^{ 3}  _{v_* } } f^n  f^n_*    |v-v_*|^\gamma dvdv_* + \frac{1}{2} \int_{0}^{T} dt \int_{\mathbb{R}^{ 3}_v  } e^n dv\\
			\leq & C_{\nu ,R} \sup_{ t \in [0, T]} \int_{  {\R}^3_v } f^n dv  \int_{   \R ^3_{v_*}  } f^n ( 1+ |v_*|^2) dv_* + \frac{1}{2} \int_{0}^{T} dt \int_{\mathbb{R}^{ 3} _v } e^n dv.
		\end{aligned}
	\end{equation} 
	Therefore, the average time-velocity estimate for $(\mathcal{R}_3)^n$  is established.
\end{proof}

\begin{remark}\label{Rmk-6.2}
	For hard potentials where $0 \leq \gamma \leq 1$, the inequality \eqref{R3n-1} holds, thereby yielding the desired average time-velocity bound. However, when $-3 < \gamma < 0$, the above argument only provides the weaker bound 
	$$
	C _\nu \int_{0}^{T}   dt \iint_{\R^3_v \times \R^3_{v_*}} f^n f^n_* |v-v_*|^\gamma \, dv \, dv_* + \frac{1}{2} \int_{0}^{T} \! dt \int_{\R^3_v} e^n \, dv .
	$$
	As shown in Example \ref{soft-f}, the integral $\int_{0}^{T}   dt \iint_{\R^3_v \times \R^3_{v_*}} f^n f^n_* |v-v_*|^\gamma \, dv \, dv_*$ may diverge due to the singularity of the collision kernel at small relative velocities. Consequently, the average time-velocity estimate generally fails in the case of soft potentials. Therefore, the approach developed for the hard potential model in	this paper does not apply to the soft potential model.
\end{remark}

Having established the corresponding bounds, we now investigate the equi-integrability of $(\mathcal{R}_3)^n$. Given that the function $\beta_\delta ( s ) $ satisfies the renormalization condition \eqref{beta}  and $f^n$ is the unique solution to the approximate equation \eqref{approximation} constructed in Section \ref{Sec-GEAE}, it follows that $ \beta_\delta ( f^n) $  satisfies 
\begin{equation}\label{eq6.75}
	\frac{\partial }{\partial t} \beta_\delta ( f^n) + v \cdot \nabla_x \beta_\delta ( f^n)  = \beta_\delta ' ( f^n ) \tilde{Q}_n (f^n,f^n),
\end{equation}
where  
\begin{align*}
	\beta_\delta ' ( f^n ) \tilde{Q}_n (f^n,f^n)  &= \left( 1+ \frac{1}{n} \int_{\R^3_v} f^n d v \right)^{-1}   \beta_\delta ' ( f^n ) Q _n (f^n,f^n)\\  
	&= \left( 1+ \frac{1}{n} \int_{\R^3_v } f^n d v \right)^{-1}  \left\{   (\mathcal{R}_1)^n + (\mathcal{R}_2)^n +(\mathcal{R}_3)^n \right\}.
\end{align*}
\begin{lemma}\label{equi-R3}
	For any function $\varphi ( t,x , v) \in L^\infty \left( [0,T] \times \R^3_x ; C^{\alpha}_c (\R^3_v )\right)$ with compact  support in  $[0,T] \times B_R(x) \times  B_R (v)$, the sequence $	\int_{0}^{T} d t \int_{ \R^3_v } \left( 1+ \frac{1}{n} \int_{\R^3_v } f^n d v \right)^{-1} ( \mathcal{R}_3)^n   \varphi d v$ is equi-integrable on $B_R (x)$. 
\end{lemma}
\begin{proof}	 
	Following a procedure similar to that in Lemma \ref{p-R3}, we multiply both sides of the equation \eqref{eq6.75} by a  function  $\varphi ( t,x , v) \in L^\infty \left( [0,T] \times \R^3_x ; C^{\alpha}_c (\R^3_v )\right)$  with compact  support in  $[0,T] \times B_R (x)\times  B_R(v)$, and integrate it over $(t,x,v) \in ([ 0,T ] \times \R^3_x \times \R^3_v  )$. This yields  
	\begin{equation}\label{rs-n}
		\begin{aligned}
			&\iint_{ \mathbb{R}^3_x \times \mathbb{R}^3 _v} \beta_\delta ( f^n )( T,\cdot ) \varphi d x d v - \iint_{ \mathbb{R}^3 _x\times \mathbb{R}^3_v } \beta_\delta ( f^n ) ( 0, \cdot ) \varphi d x d v \\
			= 	& \int_{0}^{T} d t \iint_{ \mathbb{R}^3 _x\times \mathbb{R}^3 _v} \left( 1+ \frac{1}{n} \int_{\R^3_v } f^n d v \right)^{-1}  \left(  (\mathcal{R}_1 )^n + (\mathcal{R}_2 )^n + (\mathcal{R}_3 )^n \right) \varphi d x d v,
		\end{aligned}
	\end{equation}
	which implies that
	\begin{equation}\label{R3n-bdd}
		\begin{aligned}
			&\int_{0}^{T} d t \iint_{ \mathbb{R}^3_x \times \mathbb{R}^3 _v } \left( 1+ \frac{1}{n} \int_{\R^3_v } f^n d v \right)^{-1}  (\mathcal{R}_3 )^n \varphi d x d v \\ \leq &  \left|  \iint_{ \mathbb{R}^3_x \times \mathbb{R}^3_v  }  (\beta_\delta ( f^n )( T,\cdot ) + \beta_\delta ( f^n ) ( 0, \cdot )) \varphi d x d v \right|  \\
			& + \left|  \int_{0}^{T} d t \iint_{ \mathbb{R}^3 _x\times \mathbb{R}^3 _v }\left( 1+ \frac{1}{n} \int_{\R^3_v } f^n d v \right)^{-1}   ( (\mathcal{R}_1 )^n + (\mathcal{R}_2 )^n ) \varphi dx dv  \right|.
		\end{aligned}
	\end{equation}
	
	Recall that Sections \ref{Subsec:Cnv-R1} and \ref{Subsec:Cnv-R2} established the weak convergence of the integrals $ \int_{\mathbb{R}^3_v}\big( 1+ \frac{1}{n} \int_{\mathbb{R}^3_v} f^n \, dv \big)^{-1}  (\mathcal{R}_1)^n \varphi \, dv$ and  $ \int_{  \R^3_v  } \left( 1+ \frac{1}{n} \int_{\R^3_v } f^n d v \right)^{-1}( \mathcal{R} _2 )^n \varphi d v  $, respectively. Applying the Dunford-Pettis theorem (Lemma \ref{theorem-dunford})  confirms that these sequences are equi-integrable on $[0,T] \times \R^3$. Furthermore, the bound $\beta_\delta (t) \leq t$ and the weak convergence of $ f^n $ imply that  $\{\beta_\delta (f^n)\}_{n\geq 1}$ is also equi-integrable. Thus  for every $\epsilon >0 $ there exists $\delta >0 $ such that for any measurable set $E  \subset   B_R (x)$ with $\mu (E ) < \delta $, we have
	\begin{equation*}
		\begin{aligned}
			& \int_{0}^{T} d t \iint_{E \times \R^3_v }  \left( 1+ \frac{1}{n} \int_{\R^3_v} f^n d v \right)^{-1}   (\mathcal{R}_3 )^n \varphi  d x d v \\ \leq &  \int_{0}^{T} d t   \iint_{ E \times \R^3_v  } \left| (\beta_\delta ( f^n )( T,\cdot ) + \beta_\delta ( f^n ) ( 0, \cdot )) \varphi \right| d x d v   \\
			& + \int_{0}^{T} d t   \iint_{E \times \R^3 _v }  \left| \left( 1+ \frac{1}{n} \int_{\R^3_v } f^n d v \right)^{-1}   ( (\mathcal{R}_1 )^n + (\mathcal{R}_2 )^n ) \varphi \right| dx dv< \epsilon .
		\end{aligned}
	\end{equation*}
Thus, the sequence $	\int_{0}^{T} d t \int_{ \R^3_v } \left( 1+ \frac{1}{n} \int_{\R^3_v } f^n d v \right)^{-1} ( \mathcal{R}_3)^n   \varphi d v$ is equi-integrable on $B_R (x)$. 
\end{proof}
 Next, to address the singularity of the collision kernel at $ \theta = 0 $, we investigate the convergence of $ (\mathcal{R}_3)^n$ by decomposing the spherical integral into two regions: one near the singularity and one away from it.  Specifically, we define the truncated kernels as
\begin{equation}\label{Bep}
\begin{aligned}
 B_{n}^k (|z|,\cos \theta ) = B(|z|,\cos \theta )  \mathbf{1}_{ \frac{1}{k} \leq \theta \leq \frac{\pi}{2} },  &\quad B_{ k} (|z|,\cos \theta ) =   B ^k (|z|,\cos \theta ) \mathbf{1}_{ \frac{1}{k} \leq \theta \leq \frac{\pi}{2} },  \\
 B_{n,k} (|z|,\cos \theta ) = B_{n} (|z|,\cos \theta )  \mathbf{1}_{ 0 \leq \theta \leq \frac{1}{k}}, & \quad B_{ k} (|z|,\cos \theta ) =B   (|z|,\cos \theta ) \mathbf{1}_{ 0 \leq \theta \leq \frac{1}{k}}.
\end{aligned}
\end{equation}
Let $(\mathcal{R}_3) ^{n, k}$  and $(\mathcal{R}_3) ^{n  }_k$ denote the renormalized collision operators corresponding to the kernels $ B_{n}^k $ and $B_{n,k} $, respectively. The operators $(\mathcal{R}_3) ^{ k}$ and $(\mathcal{R}_3)  _k$ are defined analogously. Furthermore, we set
 \begin{equation}\label{ML}
 \begin{aligned} 
 	M^n (  x) = \sup_{ t \in [0, T]} \int_{\R^3 _v} f^n ( 1+ |v|^2 )d v+ \int_{0}^{T} dt \int_{   \R ^3_v  } e^n d v , \quad &	 N^n_L = \left\{  x \in B_R(x ) : M^n (  x) >L \right\}, 	 \\
M  (  x) = \sup_{ t \in [0, T]} \int_{\R^3 _v} f  ( 1+ |v|^2 )d v+ \int_{0}^{T} dt \int_{   \R ^3_v  } e  d v,  \quad & N  _L = \left\{  x \in B_R(x ) : M  (  x) >L \right\}.
 	 \end{aligned}
 \end{equation}
 Then by Chebyshev's inequality, we have
 \begin{align*}
 	\text{meas}(N^n_L )  \leq& \frac{1}{L} \sup_{ t \in [0, T]} \iint_{ \R ^3 _x\times \R ^3 _v } f^n ( 1+ |v|^2)  d x d v  +\int_{0}^{T} dt \iint_{   \R ^3_x  \times \R^3 _v } e^n d x  d v \leq \frac{C}{L} ,\\
 		\text{meas}(N _L )  \leq& \frac{1}{L} \sup_{ t \in [0, T]} \iint_{ \R ^3 _x\times \R ^3 _v } f  ( 1+ |v|^2)  d x d v  +\int_{0}^{T} dt \iint_{   \R ^3_x  \times \R^3 _v } e  d x  d v \leq \frac{C}{L}.
 \end{align*} 
 
 With these truncations in place, we first establish the convergence of the non-singular part in the following lemma. 
 
\begin{lemma}\label{L-R3nc}
There exists a sufficiently large constant $L>0$ such that
\begin{align}\label{R3k}
\lim\limits_{n\to \infty} \int_0^T dt \iint_{\mathbb{R}^3_x \times \mathbb{R}^3_v }  \left(1 + \frac{1}{n} \int_{\mathbb{R}^3_v} f^n \, dv \right)^{-1}  (\mathcal{R}_3) ^{n, k} \varphi \mathbf{1}_{ (N_L^n)^c}  dx dv = \int_0^T dt \iint_{\mathbb{R}^3_x \times \mathbb{R}^3_v}  (\mathcal{R}_3) ^k \varphi \mathbf{1}_{(N_L)^c} \, dx dv.
\end{align}
\end{lemma}
\begin{proof}
The definition of $(\mathcal{R}_3) ^{n, k}$ implies that for any fixed $t,x,v$, the operator satisfies the pointwise estimate
 \begin{equation*}
 \begin{aligned}  
 (\mathcal{R}_3) ^{n, k}  = -  \iint_{\mathbb{R}^3_{v_*} \times \S^2} B_n^k ( v-v_*, \sigma )f_*^n \Gamma(f^n, {f ^n}') \, dv_* d\sigma  
  \leq    C(f^n + {f^n}') \int_{\mathbb{R}^3_{v_*}} f_*^n \, dv_*.
 \end{aligned}
 \end{equation*} 
Combining this estimate with the convergence $f^n \to f$ a.e. in $(t,x,v)$ and $B^k_n \to B^k$ a.e. in $(v_*,\sigma)$, we obtain the pointwise convergence 
 \begin{equation*}
 (\mathcal{R}_3) ^{n, k} \to (\mathcal{R}_3) ^{  k}\quad \text{a.e. } (t, x, v) \in [0, T] \times (N_L^n \cup N_L)^c   \times B_R(v).
 \end{equation*}
Furthermore, Lemma \ref{Lmm-ATV} provides the uniform bound
 \begin{equation*}
 \int_0^T dt \int_{  {\R}^3_v }  \left(1 + \frac{1}{n} \int_{ {\R}^3} f^n \, dv \right)^{-1}  (\mathcal{R}_3) ^{n, k} \varphi \mathbf{1}_{ (N_L^n)^c}   dv \leq C_L.
 \end{equation*} 
 Consequently, by Lebesgue's Dominated Convergence Theorem (Lemma \ref{LCT}), for any $\delta > 0$   there exists $N \in \mathbb{N}_+$ such that for all $n > N$, the following inequality holds
 \begin{equation*}
 	\left| \int_0^T dt \iint_{\R^3_x \times \R^3_v} \left[ \left(1 + \frac{1}{n} \int_{\R^3_v} f^n \, dv \right)^{-1} (\mathcal{R}_3)^{n, k} - (\mathcal{R}_3)^k \right] \varphi \mathbf{1}_{(N_L^n \cup N_L)^c} \, dx dv \right| < \frac{\delta}{3}.
 \end{equation*} 
 Given that $\text{meas}(N_L) \le \frac{C}{L}$, and invoking the  equi-integrability of the term $\int_0^T dt \int_{\mathbb{R}^3_v} \big(1 + \frac{1}{n} \int_{ {\R}^3} f^n \, dv \big)^{-1}  (\mathcal{R}_3) ^{n } \varphi   dv$ on $B_R(x)$, we can choose $L > 0$ sufficiently large such that 
 \begin{align*}
 &\sup_{n \ge 1} 	\int_0^T dt \iint_{\mathbb{R}^3_x \times \mathbb{R}^3_v}  \left(1 + \frac{1}{n} \int_{\mathbb{R}^3_v} f^n \, dv \right)^{-1} (\mathcal{R}_3)^{n, k} \varphi \mathbf{1}_{(N_L^n)^c \cap N_L} \, dx dv \\
 	\leq & \sup_{n \ge 1}  \int_0^T dt \iint_{\mathbb{R}^3_x \times \mathbb{R}^3_v} \left(1 + \frac{1}{n} \int_{\mathbb{R}^3_v} f^n \, dv \right)^{-1} (\mathcal{R}_3)^{n } \varphi \mathbf{1}_{N_L} \, dx dv < \frac{\delta}{3}.
 \end{align*}
 Similarly, we have  
 \begin{equation*}
\sup_{n \ge 1}  \int_0^T dt \iint_{\mathbb{R}^3_x \times \mathbb{R}^3_v}  (\mathcal{R}_3)^k \varphi \mathbf{1}_{(N_L)^c \cap N_L^n} \, dx dv < \frac{\delta}{3}.
 \end{equation*}
 
 Therefore, for any $\delta > 0$, there exists $N \in \mathbb{N}_+$ and $L > 0$ such that for $n > N$, we have
 \begin{align*}
 	&\Bigg| \int_0^T dt \iint_{\mathbb{R}^3_x \times \mathbb{R}^3_v}  \left(1 + \frac{1}{n} \int_{\mathbb{R}^3_v} f^n \, dv \right)^{-1} (\mathcal{R}_3)^{n, k} \varphi \mathbf{1}_{(N_L^n)^c} \, dx dv - \int_0^T dt \iint_{\mathbb{R}^3_x \times \mathbb{R}^3_v} (\mathcal{R}_3)^k \varphi \mathbf{1}_{(N_L)^c} \, dx dv \Bigg| \\
 	\le & 	\left| \int_0^T dt \iint_{\R^3_x \times \R^3_v} \left[ \left(1 + \frac{1}{n} \int_{\R^3_v} f^n \, dv \right)^{-1} (\mathcal{R}_3)^{n, k} - (\mathcal{R}_3)^k \right] \varphi \mathbf{1}_{(N_L^n)^c \cap (N_L)^c} \, dx dv \right| \\
 	& +\sup_{n \ge 1}  \int_0^T dt \iint_{\mathbb{R}^3_x \times \mathbb{R}^3_v}  \left(1 + \frac{1}{n} \int_{\mathbb{R}^3_v} f^n \, dv \right)^{-1} (\mathcal{R}_3)^{n, k} \varphi \mathbf{1}_{(N_L^n)^c \cap N_L} \, dx dv \\
 	& + \sup_{n \ge 1}  \int_0^T dt \iint_{\mathbb{R}^3_x \times \mathbb{R}^3_v}  (\mathcal{R}_3)^k \varphi \mathbf{1}_{(N_L)^c \cap N_L^n} \, dx dv 
 	<    \delta.
 \end{align*}
This completes the proof  of Lemma \ref{L-R3nc}. 
\end{proof}
Having established the convergence of the non-singular part in Lemma \ref{L-R3nc}, we now turn to the singular part. We begin by presenting an estimate that plays a crucial role in proving the uniform convergence of the singular integrals.

\begin{lemma}[Uniform smallness estimate for small angular deflections $\theta$]\label{LBep}
Let $0 \leq \gamma \leq 1$, $\nu \in [0, 1/2]$ defined in \eqref{nu}, and $R > 0$. For the collision kernel $B_k$ defined in \eqref{Bep}, the following estimate holds
\begin{equation}\label{Hnu}
	\iint_{B_R(v)\times \S ^2} B_k ( v-v_*, \sigma )  \left( g(v') - g(v) \right)^2   \, d\sigma dv \le  \frac{C}{k} (1+|v_*|)^{1+\gamma+\nu} \|g\|^2_{{H^{\nu/2}} {(B_R (v) )}}
\end{equation}
for all $k \geq 1$.
\end{lemma}

\begin{proof}  
 We fix $( v ,v_*) \in  \R^3_v \times \R^3_{v_*}$ and denote the relative velocity by $u = v - v_*$. The scattering angle $\theta \in (0, 1/k ]$ is defined by the relation $  u \cdot \sigma =  \cos\theta |u| $.
 Recalling the post-collisional velocity formula $	v' = \frac{v+v_*}{2} + \frac{|v-v_*|}{2}\sigma $, it follows immediately that the velocity change $h:= v'-v$ can be expressed as
 \begin{equation*}
  h = \frac{v_*-v}{2} + \frac{|v-v_*|}{2}\sigma = \frac{|u|}{2}\left( \sigma - \frac{u}{|u|} \right) .
 \end{equation*}
 From this representation, we obtain $ |h| = |u|\sin\left(\frac{\theta}{ 2} \right) $. We now turn to a more detailed decomposition of $h$. To this end, we decompose the unit vector   $\sigma$ into its components parallel and orthogonal to $u$, namely,
 \begin{equation*}
 	\sigma = \cos\theta \frac{u}{|u|} + \sin\theta \,  \omega,
 \end{equation*}
where $\omega \in \S^2$ satisfies $\omega \perp u$.  Substituting this decomposition into the expression for $h$,  we obtain 
 \begin{equation*}
 	h =   \frac{|u|}{2}\left( (\cos\theta - 1)\frac{u}{|u|} + \sin\theta \, \omega \right).
 \end{equation*}
 Accordingly, it is natural to decompose   $h$ into a component parallel $u$ to and a component perpendicular to  $u$:
 $$
 h=h_1+h_2,
 $$
 where
$$
	h_1  = \frac{|u|}{2}(\cos\theta - 1)\frac{u}{|u|} = -u    \sin^2\left(\frac{\theta}{2}\right)     , 
$$ 
and
$$
	h_2  = \frac{|u|}{2}\sin\theta \, \omega = |u|\sin\left(\frac{\theta}{2}\right)\cos \left(\frac{\theta}{2}\right)  \, \omega.
$$ 
 In particular, this decomposition implies that  $|h_2|  =\frac{1}{2} |u|  \sin  \theta $. 
 Moreover, since  $\theta \in [0, 1 /k ]$, we can use the inequality $\sin \theta \leq \theta$ to conclude that 
 $|h_2| \leq  \frac{1}{2k}|u| $. Next, we introduce the radial variable   $ r := |h_2|$. In the plane orthogonal to $u$, the corresponding measure element takes the form $dh_2 = r \, dr \, d\phi$, where $\phi$ denotes the azimuthal angle. At the same time, the above relations allow us to rewrite the components of $\sigma $ in terms of $h_1$ and $h_2$ as 
 \begin{equation*}
 	\sin\theta \, \omega = \frac{2}{|u|} h_2, \quad \cos\theta \frac{u}{|u|} = \frac{2}{|u|}h_1 + \frac{u}{|u|}.
 \end{equation*}

 We now proceed to the change of variables from the spherical measure $d\sigma$  to the perpendicular velocity variable  $h_2$. Recalling that 
 \begin{align*}
 d\sigma = \sin \theta \, d \theta \, d \phi  \quad \text{ and } r = \frac{|u|}{2}\sin\theta,
 \end{align*}
 we deduce that
  $$
 dr = \frac{|u|}{2}\cos\theta \, d\theta ,\quad d\theta = \frac{2}{|u|\cos\theta} dr.
 $$
 Furthermore, using the identities $ \sin\theta = \frac{2r}{|u|}, $ and $\cos\theta = \sqrt{1 - \frac{4r^2}{|u|^2}}$, 
 we can  rewrite the surface measure $d\sigma$ in terms of $r$, and obtain
 \begin{align*}
  d\sigma = \sin \theta \, d \theta \, d \phi = \frac{4}{|u|^2} \frac{ r}{\sqrt{1 - \frac{4r^2}{|u|^2}}  } d r d \phi .
 \end{align*}
Finally, observing that $d h_2 = r dr d\phi $, we conclude that the Jacobian of the transformation is given by  
 \begin{equation*}
 	d\sigma = \frac{4}{|u|^2} \frac{1}{\sqrt{1 - \frac{4}{|u|^2}r^2 }} \, dh_2.
 \end{equation*}
 
 Under the singular kernel assumption
 \begin{align*}
  b(\cos\theta) \sim K \theta^{-2-\nu} , \quad \text{as }   \theta \to 0^+ ,
 \end{align*}
we note that, for grazing collisions,   $|h_2| \approx \frac{1}{2}|u|\theta$.
  Consequently, $as  \theta \to 0^+$, the kernel admits the asymptotic representation
 \begin{equation*}
 	b(\cos\theta) \sim K \theta^{-2-\nu} \sim \left( \frac{|h_2|}{|u|} \right)^{-2-\nu} .
 \end{equation*}
 Let $r_k := \frac{1}{2}|u|\sin\frac{1 }{k}$. We now estimate the collision integral associated with $B_k$.
 Using the above asymptotic behavior together with the change of variables established previously, we obtain  
 \begin{align*}
 	& 	\int_{  \S ^2} B_k ( v-v_*, \sigma )  \left( g(v') - g(v) \right)^2   \, d\sigma \\
 	\leq &C_{K,\nu} \int_{|h_2| < r_k} |u|^\gamma \frac{(g(v+h) - g(v))^2}{|h_2|^{2+\nu}} \frac{|u|^{2+\nu}}{|u|^2} \frac{1}{\sqrt{1 - \frac{4r^2}{|u|^2}}} \, dh_2 \\
 	\leq& C'_{K,\nu} \int_{|h_2| < r_k} |u|^{\gamma+\nu} \frac{(g(v+h) - g(v))^2}{|h|^{3+\nu}} \cdot \frac{|h|^{3+\nu}}{|h_2|^{2+\nu}} \, dh_2  .
 \end{align*} 
 Here, we observe the following geometric bound for the ratio of the norms
 \begin{equation*}
 	\frac{|h|^{3+\nu}}{|h_2|^{2+\nu}} = \frac{\left( |u|\sin\frac{\theta}{2} \right)^{3+\nu}}{\left( |u|\sin\frac{\theta}{2} \cos\frac{\theta}{2} \right)^{2+\nu}} \leq C |u| \sin\frac{\theta}{2} \leq C |u|\frac{1 }{k},
 \end{equation*}
where we have used the restriction $\theta \in (0, 1/k] $. Substituting this bound into the previous estimate yields
 \begin{align*}
 	&\int_{  \S ^2} B_k ( v-v_*, \sigma )  \left( g(v') - g(v) \right)^2   \, d\sigma  \\
 	\leq& \frac{1 }{k} C''_{K,\nu} |u|^{1+\gamma+\nu} \int_{|h_2| < r_k} \frac{(g(v+h) - g(v))^2}{|h|^{3+\nu}} \, dh_2   \\
 	\leq &\frac{1 }{k} \tilde{C}_{K,\nu} |u|^{1+\gamma+\nu} \int_{|h| < \sqrt{2}r_k} \frac{(g(v+h) - g(v))^2}{|h|^{3+\nu}} \, dh.
 \end{align*}
Given that  $|v|\leq R $, we have the estimate $|u|^{1+\gamma+\nu} \leq C_R ( 1+ |v_*|^{1+\gamma + \nu })$. Consequently,
 \begin{align*}
 &\iint_{B_R(v) \times   \S ^2} B_k ( v-v_*, \sigma )  \left( g(v') - g(v) \right)^2   \, d\sigma dv \\
 \leq & \frac{1 }{k} C_{K,\nu,R} \left( 1+ |v_*|^{1+\gamma + \nu } \right) \int_{ B_R(v) } dv\int_{ |h| \leq \sqrt{2}r_k  }\frac{(g(v+h) - g(v))^2}{|h|^{3+\nu}} \, dh .
 \end{align*}
Finally, recalling from Section 7.3 of \cite{GS-classical-2009} that the Sobolev norm admits the equivalent characterization  
\begin{align*}
\|f\|_{H^{s}(\mathbb{R}^3_v)}^2 \approx \|f\|_{L^2(\mathbb{R}^3_v)}^2 + \iint_{\mathbb{R}^3_v \times \mathbb{R}^3_{v'}} \frac{(f(v') - f(v))^2}{|v' - v |^{3+2s }} \mathbf{1}_{|v - v'| \le 1} \,  dvdv',
\end{align*}
we conclude that
 \begin{align*}
   \iint_{B_R(v) \times   \S ^2} B_k ( v-v_*, \sigma )  \left( g(v') - g(v) \right)^2   \, d\sigma dv  
  \leq   \frac{1 }{k} C_{K,\nu,R} \left( 1+ |v_*|^{1+\gamma + \nu } \right) \|g\|_{H^{\nu/2}(B_R (v))}^2 .
 \end{align*}
 This completes the proof  of Lemma \ref{LBep}. 
\end{proof}
Finally, combining the convergence of the non-singular part with the estimates for the singular part, we obtain the convergence of $( \mathcal{R}_3 )^n$.
 
\begin{lemma}\label{l-cnvR3}
	For any function $\varphi ( t,x , v) \in L^\infty \left( [0,T] \times \R^3_x ; C^{\alpha}_c (\R ^3_v )\right)$ with compact  support in  $[0,T] \times B_R(x) \times  B_R (v)$, we have
	\begin{align}\label{cnv-R3}
		\lim_{n \to \infty}  \int_{0}^{T} dt \iint_{ \mathbb{R}^3_x \times \mathbb{R}^3_v } \left( 1+ \frac{1}{n} \int_{\R^3_v} f^n d v \right)^{-1}  ( \mathcal{R}_3 )^n \varphi d x d v =  \int_{0}^{T} dt \iint_{ \mathbb{R}^3_x \times \mathbb{R}^3_v } ( \mathcal{R}_3 )  \varphi d x d v .
	\end{align}
\end{lemma}
\begin{proof} 
Analogous to the estimate in \eqref{R3n-1}, we bound $( \mathcal{R}_3 )_k^n$ by
	 \begin{align*}
		\int_{\mathbb{R}^3_v}  ( \mathcal{R}_3 )_k^n  \varphi \, dv 
		 \leq \frac{1}{2} \iiint_{\mathbb{R}^3_v \times \mathbb{R}^3_{v_*} \times \S^2} B_{n, k} f_*^n \left( \sqrt{f^{n \prime}} - \sqrt{f^n} \right)^2 \varphi \, dv \, dv_* \, d\sigma  .
	\end{align*} 
Recalling the parameters $\gamma$  and $\nu$  defined in  \eqref{B-gamma} and \eqref{nu},  we find that
 \begin{align*}
 	1+ \gamma + \nu = 1 + 1 - \frac{4}{s-1} + \frac{2}{s-1} = 2 -  \frac{2}{s-1} \leq 2.
 \end{align*}
 Combining this inequality with Lemma \ref{LBep} leads to 
 \begin{align*}
 	\int_{\mathbb{R}^3_v}  ( \mathcal{R}_3 )_k^n  \varphi \, dv 
 	&\leq \frac{1}{2}\int_{ \mathbb{R}^3_{v_*}}f_*^n dv_* \iint_{ B_R(v)   \times \S^2} B_{n, k}  \left( \sqrt{f^{n \prime}} - \sqrt{f^n} \right)^2 \varphi \, dv   d\sigma \\ 
 	&\leq \frac{1 }{k}   C_{k, v, R} \int_{\mathbb{R}^3_{v_*}} f_*^n (1 + |v_*|^{1+\gamma + \nu }) \, dv_* \cdot \| \sqrt{f^n} \|^2_{H^{\nu/2}(B_R (v))} \\
 		&\leq \frac{1 }{k}   C_{k, v, R} \int_{\mathbb{R}^3_{v_*}} f_*^n (1 + |v_*|^{2}) \, dv_* \cdot \| \sqrt{f^n} \|^2_{H^{\nu/2}(B_R (v))} .
 \end{align*}
Moreover, as shown in \cite{ADVW-bound-2000}, estimate \eqref{F-Zn} provides a fractional Sobolev control in the velocity variable, which implies that 
\begin{align*}
\| \sqrt{f^n} \|^2_{H^{\nu/2}(B_R (v))}	\leq  \frac{  C_0  J^2   (2\pi )^3   }{  \|f ^n \|_{L^1( B_R(v)) } \left( r_0 /2 \right)^\gamma    }  \left\{   \int_{  \mathbb{R}^{3}_v } e^n  ( t,x,v) \,   dv   +   \left( C_{\nu }  + \frac{ C_1 }{ r ^2 _0 } \right)  \|f ^n\|^2_{L^1_2 ( \R^3_v ) }   \right\}.
\end{align*}

Recalling the definition in \eqref{W-epsion}, the set $W_\epsilon$ is given by  
\begin{align*} 
	W_\epsilon = \left\{ (t,x ) \in [0,T] \times B_R(x): \int_{B_R(v)} g^2  d v >\epsilon \right\},
\end{align*}
where $g(t,x,v)$ denotes the weak limit of  $\sqrt{f^n}$ in $L^2 ([0,T] \times B_R(x) \times B_R(x))$. Due to the strong convergence of $f^n$, we identify $g=\sqrt{f}$ almost everywhere in $[0,T] \times B_R(x) \times B_R(x)$. Furthermore, by Lemma \ref{v-regular}, there exists a Borel set $U_\delta$ with measure less than $\delta$ such that for sufficiently large $n \in \mathbb{N}_+$, the  lower bound $\int_{B_R(v) } f^n \, dv \geq \epsilon /2$ holds on  $W_\epsilon \setminus U_\delta$.  Accordingly, we can find an integer   $N \in \mathbb{N}$ such that for all $n>N$, we have
$ 
\|f ^n \|_{L^1( B_R(v))} \geq \epsilon/2 $   on  $ W_\epsilon  \setminus  U_\delta $  and $  \|f ^n \|_{L^1( B_R(v))} \leq 2\epsilon \ $  a.e. on $ ([0,T] \times B_R(x)) \setminus W_\epsilon$. 
  Hence, for any fixed $L>0$ and any $n > N$, the integral over the set $W_\epsilon  \setminus (U_\delta \cup U_L^n)$ satisfies
 \begin{align*} 
 &  \int_0^T dt 	\iint_{\R^3_x \times \mathbb{R}^3_v} ( \mathcal{R}_3 )_k^n  \varphi \mathbf{1}_{W_\epsilon  \setminus (U_\delta \cup U_L^n)} \, dx dv  \\
 \leq & \frac{1 }{ k}   C_{k, v, R}  \int_0^T dt \int_{   B_R(x)} \bigg(   \int_{\mathbb{R}^3_{v_*}} f_*^n (1 + |v_*|^{2}) \, dv_* \cdot \| \sqrt{f^n} \|^2_{H^{\nu/2}(B_R (v))}  \bigg) \mathbf{1}_{W_\epsilon  \setminus (U_\delta \cup U_L^n)}  \, dx \\
  \leq& \frac{L }{  k}   C_{k, v, R} \int_0^T dt  \int_{B_R(x)}  \| \sqrt{f^n} \|^2_{H^{\nu/2}(B_R (v))}\mathbf{1}_{W_\epsilon  \setminus (U_\delta \cup U_L^n)} dx \\
  \leq &  \frac{L }{\epsilon  k} \tilde{C} \bigg( \int_0^T dt 	\iint_{\R^3_x \times \mathbb{R}^3_v} e^n dxdv + \bigg( C_{\nu }  + \frac{ C_1 }{ r ^2 _0 } \bigg) \iint_{\R^3_x \times \mathbb{R}^3_v}  f ^n (1+|v|^2) dxdv  \bigg) 
   \to 0 \quad ( k \to \infty ).
 \end{align*}
Here $\tilde{C}=  C_{k, v, R}  C_0 J^2 (2\pi)^3 (r_0/2)^{-\gamma}$. 
The second inequality follows from the definition of the set $ (U^n_L)^c$, while the last step is derived from the uniform lower bound of $\int_{B_R(v) } f^n \, dv$ on $W_\epsilon \setminus U_\delta$ combined with the estimate \eqref{F-Zn}.

 Furthermore, invoking the measure bound  $\text{meas}(U_\delta \cup U_L^n )\leq \delta+\frac{C}{L } $ and the equi-integrability of  $\int_0^T dt \int_{\mathbb{R}^3_v} \big(1 + \frac{1}{n} \int_{ {\R}^3} f^n \, dv \big)^{-1}  (\mathcal{R}_3) ^{n } \varphi   dv$ on $B_R(x)$ estiblished in Lemma \ref{equi-R3}, we obtain that
  	 \begin{align*} 
  		&\sup_{n>N} \int_0^T dt 	\iint_{\R^3_x \times \mathbb{R}^3_v}   \bigg(1 + \frac{1}{n} \int_{ {\R}^3_v} f^n \, dv \bigg)^{-1}   ( \mathcal{R}_3 )_k^n  \varphi \mathbf{1}_{  (U_\delta \cup U_L^n)} \, dx dv \\  
  		\leq&  \sup_{n>N} \int_0^T dt 	\iint_{\R^3_x \times \mathbb{R}^3_v}    \bigg(1 + \frac{1}{n} \int_{ {\R}^3_v} f^n \, dv \bigg)^{-1}  ( \mathcal{R}_3 ) ^n  \varphi \mathbf{1}_{  (U_\delta \cup U_L^n)} \, dx dv \to 0 
  	\end{align*}
 as $\delta\to 0 $ and $L \rightarrow \infty$. Therefore, we conclude that
\begin{equation*} 
\lim_{k \to \infty} \sup_{n>N} \int_0^T dt 	\iint_{\R^3_x \times \mathbb{R}^3_v} \bigg(1 + \frac{1}{n} \int_{ {\R}^3_v} f^n \, dv \bigg)^{-1} ( \mathcal{R}_3 )_k^n  \varphi \mathbf{1}_{W_\epsilon } \, dx dv = 0. 
\end{equation*} 
Furthermore, invoking Lemma \ref{Lmm-ATV} yields the estimate 
  \begin{align*}
 & \sup_{n>N} \int_0^T dt 	\iint_{\R^3_x \times \mathbb{R}^3_v} ( \mathcal{R}_3 )_k^n  \varphi \mathbf{1}_{([0,T] \times B_R(x)) \setminus W_\epsilon} \, dx dv \\ 
 \leq &C_{\nu ,R}  \sup_{n>N} \sup_{ t \in [0, T]} \iint_{\R^3_x \times   \R ^3_{v_*}  } f^n_* ( 1+ |v_*|^2) dxdv_* \int_{ B_R(v) } f^n \mathbf{1}_{([0,T] \times B_R(x)) \setminus W_\epsilon}  dv   + \frac{1}{2} \int_{0}^{T} dt \int_{\mathbb{R}^{ 3} _v } e^n dv  \\
 	\leq &2 \epsilon C_{\nu ,R} \sup_{ t \in [0, T]}   \iint_{\R^3_x \times   \R ^3_{v_*}  } f^n_* ( 1+ |v_*|^2) dxdv_* + \frac{1}{2} \int_{0}^{T} dt \iint_{\R^3_x \times\mathbb{R}^{ 3} _v } e^n dx dv   \to 0 \quad ( \epsilon \to 0).
  \end{align*} 
Here, the last inequality follows from the bound $  \|f ^n \|_{L^1( B_R(v))} \leq 2\epsilon \ $, which is valid a.e. on $ ([0,T] \times B_R(x)) \setminus W_\epsilon$.  Combining these convergence results, we obtain
 \begin{align}\label{k-1}
 	\lim_{k \to \infty} \sup_{n>N} \int_0^T dt 	\iint_{\R^3_x \times \mathbb{R}^3_v} \bigg(1 + \frac{1}{n} \int_{ {\R}^3_v} f^n \, dv \bigg)^{-1} ( \mathcal{R}_3 )_k^n  \varphi   \, dx dv = 0.
 \end{align}
  On the other hand, for any fixed $n$ with $1\leq n \leq N$ and an arbitrary $\delta>0$, there exists a threshold $k_n>0$ such that for all $k \geq k_n$ 
  \begin{align*}
  	\int_0^T dt 	\iint_{\R^3_x \times \mathbb{R}^3_v} \bigg(1 + \frac{1}{n} \int_{ {\R}^3_v} f^n \, dv \bigg)^{-1} ( \mathcal{R}_3 )^n_k   \varphi   \, dx dv < \delta.
  \end{align*}
  By selecting a uniform threshold $\tilde{k}= \max \{ k_1,\cdots, k_N\}$, we ensure that the estimate holds simultaneously for all $1 \leq n \leq N$ whenever $k>\tilde{k}$. Consequently, given any $\delta>0$ and any $k>\tilde{k}$, we have 
  \begin{align*} 
    \int_0^T dt \iint_{\mathbb{R}^3_x \times \mathbb{R}^3_v} \bigg(1 + \frac{1}{n} \int_{ {\R}^3_v} f^n \, dv \bigg)^{-1} (\mathcal{R}_3)^n_k   \varphi  dx dv   < \delta, \quad \forall   1 \leq n \leq N. 
  \end{align*}
 This establishes the desired convergence
  \begin{align}\label{k-2}
  	\lim_{k \to \infty} \int_0^T dt \iint_{\mathbb{R}^3_x \times \mathbb{R}^3_v} \bigg(\bigg(1 + \frac{1}{n} \int_{ {\R}^3_v} f^n \, dv \bigg)^{-1}(\mathcal{R}_3)^n_k +( \mathcal{R}_3 )_k\bigg)    \varphi   dx dv = 0, \quad \forall   1 \leq n \leq N. 
  	\end{align} 
  	Thus, combining the convergence results \eqref{k-1} and \eqref{k-2}, for sufficiently large $L>0$, we establish the convergence 
 \begin{align}\label{R3-0}
  	\lim_{k \to \infty } \sup_{n} \int_0^T dt 	\iint_{\R^3_x \times \mathbb{R}^3_v}  \bigg( \bigg(1 + \frac{1}{n} \int_{ {\R}^3_v} f^n \, dv \bigg)^{-1}( \mathcal{R}_3 )_k^n   + ( \mathcal{R}_3 )_k \bigg)\varphi dxdv =0.
 \end{align}

 Furthermore, the measure bound $\text{meas}(N_L^n) \leq C/L$, combined with the equi-integrability of the sequence $	\int_{0}^{T} d t \int_{ \R^3_v } \big( 1+ \frac{1}{n} \int_{\R^3_v } f^n d v \big)^{-1} ( \mathcal{R}_3)^n   \varphi d v$ on $B_R(x)$, implies that
\begin{equation}\label{R3-2}
 \begin{aligned}
 & \sup_{n \ge 1} \int_0^T dt \iint_{\mathbb{R}^3 _x\times \mathbb{R}^3_v}  \bigg( 1+ \frac{1}{n} \int_{\R^3_v } f^n d v \bigg)^{-1}  ( \mathcal{R}_3) ^{n,k} \varphi\, \mathbf{1}_{N_L^n} \, dx \, dv \\
  \le &  \sup_{n \ge 1} \int_0^T dt \iint_{\mathbb{R}^3 _x\times \mathbb{R}^3_v} \bigg( 1+ \frac{1}{n} \int_{\R^3_v } f^n d v \bigg)^{-1} ( \mathcal{R}_3) ^n \varphi \, \mathbf{1}_{N_L^n} \, dx \, dv 
   \to 0 \quad (  L \to \infty).
 \end{aligned}
\end{equation}
Similarly, we have
\begin{align}\label{R3-3}
\lim_{L \to \infty} \int_0^T dt \iint_{\mathbb{R}^3 _x\times \mathbb{R}^3_v}   ( \mathcal{R}_3)^k \varphi \, \mathbf{1}_{N_L } \, dx \, dv =0.
\end{align}
By Lemma \ref{L-R3nc}, the convergence for the angular cutoff case satisfies
  \begin{align*}
  	\lim\limits_{n\to \infty} \int_0^T dt \iint_{\mathbb{R}^3_x \times \mathbb{R}^3_v }  \bigg(1 + \frac{1}{n} \int_{\mathbb{R}^3_v} f^n \, dv \bigg)^{-1}  (\mathcal{R}_3) ^{n, k} \varphi \mathbf{1}_{ (N_L^n)^c}  dx dv  \\
  = \int_0^T dt \iint_{\mathbb{R}^3_x \times \mathbb{R}^3_v}  (\mathcal{R}_3) ^k \varphi \mathbf{1}_{(N_L)^c} \, dx dv.
  \end{align*}
Combining this result with the convergences established in \eqref{R3-0}-\eqref{R3-3}, we obtain 
 \begin{align*}
&\bigg| \int_{0}^{T} dt \iint_{ \mathbb{R}^3_x \times \mathbb{R}^3_v } \bigg( 1+ \frac{1}{n} \int_{\R^3_v} f^n d v \bigg)^{-1}  ( \mathcal{R}_3 )^n \varphi d x d v - \int_{0}^{T} dt \iint_{ \mathbb{R}^3_x \times \mathbb{R}^3_v } ( \mathcal{R}_3 )  \varphi d x d v \bigg| \\
\leq & \bigg|  \int_0^T dt \iint_{\mathbb{R}^3_x \times \mathbb{R}^3_v }  \bigg(1 + \frac{1}{n} \int_{\mathbb{R}^3_v} f^n \, dv \bigg)^{-1}  (\mathcal{R}_3) ^{n, k} \varphi \mathbf{1}_{ (N_L^n)^c}  dx dv - \int_0^T dt \iint_{\mathbb{R}^3_x \times \mathbb{R}^3_v}  (\mathcal{R}_3) ^k \varphi \mathbf{1}_{(N_L)^c} \, dx dv  \bigg|  \\ 
&+ \sup_{n \ge 1} \int_0^T dt \iint_{\mathbb{R}^3_x \times \mathbb{R}^3_v} \bigg(  \bigg( 1+ \frac{1}{n} \int_{\R^3_v } f^n d v \bigg)^{-1}     ( \mathcal{R}_3) ^{n,k}  \mathbf{1}_{N_L^n}    +  ( \mathcal{R}_3 )^k   \mathbf{1}_{N_L} \bigg) \varphi  \, dx \, dv \\
& + \sup_{n \ge 1} \int_0^T dt 	\iint_{\R^3_x \times \mathbb{R}^3_v} \bigg(\bigg(1 + \frac{1}{n} \int_{ {\R}^3_v} f^n \, dv \bigg)^{-1} ( \mathcal{R}_3 )_k^n     + ( \mathcal{R}_3 )_k \bigg)\varphi  \, dx dv  \to 0  
 \end{align*} 
as $n \to \infty , k \to \infty$ and $L \to \infty$. This completes the proof of Lemma  \ref{l-cnvR3}.
\end{proof}

 \section{Existence and entropy inequality}\label{Sec-ee}
 This section is devoted to completing the proof of Theorem \ref{MainThm}. By Lemma \ref{Lemma 4.14}, the approximating equation \eqref{approximation} admits a unique distributional solution $f^n$.    Moreover, we have shown that the sequence $\{f^n\}_{n\ge1}$ converges strongly in $L^1([0,T] \times \R^3_x \times \R^3_v)$ to a limit function $f$.    Accordingly, the proof of Theorem \ref{MainThm} is divided into two parts. In the first part, we show that the limit function $f$ is a renormalized solution of the Boltzmann equation \eqref{Boltz}. We then verify that   $f \in C( \R _+ ; w\text{-}L^1 \left( \R ^3_x ; B^{-\alpha}_{1,1}( \R^3_v  ) \right)   )\cap L^\infty ( \R_+; L^1 ( 1+|x|^2 +|v|^2 ) d x d v ) $, and that $f$ satisfies the a priori bound \eqref{bdd-f}.   In the second part, we prove that $f$ satisfies the entropy inequality associated with the Boltzmann equation \eqref{Boltz}. 
 
\subsection{Existence and continuity of  time : Part $\mathrm{I}$ of Theorem}\label{subsec-exist} 
 Since $f^n$ is the unique distributional solution to the approximating equation \eqref{approximation}, Definition \ref{def-re} of renormalized solutions implies 
 \begin{equation*}
 	\frac{\partial }{\partial t} \beta_\delta ( f^n) + v \cdot \nabla_x \beta_\delta ( f^n)  = \beta_\delta ' ( f^n ) \tilde{Q}_n (f^n,f^n),
 \end{equation*}
 where  
 \begin{align*}
 	\beta_\delta ' ( f^n ) \tilde{Q}_n (f^n,f^n)  &= \bigg( 1+ \frac{1}{n} \int_{\R^3_v} f^n d v \bigg)^{-1}   \beta_\delta ' ( f^n ) Q _n (f^n,f^n)\\  
 	&= \bigg( 1+ \frac{1}{n} \int_{\R^3_v } f^n d v \bigg)^{-1}  \left\{   (\mathcal{R}_1)^n + (\mathcal{R}_2)^n +(\mathcal{R}_3)^n \right\}.
 \end{align*}
 Multiplying both sides of the above equation by a test function $\varphi ( t,x , v) \in L^\infty ( [0,T] \times \R^3_x ; C^{\alpha }_c (\R^3_v ))$  with compact  support in  $[0,T] \times B_R (x)\times  B_R(v)$  and integrating over $(t,x,v) \in ([ 0,T ] \times \R^3_x \times \R^3_v  )$, we obtain 
 	\begin{equation*}
 		\begin{aligned}
 			&\iint_{ \mathbb{R}^3_x \times \mathbb{R}^3 _v} \beta_\delta ( f^n )( T,\cdot ) \varphi d x d v - \iint_{ \mathbb{R}^3 _x\times \mathbb{R}^3_v } \beta_\delta ( f^n ) ( 0, \cdot ) \varphi d x d v \\
 			= 	& \int_{0}^{T} d t \iint_{ \mathbb{R}^3 _x\times \mathbb{R}^3 _v} \bigg( 1+ \frac{1}{n} \int_{\R^3_v } f^n d v \bigg)^{-1}  \left(  (\mathcal{R}_1 )^n + (\mathcal{R}_2 )^n + (\mathcal{R}_3 )^n \right) \varphi d x d v.
 		\end{aligned}
 	\end{equation*}
Taking the limit as $n \to \infty $ on both sides, and noting that
 $f^n $ converges strongly to $f$ and $\beta_\delta \in C^2$, we have
 \begin{equation*}
 	\frac{\partial \beta_\delta ( f^n) }{ \partial t} \to  \frac{\partial \beta_\delta ( f ) }{ \partial t} \quad \text{in } \  \mathscr{D}'((0,\infty)\times \R^3_x \times \R^3_v).
 \end{equation*} 
 By combining this with the convergence properties established in \eqref{cnv-R1}, \eqref{cnv-R2} and \eqref{cnv-R3}, we obtain
 \begin{equation}\label{rs}
 	\begin{aligned} 
 		\iint_{ \R ^3_x \times \R ^3_v }( \beta_\delta ( f  ) ( T,\cdot ) -  \beta_\delta ( f  ) ( 0,\cdot ) )\varphi   d x d v =  & \int_{0}^{T} dt \iint_{ \R ^3 _x\times \R ^3_v } \left\{ (\mathcal{R}_1)  + (\mathcal{R}_2)  +(\mathcal{R}_3)  \right\} \varphi d x d v \\
 		 =& \int_{0}^{T} dt \iint_{ \R ^3_x \times \R ^3_v }  \beta_\delta ' ( f  ) Q  (f ,f ) \varphi  d x d v.
 	\end{aligned}
 \end{equation}
 This demonstrates that $f$ is actually a renormalized solution of \eqref{Boltz}. Furthermore, we find that for any $\Phi \in L^\infty ( \R^3_x ; C^{\alpha }_c (\R^3_v ))$ and any $0 \leq s< t < \infty  $, the following equality holds
  \begin{equation*} 
 	\begin{aligned} 
 		\iint_{ \R ^3_x \times \R ^3 _v}( \beta_\delta ( f  ) ( t,\cdot ) -  \beta_\delta ( f  ) ( s,\cdot ) )\Phi   d x d v  
 		=& \int_{s}^{t} d\tau  \iint_{ \R ^3_x \times \R ^3 _v}  \beta_\delta ' ( f  ) Q  (f ,f ) \Phi  d x d v.
 	\end{aligned}
 \end{equation*}
Invoking the absolute continuity of the integral, we deduce that $$\beta_\delta ( f  )  \in C \left( \R _+  ;\  w\text{-} L^1  \left( \R ^3_x ; B^{-\alpha}_{1,1}( \R^3_v  ) \right)  \right)_+    . $$
Additionally, observing that $\beta_\delta ( s) = s / (1+ \delta s) $ satisfies
 \begin{align*}
 	0 \leq   s - \beta_\delta ( s) \leq \epsilon_K (\delta ) s + 1_{s \geq K } s  \quad \text{ on } \ [0,\infty ),
 \end{align*} 		 
 where $\epsilon_K ( \delta ) \to 0 $ as $ \delta \to 0_+$, we apply the absolute continuity of the integral again to show that for any $ t \in [ 0, \infty )  $, 
 \begin{align*}
 	\| f ( t,\cdot ) - \beta_\delta ( f )( t, \cdot ) \| _{L^1 ( \R^3_x \times \R^3_v ) } \leq \epsilon_K ( \delta ) \|f  \|_{ L^1 ( \R^3 _x\times \R^3_v )}  + \iint_{ \R ^3_x \times \R ^3 _v} f 1_{ f  \geq K } d x d v \to 0 
 \end{align*}
 as $\delta \to 0_+$ and $ K \to \infty $.  Consequently, we conclude that $f \in C \left(\R_+; w \text{-}  L^1(\mathbb{R}^3_x; B^{-\alpha}_{1,1}(\mathbb{R}^3_v))  \right)_+  $.

Next, we establish the bound \eqref{bdd-f}, which is given by
 \begin{equation*}  
 	\iint_{ \R ^3_x \times \R ^3 _v} f ( t, \cdot )( 1 + |x|^2 + |v|^2 + | \log f |) d x d v \leq C_T.
 \end{equation*}
 First, for any $R>0$ and $ t \in  [0,T]  $, we observe that
 \begin{equation*}
 	\varphi_R \triangleq \left(1+|x|^2 +|v|^2 \right) \mathbf{1}_{|x|\leq R } \mathbf{1}_{|v|\leq R}  \in L^\infty (\R^3 _x \times \R^3 _v  ).
 \end{equation*}
 Recall that $f^n (t) $ converges weakly  to $ f (t)$ in $L^1 (\R^3_x \times \R^3_v  )$ for all $t \in [0, T]$. Then, Lemma \ref{Lemma 4.17} implies  that there exists a constant $C_T > 0$, independent of $R$, such that
 \begin{equation*}
 	\begin{aligned}
 		\iint_{ \R^3 _x \times \R^3_v  } f  (t,x,v ) \varphi_R dxdv
 		=& \lim \limits_{n\rightarrow \infty } \iint_{ \R^3_x  \times \R^3 _v  } f^n  (t,x,v   ) \varphi_R dxdv \\
 		\leq & \sup \limits_{n \geq 1} \iint_{ \R^3_x \times \R^3 _v } f^n  (1+|x|^2 +|v|^2  ) dxdv \leq C_T \,.
 	\end{aligned}
 \end{equation*}
 Taking the limit as $R \to + \infty$, we obtain by the Monotone Convergence Theorem (Lemma \ref{MCT}) that
 \begin{equation} \label{eq6.83+1}
 	\sup_{t \in [0, T]} \iint_{ \R^3_x \times \R^3_v } f  (1+|x|^2 +|v|^2 ) dxdv\leq C_T.
 \end{equation}
 
 Moreover, Lemmas \ref{Lemma 4.6} and \ref{Lemma 4.17} imply that for $t \in [0, T]$, the $H$-functional satisfies
 	\begin{equation*} 
 		H (f) \leq \liminf \limits_{n\rightarrow \infty } H (f^n) \leq C_T'.
 \end{equation*}  
Using the fact that $t\log \frac{1}{t}\leq C_0 \sqrt{t}$ for $t\in (0,1)$ with some $C_0\geq 0$, we infer that
\begin{equation*} 
		\begin{aligned}
			-  \iint_{0<f  \leq 1} f  \log f   dxdv = & -  \iint_{0<f  \leq e^{ - |x|^2 - |v|^2 }} f  \log f   dxdv -   \iint_{e^{ - |x|^2 - |v|^2 }<f  \leq 1} f  \log f   dxdv \\
			\leq &   C_0 \iint_{\R^3 _x\times \R^3_v } e^{ - \frac{|x|^2 + |v|^2}{2} } d x d v +   \iint_{\R^3_x \times \R^3_v } f  (|x|^2 +|v|^2) dxdv \,.
	\end{aligned} 
\end{equation*}
 Consequently, we have
 	\begin{equation}\label{bdd-flogf}
 	\begin{aligned}
 		\iint_{ \R^3 _x \times \R^3_v } f  |\log f  | dxdv 
 		=&\iint_{ \R^3 _x\times \R^3_v } f \log f   dxdv -2 \iint_{0<f  \leq 1} f  \log f   dxdv \leq C_T.
 	\end{aligned}
 \end{equation} 
By combining \eqref{eq6.83+1} and \eqref{bdd-flogf}, we conclude that the bound \eqref{bdd-f} holds.

 \subsection{Entropy inequality: Part $\mathrm{II}$ of Theorem}
 
 In this section, we focus on proving Part $\mathrm{II}$ of Theorem \ref{MainThm}. Specifically, we show that the renormalized solution $f$ of \eqref{Boltz} satisfies the entropy inequality \eqref{entropy-th} given by
 \begin{equation}\label{entropy}
 	H( f ) ( t ) + \frac{1}{4} \int_{0}^{t} ds \iiiint_{ \R ^3_x \times \R ^3 _v \times \R ^3_{v_* }\times \S^2 }  B    
 	( { f }' {f'_* }  - f  f _* ) \log \left( \frac{ {f }'  {f'_* }  }{f  f _*} \right) d x d v d v_* d \sigma \leq H ( f_0 ).
 \end{equation} 
As stated in \eqref{entro} of Lemma \ref{Lemma 4.17}, the approximate distributional solution $f^n$ to the approximating problem \eqref{approximation} admits the entropy identity. Consequently,   for all $t \geq 0 $, we have
 \begin{equation*}
 	H(f^n_{0})   = H(f^n)(t) 
 +\int_{0}^{t} d \tau \iint_{ \R^3 _x \times \R^3 _v} \tilde{e} ^n(\tau,x, v ) dxdv,
 \end{equation*}
 where $\tilde{e} ^n$ and $ H(f^n)$ are defined in \eqref{e na} and \eqref{Hf}, respectively. Lemma \ref{Lemma 4.8} implies that $ \lim_{n \to + \infty} H(f^n_{0}) = H (f_0) $. In order to prove the entropy inequality \eqref{entropy-th} (or \eqref{entropy}), it suffices to show that for almost all $(t,x)\in \R_+ \times \R^3_x$, we have
 \begin{equation}\label{ena} 
 		\liminf_{n\rightarrow \infty} \int_{\R^3_v } \tilde{e} ^n dv  \geq \int_{ \R^3_v } e  dv .
 \end{equation}
 
 For $R \in ( 1, \infty )$, we define 
 \begin{equation*}
 	\begin{aligned}
 		&E_R  =B_R (v)\times B_R (v_*)\times \mathbb{S}^2, \quad d \Theta  = d v  dv _* d\sigma ,  \quad 
          N(f) = 1+   \int_{\R^3_v} f  d v\\
 		& L^R ( f )  = \iint_{B_R  (v_*)\times \S^2 } f_* B^R ( v-v_* , \sigma ) d v_* d \sigma ,\quad 
 			  L^R_n ( f )  = \iint_{B_R  (v_*)\times \S^2 } f_* B^R_n ( v-v_* , \sigma ) d v_* d \sigma ,
 	\end{aligned}
 \end{equation*}
 where 
 \begin{equation*}
 	\begin{aligned}
  	B^R (|z| , \cos \theta ) = \left(B (|z| , \cos \theta )1_{\frac{1}{R} \leq \theta \leq \frac{\pi}{2}  } \right) \wedge R ,  \quad
  	B^R_{ n } (|z| , \cos \theta ) = (B_n (|z| , \cos \theta )1_{\frac{1}{R} \leq \theta \leq \frac{\pi}{2} })\wedge R   .
 	\end{aligned}
 \end{equation*}
 Since $ \int_{ \S ^{   2 } } B^R_n (v-v_*, \sigma ) d \sigma$ forms a bounded sequence in $ L^\infty ( ( 0,R) \times B_R(x) \times B_R(v_*) ; L^1 ( B_R(v)) )$, we deduce the strong convergence of $L^R_n (f^n) $ by applying the generalized averaged velocity lemma (Lemma \ref{theo-6.4.1}). For further details, we refer to Section 4 of Reference \cite{Diperna-Lions}. More precisely, we have
 \begin{equation}\label{cnv-Ln}
 	L^R_n(f^n) \to L^R(f) \quad \text{ in } L^1 (( 0,R)\times B_R(x) \times B_R (v)).
 \end{equation}  
 
 \begin{lemma}\label{lemma-rwc}
 For almost all $(t,x) \in (0,R) \times   B_R (x)$, the sequences $  \{   {f ^n}' {f_{*}^n}'  B^R_n \}_{n  \geq 1 }$ and $  \{f^n f^n_* B^R_n\}_{n  \geq 1 }$ are relatively weakly compact in $L^1(E_R,d \Theta  )$.
 \end{lemma}
 \begin{proof}  
 We first take an increasing function $\Psi \in C(\R_+)$ satisfying
 \begin{align*}
 \Psi (t) \rightarrow +\infty \, \text{as } t \rightarrow + \infty ,\quad \Psi ( t ) ( \log t )^{-1} \rightarrow 0  \, \text{as } t \rightarrow + \infty .
 \end{align*}
 This implies that for any $\epsilon>0$, there exists a sufficiently large constant $M$ such that $\Psi (t) \leq \epsilon     \log   t $ for all $t\geq M$. Accordingly, we observe that
 \begin{align*}
 	&\left\| \int_{ B_R(v)  } f^n  \Psi (   f^n ) dv - \int_{ B_R(v)  } f   \Psi ( f  ) dv \right\|_{L^1((0,R) \times B_R(x) )} \\
 	\leq & \left\|\int_{ B_R(v)  } \left( f^n \Psi (  f^n  )\mathbf{1}_{ f^n \leq M  }  - f  \Psi ( f  )\mathbf{1}_{ f \leq M  } \right) dv\right\|_{L^1((0,R) \times B_R(x) )} \\
 	& + \epsilon \sup_{n \geq 1} \left\| \int_{ B_R(v) } \left(f^n  |\log    f^n |   + f | \log  f |   \right) dv \right\|_{L^1((0,R) \times B_R(x) ) }.
 \end{align*}  
 By combining the boundedness of the continuous function $\Psi(t)$   on $[0, M]$ with the bound of $f^n  \log  f^n $  in $L^\infty(\R_+; L^1(\mathbb{R}^3_x \times \mathbb{R}^3 _v ))$ and the convergence in \eqref{fn-f}, we conclude that the inequality above tends to $0$ as $n \to \infty$. This implies that the sequence $\int_{ B_R(v)  } f^n \Psi( f^n ) dv $ is compact in $L^1((0,R) \times B_R(x))$. Thus, we  deduce that 
 \begin{equation*}
 	\begin{aligned}
 		N(f^n)^{-1}	\int_{E_R}   f ^n f_{ *}^n    &(\Psi (   f ^n)+\Psi(  f_{ *}^n)) B^R_{  n} d \Theta  \\
 		N(f^n)^{-1}	\int_{E_R}   {f^n}' {f_{  *}^n}'  &(\Psi (   {f ^n}' )+\Psi (   {f_{ *}^n}' ) ) B^R_{  n} d \Theta  
 	\end{aligned}
 \end{equation*}
 are compact in $L^1((0,R) \times B_R(x))$.

 We claim that if a sequence $f^n(x) \geq 0$ is compact in $L^1(\Omega)$, where $\Omega$ is a bounded set, then there exists a function $g(x) \in L^1(\Omega)$ such that $0 \leq f^n(x) \leq g(x)$ a.e. in $\Omega$ (up to a subsequence). Given the compactness of $f^n$  in $L^1(\Omega)$, there exists a function $f \in L^1(\Omega)$ such that $\lim_{n \to \infty} \|f^n - f\|_{L^1(\Omega)} = 0$. We extract a subsequence such that $\|f^{n+1} - f^n\|_{L^1(\Omega)} \leq 2^{-n}$ for all $n \in \mathbb{N}_+$, and define $g(x)$ by
 \begin{equation*}
 	g(x) = f^1(x) + \sum_{n=1}^{\infty} |f^{n+1}(x) - f^n(x)|.
 \end{equation*}
By the Monotone Convergence Theorem (Lemma \ref{MCT}), $g$ belongs to $L^1(\Omega)$ because
 \begin{equation*}
 	\int_{\Omega} g(x) dx \leq \|f^1\|_{L^1(\Omega )} + \sum_{n=1}^{\infty} 2^{-n} = \|f^1\|_{L^1( \Omega)} + 1 < \infty.
 \end{equation*}
 Using the telescoping sum representation $f^m(x) = f^1(x) + \sum_{n=1}^{m-1} (f^{n+1}(x) - f^n(x))$ and the triangle inequality, we obtain $0 \leq f^n(x) \leq g(x)$ for almost every $x \in \Omega$. This establishes the claim. Consequently, there exist non-negative functions $\bar{K}, \bar{N} \in L^1((0,R) \times B_R (x))$, independent of $n$, satisfying
 \begin{equation*}
 	\begin{aligned}
 		N(f^n)^{-1}	\int_{E_R}   f ^n f_{ *}^n    &(\Psi (   f ^n)+\Psi(  f_{ *}^n)) B^R_{  n} d \Theta   \leq  \bar{K}    , \\
 		N(f^n)^{-1}	\int_{E_R}   {f^n}' {f_{  *}^n}'  &(\Psi (   {f ^n}' )+\Psi (   {f_{ *}^n}' ) ) B^R_{  n} d \Theta   \leq  \bar{K},   \\
 		N ( f^n) &\leq \bar{N}.
 	\end{aligned}
 \end{equation*}
 In addition, for all $M>0$ and any Borel subset $A\subset E_R$, it follows that for almost all $(t,x) \in (0, R) \times B_R$,
 \begin{equation*}
 	\begin{aligned}
 		\int_{A}   f^n  f^n_{ * } B^R_{  n} d \Theta  
 		\leq& R M^2 \Theta  (A) + \int_{E_R}    f^n  f^n_{ * } (\mathbf{1}_{   f ^n\geq M}+ \mathbf{1}_{   f_{  *}^n \geq M}) B^R_{  n} d \Theta \\
 		\leq & R M^2 \Theta  (A)+\Psi(M)^{-1}\bar{K} \bar{N}.
 	\end{aligned}
 \end{equation*}
This implies that the sequence $\{f^n f^n_* B^R_n\}_{n  \geq 1 }$ is  equi-integrable, and thus, by  Dunford-Pettis theorem (Lemma \ref{theorem-dunford}), it is relatively weakly compact in $L^1(E_R, d\Theta)$ for almost all $(t, x) \in (0, R) \times B_R$. 
  By employing analogous arguments, we conclude that   $  \{   {f ^n}'  {f_*^n}'B^R_{ n} \}_{n  \geq 1 }$ is also relatively weakly compact in $L^1(E_R,d \Theta  )$ for almost all $(t,x) \in (0,R) \times B_R(x)$. Finally, the proof of Lemma \ref{lemma-rwc} is finished.
  \end{proof}
 \begin{lemma}\label{lemma-sc}
 	For any $\varphi \in L^\infty (E_R)$, we establish
 	\begin{equation}\label{P con}
 	\int_{E_R}   N(f^n)^{-1}   f^n  f^n_{ * }      B^R_n   \varphi d \Theta   \to  \int_{E_R} N(f)^{-1}  f   f _{ * }       B^R     \varphi d \Theta  \quad \text{strongly in } L^1((0,R)\times B_R(x)),
 	\end{equation}
 	and
 	\begin{equation}\label{til P con}
 		\int_{E_R}  N(f^n)^{-1} {f^n}' {f_{  *}^n}'   B^R_n   \varphi   d\Theta  \rightarrow \int_{E_R}   N(f )^{-1} {f }' f_{*}'   B^R    \varphi  d\Theta   \quad \text{strongly in } L^1((0,R)\times B_R(x)).
 	\end{equation} 
 	\end{lemma}
 \begin{proof}  
Combining \eqref{cnv-Ln} with the strong convergence of $f^n$, we obtain
 \begin{align*}
 	N(f^n)^{-1}  	L^R_n(f^n) \to 	N(f)^{-1} L^R (f ) \quad \text{a.e.} \ t,x \in ( 0,R) \times B_R(x).
 \end{align*}
 Furthermore, by invoking Egorov's Theorem (Lemma \ref{egorov}), for any $\epsilon>0$, there exists a Borel $E \subset \Omega_R  = ( 0,R ) \times B_R (x )\times B_R(v) $ such that $|E^c \cap \Omega_R  |< \epsilon $ and $	N(f^n )^{-1}  	L^R_n(f^n)$ converges uniformly to $	N(f) ^{-1} L^R (f )$ on $E$. Additionally, we observe the bound
 \begin{align*}
 	0 \leq 	N(f^n)^{-1}  	L^R_n(f^n) \leq \|B^R_n \|_{L^\infty (B_R(v) ; L^1 ( \S^2) ) }.
 \end{align*}
 With these preliminaries, we proceed to prove the convergence \eqref{P con}. 
 We first observe the following decomposition
 \begin{align*}
   \left\{f^n L^R_n (f^n ) N(f^n)^{-1}    -    f  L^R (f  )N(f )^{-1}  \right\}  \varphi   
 	= &     \left\{  L^R_n (f^n ) N(f^n)^{-1} -   L^R (f  ) N(f )^{-1}   \right\} f^n  \varphi \\
 	& +  \left(     f^n -  f \right) L^R (f  ) N(f )^{-1}  \varphi .
 \end{align*}
 This yields 
  \begin{equation*}
 	\begin{aligned}
 		&\int_{0}^{R}dt \int_{B_R(x)} dx \left|\int_{E_R}  \left\{f^n L^R_n (f^n ) N(f^n)^{-1}    -    f  L^R (f  )N(f )^{-1}  \right\}  \varphi      d\Theta  \right| \\
 		\leq & \sup_{n \geq 1} \Vert f^n \Vert_{L^1 ((0, R) \times B_R (x)\times B_R (v))} \sup \limits_{E} \left| L^R_{n }(f^n) N (f^n)^{-1} - L^R  (f) N (f)^{-1} \right| \\
 		+ & \left( \sup_{n \geq 1} \| B^R_n \|_{L^\infty (B_R (v); L^1 ( \S^2) ) } + \|B^R   \|_{L^\infty (B_R(v) ; L^1 ( \S^2) ) } \right) \left( \sup_{n \geq 1} \int_{E^c} f^n  dtdxdv   +\int_{E^c} f  dtdxdv   \right) \\
 		+& \int_{0}^{R}dt \int_{B_R(x)} dx |\int_{B_R (v) } f^n  \tilde{\varphi} dv   -\int_{B_R (v) } f  \tilde{\varphi} dv   |,
 	\end{aligned}
 \end{equation*}
 where $\tilde{\varphi}= \mathbf{1}_E L^R (f) N (f)^{-1} \varphi \in L^\infty ( [0,T] \times B_R (x) \times B_R (v))$. 
Since 
$$\sup_{n \geq 1} \Vert f^n \Vert_{L^1 ((0, R) \times B_R (x)\times B_R (v) )} \leq C_R$$ 
by Lemma \ref{Lemma 4.17}, we conclude that the first quantity on the right-hand side of the above inequality converges to 0 as $n \to + \infty$. Moreover, the convergence in \eqref{fn-f} shows that the last term on the right-hand side of the above inequality converges to 0 as $n \to + \infty$. Furthermore, we have $\sup_{n \geq 1} \| B^R_n \|_{L^\infty (B_R(v) ; L^1 ( \S^2) ) } + \|B^R  \|_{L^\infty (B_R(v) ; L^1 ( \S^2) ) }\leq 2 R$ and $f^n $ converges weakly to $f $ in $L^1 ( (0, R) \times B_R (x)\times B_R (v)  )$, which implies that the second term on the right-hand side of the above inequality vanishes as $\epsilon \to 0_+$. Thus, the assertion \eqref{P con} holds.
 
To establish the convergence \eqref{til P con}, we first observe the identity
 \begin{equation*}{\small
 		\begin{aligned}
 			\int_{E_R}  N(f^n)^{-1} {f^n}' {f_{  *}^n}'   B^R_n   \varphi   d\Theta 
 			=&   \iiint_{ B_R (v)  \times B_R (v_*)  \times \mathbb{S}^2} N(f^n)^{-1} f^n  f^n_{ *} B^R_{ n} 
 			  \varphi'   d\sigma dv  dv_*, \\
 			\int_{E_R}  N(f )^{-1} {f }' {f_{*}'}   B^R    \varphi   d\Theta 
 		=&   \iiint_{B_R (v)  \times B_R (v_*)  \times \mathbb{S}^2} N(f )^{-1} f   f _{ *} B^R 
 		\varphi'   d\sigma dv  dv_*, 
 	\end{aligned}}
 \end{equation*}
 where $\varphi'=\varphi (v')$. 
By employing arguments analogous to those used for \eqref{P con}, we deduce that the convergence \eqref{til P con} holds. This completes the proof of Lemma \ref{lemma-sc}.
\end{proof}
 \begin{lemma} 
 For almost all $(t,x)\in (0,R)\times B_R$, we obtain
 \begin{equation}\label{cnv-fn'*}
 \begin{aligned} 
 	\left( 1+ \frac{1}{n} \int_{\R^3_v} f^n d v \right)^{-1} f^n  f^n_{ * }    B^R_{  n}  \rightarrow {f  }   {f_* }  B^R,
 	\left( 1+ \frac{1}{n} \int_{\R^3_v} f^n d v \right)^{-1} {f ^n}'  {f_*^n}'  B^R_n   \rightarrow \ f'   f'_*   B^R 
 	\end{aligned}
 \end{equation} 
 weakly in $L^1 (E_R, d \Theta).$
\end{lemma}
\begin{proof} 
To simplify the notation, we define
  \begin{equation*}
  	\begin{aligned}
  		P^n  = N  (f^n)^{-1} f^n  f^n_{ * } B^R_{  n} \,, \quad &	P  = N  (f )^{-1}{f  }   {f_* }  B^R  \, ,\\
  	 \tilde{P}^n  = N  (f^n)^{-1} {f ^n}'  {f_*^n}' B^R_n \,, \quad& \tilde{P} =N  (f )^{-1}  f'   f'_*   B^R \,.
  	\end{aligned}
  \end{equation*} 
The strong convergence of $	\int_{E_R}   P^n \varphi d \Theta$ implies the existence of a subsequence converging pointwise a.e. in $(t,x)$ to $\int_{E_R}  P \varphi d \Theta$. However, it is important to note that this subsequence may depend on the choice of $\varphi$.  
Indeed, for any Borel subset $A$ of $E_R$,  there exists a sequence of measurable subsets $\{ A_k\}_{k=1}^\infty $ such that $\mathbf{1}_{A_k} \to \mathbf{1}_A $ in $L^1( E_R)$. By choosing $\varphi = \mathbf{1}_{A_k}$ and invoking the strong convergence of $	\int_{E_R}  P^n \mathbf{1}_{A_k} d\Theta$, we can extract a subsequence (still denoted by $	\int_{E_R}   P^n \mathbf{1}_{A_k} d\Theta$) such that
\begin{equation*}
 	\int_{E_R}   P^n \mathbf{1}_{A_k} d\Theta  \to  	\int_{E_R}   P  \mathbf{1}_{A_k} d\Theta  \quad \text{a.e.} \ (t,x) \in (0,R) \times B_R.
\end{equation*} 
Thus, for any Borel subset $A$ of $E_R$ and any constant $M>0$,
\begin{align*}
& \left| \int_{E_R} P^n ( \mathbf{1}_{A_k}  -  \mathbf{1}_{A }   ) d\Theta \right| \\
\leq & M^2 \int_{E_R} |\mathbf{1}_{A_k}  -  \mathbf{1}_{A }| d\Theta +  \int_{E_R}    f^n  f^n_{ * } (\mathbf{1}_{   f ^n\geq M}+ \mathbf{1}_{   f_{  *}^n \geq M}) B^R_{  n} d \Theta \\
\leq &  M^2 \int_{E_R} |\mathbf{1}_{A_k}  -  \mathbf{1}_{A }| d\Theta + \Psi(M)^{-1}\bar{K} \bar{N}  \to 0 \quad ( k \to \infty , \  M \to \infty ).
\end{align*}
Therefore,
\begin{align*}
&\left| 	\int_{E_R}   P^n \mathbf{1}_{A } d\Theta - 	\int_{E_R}   P  \mathbf{1}_{A } d\Theta \right| \\
\leq & \left| \int_{E_R} P^n ( \mathbf{1}_{A_k}  -  \mathbf{1}_{A }   ) d\Theta \right| + \left| 	\int_{E_R} ( P^n -P ) \mathbf{1}_{A_k}  d\Theta \right| +  \left| \int_{E_R} P  ( \mathbf{1}_{A_k}  -  \mathbf{1}_{A }   ) d\Theta \right| \\
\to & 0 \quad ( k \to \infty , \  n \to \infty ).
\end{align*} 
Since the above holds for any measurable subset $A$, it follows that almost every $(t,x) \in ( 0,R) \times B_R$ and for all $\varphi \in L^\infty ( E_R)$, 
 \begin{align*}
   N  (f^n)^{-1}  f^n  f^n_{ * } B^R_{  n}  \rightarrow N  (f )^{-1} {f  }   {f_* }  B^R, \quad   N  (f^n)^{-1} {f ^n}'  {f_*^n}' B^R_n \rightarrow   N  (f  )^{-1} f'   f'_*   B^R
 \end{align*} 
weakly in $L^1 ( E_R ,d \Theta)$.  Moreover, using the fact that  $N(f^n) ( t,x)$ and $\left( 1+ \frac{1}{n} \int_{\R^3_v } f^n d v \right)^{-1}$ converge a.e. to $N(f)$ and $1$ respectively (see \eqref{fn-f}), it follows  that \eqref{cnv-fn'*} holds.
\end{proof}

\begin{lemma}\label{lemma-e}
The convergence \eqref{ena} holds, i.e.,
\begin{align*}
	\liminf_{n\rightarrow \infty} \int_{\R^3 _v } \tilde{e} ^n dv  \geq \int_{ \R^3_v } e  dv .
\end{align*}
\end{lemma} 
\begin{proof} 
We first define the function
\begin{equation*}
	j(a,b)=\left\{\begin{array}{rl}
		(a-b) \log \frac{a}{b}, & \quad \text{for } a, b>0 \,, \\
		+\infty, & \quad \text{for } a \text { or } b \leq 0 \,,
	\end{array}\right.
\end{equation*}
whose  Hessian matrix 
\begin{equation*}
	\begin{aligned}
		\tfrac{1}{ab} 
		\begin{pmatrix}
			b+\frac{b^2}{a}  &  -(a+b)\\
			-(a+b)   &  a+\frac{a^2}{b} 
		\end{pmatrix}, 
		\quad a,b>0
	\end{aligned}
\end{equation*}
is non-negative. This implies that $j(a,b)$ is convex function. Next, we further consider the integral 
\begin{equation}
	J(F,G)=\int_{E_R}j(F,G)d \Theta , \quad F,\, G \in L^1(E_R,d \Theta ).
\end{equation}
The convexity of $j (a, b)$ implies that
\begin{equation*}
	\begin{aligned}
		j(F^n,G^n)\geq j(a,b)+ \bigg(\log \frac{a}{b}+\frac{a-b}{b} \bigg)(F^n-a)+ \bigg(\log \frac{b}{a}+\frac{b-a}{b}\bigg)(G^n-b) 
	\end{aligned}
\end{equation*}
for $F^n,G^n,a,b>0$. We assume that $F^n$ and $G^n$  converge weakly to $F \geq 0$ and $G\geq 0$ respectively in $L^1(E_R,d \Theta )$. For any $\epsilon>0$ and $M > 0$, we set $a=\epsilon +F\wedge M$ and $ b=\epsilon +G\wedge M$. It follows from the convexity of $j (a, b)$ that
\begin{equation*}
	\begin{aligned}
		\liminf \limits_{n\rightarrow \infty} J(F^n,G^n) 
		&\geq \int_{E_R} j(a,b)+(\log \tfrac{a}{b}+\tfrac{a-b}{b})(F-a)+(\log \tfrac{b}{a}+\tfrac{b-a}{b})(G-b) d\Theta \\
		&\geq \int_{E_R} j(a,b)+(\log \tfrac{a}{b}+\tfrac{a-b}{b})(F\wedge M-a)+(\log \tfrac{b}{a}+\tfrac{b-a}{b})(G\wedge M-b) d\Theta \\
		&= \int_{E_R} j(a,b) d \Theta  -\epsilon \int_{E_R} (a-b)(\tfrac{1}{a}-\tfrac{1}{b})d\Theta  \geq \int_{E_R} j(a,b) d\Theta .
	\end{aligned}
\end{equation*}
Hence, for any $\epsilon>0$ and $M > 0$,
\begin{equation*}
	\liminf \limits_{n\rightarrow \infty} \int_{E_R}j(F^n,G^n)d\Theta  \geq \int_{E_R}j(\epsilon +F\wedge M ,\epsilon +G\wedge M)d\Theta .
\end{equation*}
By letting $\epsilon \to 0_+$ and $M \to +\infty$, we conclude from Fatou's lemma (Lemma \ref{fatou}) that
\begin{equation}
	\liminf \limits_{n\rightarrow \infty} \int_{E_R}j(F^n,G^n)d \Theta  \geq \int_{E_R}j(F ,G) d \Theta .
\end{equation}
We now denote by $e^R$ the functions obtained from $e$ by replacing the collision kernel $B$ with $B_R$, and similarly for   $\tilde{e}^{n,R}$. By taking 
$$
F^n=  \bigg(1+ \frac{1}{n} \int_{\R^3_v } f^n d v \bigg)^{-1} {f ^n}'  {f_*^n}' B^R_n ,\quad  G^n= \bigg(1+ \frac{1}{n} \int_{\R^3_v } f^n d v \bigg)^{-1}  f^n  f^n_{ * }     B^R_{  n}  ,
$$ 
one has 
\begin{equation}\label{eR}
 	\liminf_{n\rightarrow \infty} \int_{\R^3_v } \tilde{e} ^{n,R} dv  \geq \int_{ \R^3 _v } e^R  dv . \quad \text{a.e.} \enspace (t, x) \in (0,R) \times B_R.
\end{equation}
Given that $  \tilde{e}^{n, R}  \leq \tilde{e}^n $, we deduce that
\begin{align*}
 \liminf_{  n \rightarrow \infty } \int_{   \R ^3 _v} \tilde{e}^n d v \geq 	\liminf_{n\rightarrow \infty} \int_{\R^3 _v} \tilde{e} ^{n,R} dv  \geq \int_{   \R ^3 _v} e^R d v.
\end{align*}
By taking the limit as $R \to + \infty $ and applying the Lebesgue Dominated Convergence Theorem (Lemma \ref{LCT}), we conclude that the convergence  \eqref{ena} holds.
\end{proof}
The entropy inequality \eqref{entropy-th}  follows immediately from the above lemma. Finally, the proof of Theorem \ref{MainThm} is therefore finished.

\section{On soft potential model $-3 < \gamma < 0$}\label{sec-cnv-soft}

 For the non-cutoff Boltzmann equation, the angular singularity of the collision kernel prevents us from establishing an Arkeryd-type inequality analogous to the one in \cite{Arkeryd-inequality}. Consequently, our analysis relies primarily on a priori estimates derived from entropy dissipation. For hard potentials, the bound established by \eqref{F-Zn} is sufficient to ensure the strong convergence of $f^n$ in $L^1([0,T] \times \mathbb{R}^3_x \times \mathbb{R}^3_v)$. However, for soft potentials, the upper bound in \eqref{F-Zn} takes the following form 
 $$
   \frac{  C_0   (2\pi )^3   }{  \|f ^n \|_{L^1( B_R (v )) } \left(2 \sqrt{2 }R \right)^\gamma  }  \bigg\{ \iint_{\mathbb{R}^{3}_v  \times \mathbb{R}^{3}_{v_*} }   f^n f^n_* |v-v_*|^\gamma   \, dv \, dv_*  + \|e^n \|_{L^1( \R^3_v )} +   C_1  \|f ^n\|^2_{L^1_2 ( \R^3_v ) }  \bigg\}  .
 $$

 Even if we assume that $f$ satisfies the fundamental conservation laws and entropy bounds, specifically $f \in L^\infty (\mathbb{R}_+; L^1_2(\mathbb{R}^3_x \times \mathbb{R}^3_v) \cap L \log L)$, this function space, despite its control over global moments and entropy, still permits a certain degree of local concentration. When such concentration combines with the singularity of the soft potential kernel $|v - v_*|^\gamma$ (where $-3 < \gamma < 0$), the collision integral term may diverge to infinity. In other words, for soft potentials, finite entropy is insufficient to guarantee the finiteness of the collision integral $\iint_{\mathbb{R}^{3}_v  \times \mathbb{R}^{3}_{v_*} }   f^n f^n_* |v-v_*|^\gamma   \, dv \, dv_*$. We will demonstrate this fact by constructing a counterexample.
 \begin{example}\label{soft-f}
 Let the positive function be defined as 
\begin{align*}
 f(v)=\begin{cases}
	|v|^{-q}, & |v|\leq 1 \\
	 e^{-|v|+1}, & |v|>1
\end{cases},
\end{align*} 
where $\frac{\gamma +6}{2}< q < 3$ with $\gamma \in (-3,0)$. Then the function $f(v)$ satisfies
\begin{align*}
\int_{\mathbb{R}^{3}_v}f(v)(1+|v|^{2}+|\log f|  )dv < \infty,
\end{align*}
 but
 \begin{align*}  
 	\iint_{\mathbb{R}^{3}_v \times \mathbb{R}^{3}_{v_*} }  f(v)f(v_{*})|v-v_{*}|^{\gamma}  dv dv_{*} = +\infty.
 \end{align*} 
  \end{example}
  
\begin{proof}

The proof will be divided into three steps.

{\bf Step 1:} {\em  Finite mass and energy.} First, by utilizing the condition $q<3$, we can demonstrate that the distribution function $f$ possesses finite mass and energy. To be specific, 
\begin{align*} 
	\int_{\mathbb{R}^{3}_v}f(v)dv 
	&= \int_{|v|\le1}|v|^{-q}dv + \int_{|v|>1} 
	e^{-|v|+1}dv \\
	&= 4\pi \int_{0}^{1}r^{2-q}dr + 4\pi  e \int_{1}^{\infty}r^{2}  e^{-r}dr \\
	&= \frac{4\pi}{3-q} + 20\pi < \infty  ,
\end{align*}
and
\begin{align*} 
	\int_{\mathbb{R}^{3}_v }|v|^{2}f(v) dv 
	&= \int_{|v|\leq 1}|v|^{2-q}dv + \int_{|v|>1}|v|^{2}  e^{-|v|+1}dv \\
	&= \frac{4\pi}{5-q} + 260\pi  < \infty .
\end{align*}
{\bf Step 2:} {\em  Finite entropy.} Subsequently, we show that the entropy is finite. 
\begin{align*} 
	\int_{\mathbb{R}^{3} _v }f|\log f|dv &= \int_{ |v|\leq 1} |v|^{-q} \left|\log |v|^{-q} \right| d v + \int_{ | v | > 1 } e^{ - | v | + 1} \left| - | v | +  1 \right| d v  \\
	&= q \int_{ | v |\leq 1 }	| v | ^{ - q } \log|v|^{-1} d v  + \int_{|v|>1}(|v|-1)e^{-|v|+1}dv \\
	&= \frac{4\pi q}{(3-q)^{2}} + 44\pi < \infty .
\end{align*}
{\bf Step 3:} {\em  Divergence of collision integral.} Finally, we verify that the following integral is divergent. 
 \begin{align*}
 I = \iint_{\mathbb{R}^{3}_v \times \mathbb{R}^{3}_{v_*}} f(v)f(v_{*})|v-v_{*}|^{\gamma} dv dv_{*}.
 \end{align*}

Since the integrand is non-negative, we can establish a lower bound for $I$ by restricting the domain of integration to a subset $\Omega$ where the singularity is most pronounced. We focus on the region near the origin where $|v| \leq 1$ and $|v_{*}| \leq 1$.
  \begin{align*}
I \geq \iint_{ | v | \leq 1, |v_{*}|\leq 1 } |v|^{-q}|v_{*}|^{-q}|v-v_{*}|^{\gamma} dv dv_{*}.
  \end{align*}
 Let $\epsilon \in (0,1)$ be sufficiently small. We define the region $\Omega$ as 
\begin{align*}
 \Omega = \left\{ (v,v_{*}) \in B(0,1) \times  B(0,1) \;\middle|\;
	 |v| \in (0,\epsilon),  \
	 |v_{*}| \in \left(\frac{|v|}{2}, |v|\right),  \
	    \theta    < \frac{\pi}{6}   \right\} ,
\end{align*}
where $\theta$ is the angle between $v$ and $v_{*}$.
Thus,
$$
I \geq \iint_{\Omega} |v|^{-q} |v_{*}|^{-q} |v-v_{*}|^{\gamma} dv_{*} dv.
$$

Next, we analyze the term $|v-v_{*}|^{\gamma}$ within $\Omega$. Using spherical coordinates  $v=r\omega$, $v_{*}=r_{*}\omega_{*}$, the Law of Cosines yields
$$
|v-v_{*}|^2 = r^2 + r_{*}^2 - 2rr_{*}\cos\theta.
$$
Furthermore, in the region $\Omega$, we have $\frac{r}{2} < r_{*} < r$ and $\cos\theta > \cos\frac{\pi}{6} = \frac{\sqrt{3}}{2}$. In this region, we can derive an upper bound for the relative velocity:
\begin{align*}
|v-v_{*}|^2  = (r-r_{*})^2 + 2rr_{*}(1-\cos\theta)  
 \leq \left(r - \frac{r}{2}\right)^2 + 2r^2 \left(1 - \frac{\sqrt{3}}{2}\right).
\end{align*} 
Specifically, there exists a constant $C_0$ such that $ |v-v_{*}| \leq C_0 r.$ Since the exponent $\gamma$ is negative, we have
$$
|v-v_{*}|^\gamma \geq  C_0^{\gamma} r^{\gamma} .
 $$
 Substituting the estimate into the integral $I$,  we obtain
\begin{align*}
I \geq C_0^{\gamma} \int_{0}^{\epsilon} \int_{\mathbb{S}^2} r^{2-q+\gamma} \left( \int_{r/2}^{r} \int_{\Omega_{\omega_*}} r_{*}^{2-q} d\omega_{*} dr_{*} \right) d\omega dr.
\end{align*}
Here, the angular integration over the spherical cap defined by $\theta < \frac{\pi}{6}$ yields a strictly positive constant, which we denote by $C_{1}$.
The inner radial integral with respect to $r_{*}$ is evaluated as
 $$
 \int_{r/2}^{r} r_{*}^{2-q} dr_{*} = \left[ \frac{r_{*}^{3-q}}{3-q} \right]_{r/2}^{r} = \frac{1}{3-q}\left(1 - \frac{1}{2^{3-q}}\right) r^{3-q}.
 $$
Let $C' = C_{0}^{\gamma} C_{1} (4\pi)^2 \frac{1}{3-q} (1 - 2^{-(3-q)})$. Since $q < 3$, it follows that $C'$ is a positive finite constant.
Finally, substituting these results back into the expression, we find that the integral $I$ satisfies
 \begin{align*}
 	I  \geq C' \int_{0}^{\epsilon} r^{2-q+\gamma} \cdot r^{3-q} dr  = C' \int_{0}^{\epsilon} r^{5-2q+\gamma} dr.
 \end{align*}
Recall that the integral $\int_{0}^{\epsilon} r^{k} \, dr$ diverges if and only if $k \leq -1$. In the present case, the exponent is   $5 - 2q + \gamma$. Under the condition $q > \frac{\gamma+6}{2}$, we have $5 - 2q + \gamma < -1$. Consequently, the integral $I$ diverges.
\end{proof}


\section*{Acknowledgments}
This work was supported by National Key Research and Development Program of China under the grant 2023YFA1010300, the National Natural Science Foundation of China under contract No. 12201220, the Guangdong Basic and Applied Basic Research Foundation under contract No. 2024A1515012358, and the Fundamental Research Funds for the Central Universities under contract No. 531118011008.

\bigskip
%

\end{document}